\documentclass[11pt]{article}

\usepackage{geometry}
\usepackage{amsmath}
\usepackage{amssymb, amsfonts, amsthm}
\usepackage{tikz-cd}
\usepackage{enumerate}
\usepackage{mathrsfs}
\usepackage{graphicx}
\usepackage{xcolor}
\usepackage{hyperref}
\usepackage{subcaption}
\usepackage{rotating}
\usepackage{titlesec}

\usepackage[
backend=biber,
style=numeric,
url=false, 
doi=false
]{biblatex}
\renewbibmacro{in:}{}
\DeclareFieldFormat[article]{title}{\mkbibitalic{#1}}
\DeclareFieldFormat[article]{journaltitle}{#1\isdot}

\newcommand{\mf}{\mathfrak}

\newcommand{\mc}{\mathcal}

\newcommand{\C}{\mathbb{C}}
\newcommand{\Q}{\mathbb{Q}}

\newcommand{\im}{\operatorname{im}}

\newcommand{\R}{\mathbb{R}}
\newcommand{\Z}{\mathbb{Z}}

\newcommand{\PP}{\mathbb{P}}
\newcommand{\Hom}{\operatorname{Hom}}
\newcommand{\End}{\operatorname{End}}
\newcommand{\Ext}{\operatorname{Ext}}
\DeclareMathOperator{\ext}{ext}

\newcommand{\supp}{\operatorname{supp}}

\newcommand{\on}{\operatorname}
\newcommand{\Id}{\operatorname{Id}}
\newcommand{\ad}{\operatorname{ad}}
\newcommand{\Ad}{\operatorname{Ad}}
\newcommand{\Spec}{\operatorname{Spec}}
\newcommand{\Aut}{\operatorname{Aut}}

\newcommand{\Heis}{\mc{H}eis}
\newcommand{\bbP}{\mathbb P}
\newcommand{\csth}{{\C^\ast_\hbar}}
\newcommand{\glas}{\mf{gl}(\mf a^*)}
\newcommand{\Serre}{\mathbb{S}}

\newcommand{\sslash}{\mathbin{/\mkern-6mu/}}

\renewcommand{\tilde}{\widetilde}

\newcommand{\GL}{\operatorname{GL}}

\newcommand{\rk}{\operatorname{rk}}

\newcommand{\msc}{\mathscr}
\newcommand{\Kos}{\on{Kos}}
\newcommand{\mrm}{\mathrm}

\newcommand{\pt}{\mrm{pt}}

\DeclareMathOperator{\Rep}{Rep}

\DeclareMathOperator{\NS}{NS}
\DeclareMathOperator{\SL}{SL}
\DeclareMathOperator{\Stab}{Stab}
\DeclareMathOperator{\Coh}{Coh}

\DeclareMathOperator{\Pos}{Pos}

\DeclareMathOperator*{\Qcoh}{Qcoh}

\DeclareMathOperator{\gr}{gr}
\DeclareMathOperator{\Pic}{Pic}
\DeclareMathOperator{\Sym}{Sym}
\DeclareMathOperator{\A}{\mathbb{A}}
\DeclareMathOperator{\shbm}{shbimod}
\DeclareMathOperator{\Tors}{Tors}
\DeclareMathOperator{\tors}{tors}
\DeclareMathOperator{\Jac}{Jac}

\newcommand{\Ind}{\on{Ind}}
\newcommand{\Res}{\on{Res}}
\newcommand{\perf}{\on{perf}}

\titleformat{\subsubsection}[runin]{\bfseries}{\thesubsubsection}{0.5em}{}

\newtheorem{thm}{Theorem}[section]
\newtheorem{cor}[thm]{Corollary}
\newtheorem{prop}[thm]{Proposition}
\newtheorem{lem}[thm]{Lemma}
\newtheorem{conj}[thm]{Conjecture}

\theoremstyle{definition}
\newtheorem{defn}[thm]{Definition}

\newtheorem{rmk}[thm]{Remark}

\makeatletter
\let\c@equation\c@thm
\makeatother
\numberwithin{equation}{section}

\title{\vspace{-2cm} Toroidal analogues of the Grothendieck-Springer map }
\author{Samuel DeHority}
\begin{document}
\maketitle
\begin{abstract}
We study moduli spaces of objects in the derived category of noncommutative ruled surfaces over orbifold curves to find equivariant deformations of moduli spaces of framed sheaves on equivariant elliptic surfaces. These derived categories are a common generalization of the category of modules of the preprojective algebra and of the category of twisted D-modules over a curve. We show that the birational geometry of the Hilbert scheme of points on an equivariant elliptic surface is controlled by an elliptic root system, and conjecture in general and prove in special cases that moduli spaces of objects in these derived categories provides a deformation of birational models for the Hilbert scheme. We also study the applications of this deformation to the geometric representation theory of the central fiber, specifically to produce geometric correspondences giving rise to Lie algebra actions on the equivariant cohomology of the central fiber and give formulas for the action of the Namikawa-Markman Weyl group of monodromy reflections in terms of this Lie algebra. 
\end{abstract}

\tableofcontents

\section{Introduction} 
The study of hyperk\"ahler geometry roughly trifurcates into the study of resolutions of conical symplectic singularities, the study of algebraic integrable systems, and the study of compact hyperk\"ahler manifolds. There are methodological differences between the branches of the field and nonetheless there is a confluence among the major results in each branch. 

The present work is a continuation of \cite{dehority_toroidal_vosa} and uses methods from all three branches (but  mostly the conical and compact cases and not the intermediate branch in which it properly lies) to study a class of examples of hyperk\"ahler manifolds 
$ X_R^{[n]}$ of Hilbert schemes associated to a series of elliptic surfaces $X_R$ projective over $\A^1$ with singular fiber encoded by the root system 
\begin{equation}\label{eq:exc_series} R \in \{ A_{-1}, A_0, A_1, A_2, D_4, E_6, E_7, E_8\} .  \end{equation} 
The surfaces $X_R$ all admit an action of an algebraic torus $T = \C^\ast_{\hbar}$ which also acts on their compactification $\overline{X_R}$ which is $E\times \PP^1$ for $R = A_{-1}$ and is a rational elliptic surface otherwise. 

Given a chern character $v$ and a framing sheaf $\mc F \in \Coh(D_\infty)$ where $D_\infty = \overline{X_R} \backslash X_R$ which is equal to $\mc E|_{D_\infty}$ for some object in $D^b_{Coh}(\overline X_R)$ obtained as $\mc E = \Phi(\mc O_{\overline{X_R}})$ where $\Phi$ is a relative Fourier-Mukai transform, let $M^{tf}(v; \mc F)$ denote the moduli stack of $\mc F$-framed torsion-free sheaves on $\overline{X_R}$ whose $S$-points consist of pairs $(\mc E_S, \phi_S)$ of a torsion-free sheaf $\mc E_S$ on $S\times \overline{X_R}$ flat over $S$ whose restriction to each $s\times \overline{X_R}$ is torsion-free, equipped with a framing $\phi: \mc E_S|_{S\times D_\infty} \simeq \mc F\boxtimes \mc O_S$. When $\mc F = \mc O_{D_\infty}$ we have $M^{tf}(v; \mc F)$ is represented by $X_R^{[n(v)]}$ for some $n(v) \in \Z$. 

\subsection{Toroidal algebra action}

Let $\mf g_{R}$ denote the Lie algebra associated to the surface $X_R$ in \cite{dehority_toroidal_vosa}. When $R\neq A_{-1}$ this is a toroidal extended affine Lie algebra and when $R = A_{-1}$ it is the algebra $\mf g_{T^*E}$ generated by symbols $w_{\gamma}^{a,b}$ for $(a,b) \in \Z^2\backslash \{(0,0)\}$  and for $\gamma \in \{ \pt, \sigma_+, \sigma_- , E\}$ which is a basis for $H^*(E)$ and by central elements $\pmb{c}_s, \pmb{c}_t$ satisfying the relation 
\begin{equation}\label{eq:g_TE_brack} [w^{a,b}_{\gamma}, w^{c,d}_{\eta} ] = -(ad-bc) w^{a+c, b+d}_{\gamma \star \eta} +\delta_{a+c, 0}\delta_{b+d,0}(a\pmb c_s + b\pmb c_t).\end{equation}

Let $V_{T^*E}$ denote the vector space 

\[ V_{T^*E} = \bigoplus_{m,n} H^*_T(M^{tf}((1,0,0) + m[E]  +n [\pt]; \mc O_{D_\infty}) .\] 

In \cite{dehority_toroidal_vosa} the author constructed an action 
\begin{equation}
  \rho: \mf g_{T^*E} \to \End(V_{T^*E})
\end{equation}
using algebraic and vertex operator theoretic methods and one of the main theorems of this paper, Theorem \ref{thm:main_repn}, is a geometric interpretation of $\rho(w^{a,b}_\gamma)$.  Specifically we produce cycles 
\[ Z_{\alpha, \gamma}\subset \bigsqcup_{m,n} M^{tf}((v+ m[E] + n[\pt]; \mc O_{D_\infty}))\times M^{tf}((v+ m[E] + n[\pt] + \alpha; \mc O_{D_\infty}))\] 
for $\alpha = a[E] + b[\pt]$ and $\gamma \in H^*(E)$ which satisfy the relations \eqref{eq:g_TE_brack} under convolution. 

In the other elliptic cases let $\mf h^{ell} \simeq \mf h \oplus \Q \delta_E^\vee \oplus \Q\delta_\pt^\vee$ denote the Cartan subalgebra of an elliptic or toroidal Lie algebra where $\mf h$ is the Cartan subalgebra of the finite dimensional semisimple Lie algebra. More generally, if $\mf g_{KM}$ is a Kac-Moody algebra associated to a star-shaped Dynkin diagram with cartan subalgebra $\mf h_{KM}$ let $\widehat{\mf h_{KM}} = \mf h_{KM} \oplus \Q \delta^\vee$ denote its Cartan subalgebra. We will be concerned with holomorphic symplectic moduli spaces whose positive weight deformation spaces are naturally isomorphic to $\mf h\otimes \C$ for $\mf h = \mf h^{ell}$ or $\mf h = \widehat{\mf h_{KM}}$ .

\subsection{Moduli spaces for noncommutative surfaces} 

Let $Q$ be the quiver with one node and one self loop and let 
\begin{equation}
  \label{eq:quiver}
  M_{\theta, \zeta}(v,w) = \mu_{\C}^{-1}(\zeta)^{\theta -ss}\sslash G(v) 
\end{equation}
denote the Nakajima quiver variety with complex moment map parameter $\zeta$. 

When $\zeta = 0$ and $w = 1$ this is also the moduli space of ideal sheaves $\mc I_Y$ of length $v$ sheaves in $\A^2$ if $\theta >0$ and the moduli space of derived duals of ideal sheaves $\mathbb{D}(\mc I_Y)$ if $\theta  < 0$. When $\zeta \neq 0$ then $M_{\theta, \zeta}(v,w)$ for any $\theta$ is isomorphic to the moduli space of right ideals in the Weyl algebra $\C[x, \partial_x]$ with fixed discrete data labelled by $v$ \cite{Berest_Wilson_2002}. 

Inspired by this our goal is to give a deformation of $X_R^{[n]}$ using particular families of noncommutative surfaces. General statements that noncommutative deformations of Poisson surfaces give rise to commutative deformations of their Hilbert schemes of points are known in \cite{rains19birational_arxiv,Matviichuk_Pym_Schedler_2022}.

Our main conjecture concerned relative moduli spaces of objects in a family of dg-categories $\msc D$ over $S$ thought of as noncommutative surfaces, where each $\mc E \in \msc D$ admits a restriction $\mc E|_{D_\infty}$ to a commutative curve in the noncommutative surface.

Then the moduli spaces 
$M_{\underline{\sigma}}(v; \mc F)$
 have $\C$-points consisting of pairs $(\mc E, \phi)$ of a $\underline{\sigma}$-stable object $\mc E$ of chern character $v$ in $\msc D$ equipped with an isomorphism
  $\phi: \mc E|_{D_\infty}\xrightarrow{\sim} \mc F$.
   Let $\Stab(\msc D_R)$ denote the space of relative stability conditions on the homotopy category of $\msc D_R$. Recall that a quasiprojective variety is semiprojective \cite{Hausel_Villegas_2013} if it admits a $\C^\ast$-action with proper fixed points and all limits at one infinity of $\C^\ast$. For $\mc F = \mc O_{D_\infty}$ we conjecture the following: 

\begin{conj}\label{conj:main_conj_geo}
  \begin{itemize} \item[1)]
    For each $R$ in \eqref{eq:exc_series} there is a family $\msc D_R/\mf h^{ell}$ of dg-categories and a space $U\subset \Stab(\msc D_R)$ so that for $\underline{\sigma} \in U $ the relative moduli spaces $ M_{\underline{\sigma}}(v; \mc F)/\mf h^{ell}$ are semiprojective and quasiprojective over $\mf h^{ell}$.
    \item[2)] The fibers over zero $ M_{\underline{\sigma}_0}(v; \mc F)$ exhaust K-trivial projective relative birational models of $X_R^{[n]}/\A^n$. 
    \item[3)]There is a 1-1 correspondence between walls in $N^1(M_{\underline{\sigma}_0}(v; \mc F)/\A^n)_\R$ of singular birational models, walls in $\mf h^{ell}$ where the wall contraction is not biregular, and a subset of roots of $R^{ell}$. 
  \end{itemize}
\end{conj}
The bijection in 3) must be compatible with the identification between the elliptic root system and torsion classes in $K_{num}(X_R)$, see \cite{Burban_Schiffmann_2013} an \cite{Rota_2022} for the cases $A_{-1}, D_4, E_6, E_7, E_8$.  Part 2) of this result is heavily inspired by \cite{Bayer_Macrì_2014_MMP} and its extensions to other projective surfaces in \cite{Kopper_2021,Arcara_Bertram_Lieblich_2013} although simplified by restricting to birational geometry relative to the base of the Lagrangian fibration. 

\begin{thm}
Conjecture \ref{conj:main_conj_geo} is true for $R= A_{-1}$. 
\end{thm}
Part 1) is Theorem \ref{thm:main_quasiprojectivity} and 2) and 3) are Theorem \ref{thm:vwalls_roots} and Theorem \ref{thm:def_equiv} which also shows the result for any stable framing sheaf $\mc F$ of coprime rank $r> 0$ and degree $d$.  

We also establish partial results for the other types. The following is Theorem \ref{thm:vwalls_roots} and Theorem \ref{thm:other_types_birational} below. 

\begin{thm}
Moduli spaces of framed Bridgeland stable objects on $\overline{X_R}$ exhaust K-trivial projective relative birational models for $X_R^{[n]}/\A^n$. The walls are in bijection with a set of roots of $R^{ell}$. 
\end{thm}

The dg-categories in Conjecture \ref{conj:main_conj_geo} are certain noncommutative analogues of smooth projective surfaces. For us, following \cite{ben2008perverse} it will be convenient to use a particular model of noncommutative surfaces introduced in \cite{vdBergh_2012} called noncommutative $\PP^1$ bundles. Given a smooth projective curve curve $C$ and a finite group $\Gamma \subset \Aut(C)$ let 
\[ HH_0([C/\Gamma]) = \Ext^1_{[C/\Gamma]\times [C/\Gamma]}(\Delta_*\mc T_{[C/\Gamma]}, \mc O_{\Delta})\] 
denote its $0$th Hochschild homology parametrizing bimodule extensions. We view this as an affine base scheme $S$. In the elliptic orbifold cases we find that we get dimensions from the following table: 
\begin{table}[h!]
  \centering
  \begin{tabular}{|l||lllll|}
  \hline
  $\Gamma$ &  $\Z/1\Z$ & $\Z/2\Z$ & $\Z/3\Z$ & $\Z/4\Z$ & $\Z/6\Z$ \\ \hline 
  $hh_0([E/\Gamma])$ & 2 & 6 & 8 & 9 & 10 \\ \hline
  \end{tabular}
  \end{table}

which is indicative of an isomorphism $HH_0([E/\Gamma]) \simeq \mf h^{ell}$. In the other cases, at least when $C/\Gamma \simeq \PP^1$ we have an identification 
\[ HH_0([C/\Gamma]) \simeq \widehat{\mf h_{KM}}\] 
where $\mf g_{KM}$ is the Kac-Moody algebra with star-shaped Dynkin diagram given by the parabolic structure associated to the curve, making the context of this paper a surface analogue of \cite{Burban_Schiffmann_2013}. 

In section \ref{sec:ncs} we introduce a family of dg-enhanced triangulated categories 
\[ D^b_{Coh}(\PP_{[C/\Gamma]_S}(\Xi))\] 
over $S$ where $\Xi$ is the universal bimodule extension and study the main properties of these dg-categories.  

In \cite{ben2008perverse,Ben-Zvi_Nevins_2007} moduli stacks of objects in fibers $D^b_{Coh}(\PP_{[C/\Gamma]}(\Xi_s))$ over closed points $s \in S$ in the case $\Gamma = \{ 1 \}$ were studied as a generalization of \eqref{eq:quiver}. We expect that in the orbifold cases \[ D^b_{Coh}(\PP_{[C/\Gamma]_S}(\Xi))\]  for $C = E$ are dg-categories which would satisfy Conjecture \ref{conj:main_conj_geo}. 

The base change of $D^b_{Coh}(\PP_{[C/\Gamma]}(\Xi))$ to closed points $s$ is simultaneously a generalization of the category of modules of the deformed preprojective algebra (see Proposition \ref{prop:preprojective_equivalence}) and the theory of twisted $D$-modules because when $\Gamma = \{ 1 \}$ these bimodules correspond to those quantizing the cotangent bundle to twisted differential operators $D(\lambda)$ for $\lambda \in \Pic(C) \otimes \C$ \cite{Beilinson_Bernstein_1993} so that as bimodules $\Xi_s \simeq D^{\le 1}(\lambda)$ the ring of twisted differential operators on $C$ of order $\le 1$. 

 We may thus think informally of the level 1 hyperplane $HH_{0}([C/\Gamma])_1 \subset HH_0([C/\Gamma])$ as consisting of bimodule extensions corresponding to $\Xi_s = D^{\le 1}(\lambda)$ for $\lambda \in \Pic(X_R)\otimes \C$. 

\subsection{Monodromy}

The other result in this paper is a description of the monodromy operators induced by the deformations in Conjecture \ref{conj:main_conj_geo}. In the conical \cite{Namikawa_2008,Namikawa_2010} and compact hyperk\"ahler cases \cite{Markman_2008,Markman_2011,Markman_2013} there is a reflection group $W$ and an representation 
\[ W \to \End(H^*(X)) \] 
induced by deformations $\mc X, \mc X'$ of $X$ over a curve $C$ and an isomorphism $\mc X_{C\backslash \{0\}} \simeq \mc X'_{C\backslash \{0\}}$.

In our case we have an action 
\begin{equation}\label{eq:iw_action_intro} IW^{ell}_{R} \to \GL(HH_0([E/\Gamma]))\end{equation}
were $IW^{ell}_R$ is a group containing the elliptic Weyl group, and in the $A_{-1}$ (Theorem \ref{thm:derived_reflection}) and conjecturally in the other cases (Conjecture \ref{conj:reflection}) these give rise to families of derived equivalences of $\PP_{[E/\Gamma]}(\Xi)$ covering \eqref{eq:iw_action_intro}. A subgroup of this group is the reflection group acting by monodromy on $H^*_T(M(v; \mc F)$. In particular using the algebra \eqref{eq:g_TE_brack} we give a formula for the generators of the monodromy action in Theorem \ref{thm:monodromy} which in the Nakajima basis for $H^*_T(T^*E^{[n]})$ takes the form
\begin{thm}\label{thm:intro_monodromy}
  The action of $\SL(2, \Z) \times \Z \times \Z/2\Z$ by monodromy on  $H^*_T(T^*E^{[n]})$ acts by 
\begin{equation}\label{eq:monodromy_intro}
  \prod_i \alpha_{-k_i}(\gamma_i) |\rangle   \mapsto   \prod_i(\pm 1)^{k_i + 1 }  w^{N, -Nk_i}_{g\cdot \gamma_i} |\rangle .
 \end{equation}
\end{thm}
In \eqref{eq:monodromy_intro}, $g \in \SL(2, \Z)$ acts by the mapping class group of $E$ on the cohomology of $T^*E$. 

Because this action contains a subgroup of an elliptic (or toroidal) Weyl group induced by the monodromy of equivariant deformations of holomorphic symplectic varieties, it serves as a toroidal analogue of Slodowy's approach to Springer theory. The real-root analogue of Theorem \ref{thm:intro_monodromy} which would follow from a number of the conjectures outlined in this paper give an monodromy interpretation of the Global Springer Theory \cite{Yun_2011}. 
\subsubsection{Acknowledgements}

A great number of insights from others were crucial in the development of this work, but in particular this project owes much to conversations with E. Rains, F. Sala, O. Schiffmann, and A. Okounkov. 

\section{Isotrivial Lagrangian fibrations}
Let $M$ be a quasiprojective holomorphic symplectic variety. 
A Lagrangian fibration is a proper map $\ell: M \to B$ whose generic fiber is a smooth Lagrangian subvariety. 
An isotrivial Lagrangian fibration $\ell:M \to B$ is one where two general fibers $\pi^{-1}(b_1)$ and $\pi^{-1}(b_2)$ are isomorphic. 

Let $\csth$ space $\A^d$ with positive weights. 

\begin{defn}
  An \emph{conical isotrivial Lagrangian fibration} of $X$ is an equivariant projective map $\pi: X \to \A^d$ which is an isotrivial Lagrangian fibration such that $\csth$ scales $\omega$ with weight $\ell > 0$. 
  \end{defn}

We will be especially interested in control over singular fibers of these varieties provided by singularities compatible with the fibration. 

  \begin{defn}
    An \emph{intermediate symplectic singularity} of $\pi: Y\to \A^d$ is symplectic variety $X^{2d}$ with a factorization $Y \xrightarrow{\pi_s} X \xrightarrow{\pi_X} \A^d$ of $\pi = \pi_X \circ \pi_s$ where $\pi_X$ is an equivariant projective birational map and $\pi_s$ is a conical Lagrangian fibration. 
    \end{defn}

\subsection{Hitchin fibrations}
A robust source of Lagrangian fibrations which require distinct fibers to be isomorphic is the Hitchin fibration 
\[ \pi_H: M_{Hit,C,G,r,d} \to B\] 
from a moduli space of stable Higgs bundles $ M_{Hit,C,G,r,d}$ of rank $r$ and degree $d$ on a projective curve $C$ to the Hitchin base $B$. There is a $\csth$ action scaling the Higgs field and acting on $B$ which forces $\pi_H^{-1}(\hbar x)\simeq \pi_H^{-1}(x)$ for $x\in B$.   
\subsubsection{Abelian variety}
We are interested in deforming the complex structure on holomorphic symplectic varieties admitting Lagrangian fibrations. The simplest case of a trivial equivariant Lagrangian fibration, which serves as a product factor in many more general cases. When the abelian variety is the Jacobian of a curve we have an identification of $T^*A$ with  the Higgs moduli space for $G = \GL(1)$ for any coprime rank and degree 
\[M_{Hit,C,G,r,d} \simeq T^*\Jac(C) \] 
and the Hitchin map is the projection onto the second factor of $T^*\Jac(C) \simeq \Jac(C) \times \A^{2g}$. In this case the deformation is essentially a case of the non-abelian Hodge correspondence \cite{Simpson_1994,simpson1996hodge} in the abelian case. The construction is a deformation of the one in \cite{rothstein1996sheaves} and the related construction in \cite{laumon96}. 

Let $\pi: T^*A \to \A^g$ be a trivial Lagrangian fibration where $A$ is a $g$-dimensional projective abelian variety. This map is equivariant with respect to $\csth$ which scales the cotangent direction with weight $\hbar$. 
We have canonical identifications $\mf a := \on{Lie}(A) = H^1(A, \mc O_A)$ and $T^*A = A \times \mf a^\ast$. 
Then we have $\Omega^1(\mf a^\ast) = \mf a$, $\Omega^1(T^*A) = \Omega^1(A) \oplus \Omega^1(\mf a^\ast) = \mf a^\ast \oplus \mf a $ and $\Omega^k(T^*A) = \oplus_{p + q = k} \wedge^p \mf a^\ast \otimes \wedge^q \mf a$. 

Consider the $g^2$-dimensional vector space, seen as an affine space, $\glas = \Ext^1(\mc O_A, \Omega^1_A)$ which is scaled uniformly with weight $\hbar$. Let $p: A\times\glas\to A, \mu: A\times \glas\to \glas $ be the projections. Let $\Xi \in \Coh_{\csth}(A \times \glas)$ be the universal extension, fitting into an exact sequence 
\[0 \to p^* \Omega^1_{A} \to \Xi \to  \mc O_{A\times \glas} \to 0 \] 
which induces an equivariant structure on $\Xi$. Then the $1$-dimensional quotient induces a canonical divisor $D_\infty \subset \PP(\Xi)$. Let $\mc X:= \PP(\Xi)\backslash D_\infty$ be the complement of the divisor which is the total space of an affine bundle over $A \times \glas$ whose fiber $\mu^{-1}(0) = X_0$ over 
$0 \in \glas$ is $T^*A$. The fiber over $\Id \in \glas$ is the main object considered in \cite{rothstein1996sheaves}. Let $\widehat{\glas}$ be the completion of $\glas$ at zero, and $\widehat{\mc X} = \mc X\times_{\glas}\widehat{\glas}$ the completion of the family. 

For a very general point $b$ of $\mf {gl}(\mf a^*)$ the fiber $\msc X_b := \mu^{-1}(b)$ is biholomorphic to $\C^{\ast \times 2g}$ but there are a potentially countable infinite number of hyperplanes $S_i \subset \mf{gl}(\mf a^*)$ where the complex structure of $\msc X_b$ differs when $b \in S_i$. In particular, $\msc X_b$ may contain Abelian subvarieties. The quantum geometric Langlands conjectures \cite{Gaitsgory_2016} in the abelian case are related to claims about how the Fourier-Mukai duality on the central fiber extends to dualities like those studied in \cite{laumon96,rothstein1996sheaves} over the family $\msc X$ or subfamilies thereof. 

\subsection{Surfaces}

The surfaces $X_R$ for $R$ from \eqref{eq:exc_series}
 have descriptions as moduli spaces of Higgs bundles, so that their 
 deformations are governed by the nonabelian Hodge correspondence. They form a subset of the H3 surfaces described in \cite{Boalch_2018}. 

In particular for any coprime $r,d$ we have 
\[ X_{A_{-1}} \simeq T^*\Jac(E) \simeq M_{Hit, E, \GL(1), r,d }\] while in the other cases the surface arises as a moduli space of meromorphic Higgs bundles for $G = \GL(r, \C)$ on $\PP^1$ with poles at specified points and specified behavior of the poles. In the $D_4$ and type $E$ cases, $X_{R}$ is a moduli space of parabolic higgs bundles with regular singularities at the branch points of the quotient map $E\to E/\Gamma$. The $A_0, A_1, A_2$ have poles of higher order. Specifically, $X_{A_2}$ is a moduli space of meromorphic Higgs bundles on $\PP^1$ with poles of order $3$ at $0$ and $1$ at $\infty$, while $X_{A_1}$ is a moduli space of Higgs bundles with a single pole of order $4$ at $0$, while $X_{A_0}$ has a single pole of order $4$ with nilpotent leading term. 

\subsubsection{Higher rank}

In the $D_4$ and type $E$ cases let $M_{R, n}$ denote a moduli space of parabolic Higgs bundles $(\mc E, \Phi)$ on a marked curve $(\PP^1, x_1,\ldots, x_k)$ with regular singularities and letting $A_i$ be the residue of $\Phi$ at $x_i$. In the $A_{-1}$ case we do the same with the curve $(E,0)$.  The residues $A_i$ are required to be compatible with flags at the marked points of dimensions given by the partitions in Table \ref{tbl:pole_orders}. 

\begin{table}[h!]
  \centering
  \begin{tabular}{|l|l|l|}
  \hline
  $R$ & $\rk$ & Pole structure  \\ \hline
  $A_{-1}$ & $n$ & $[n,1]$ \\
  $D_4$ & $2n$ & $[2n,n], [2n,n],[2n,n], [2n,n,1]$  \\
  $E_6$ & $3n$ &   $[3n,2n,n],[3n,2n,n], [3n,2n,n,1]$  \\
   $E_7$& $4n$ &   $[4n,2n], [4n,3n,2n,n],[4n,3n,2n,n,1]$  \\
   $E_8$ & $6n$ &  $[6n,3n],[6n,4n,2n],[6n,5n,4n,3n,2n,n,1]$  \\ \hline
  \end{tabular}
  \caption{\label{tbl:pole_orders} Table of pole structure of parabolic Higgs bundles in higher rank tame cases.}
  \end{table}

\begin{thm}\cite{Groechenig_2014}
For $R \in \{A_{-1}, D_4, E_6, E_7, E_8\}$ there is an isomorphism 
\[ M_{R,n} \simeq X_R^{[n]}\]  identifying the Hitchin map with the support map $X_R^{[n]} \to \A^n$. 
\end{thm}

Therefore the main theorems of the present text also provide birational geometry results for certain moduli spaces of Higgs bundles.

\section{Non-commutative surfaces} \label{sec:ncs}

We introduce families of noncommutative $\PP^1$ bundles over elliptic curves and elliptic orbifolds. Noncommutative $\PP^1$-bundles over arbitrary smooth varieties are defined in \cite{vdBergh_2012}, which initiated the study of noncommutative deformations of Hirzebruch surfaces. The moduli of noncommutative Hirzebruch surfaces is studied in \cite{MOU19arxiv}. A general reference on families of noncommutative surfaces is \cite{rains19birational_arxiv}. 

Let $S = \Spec(R)$ be, for now, a Noetherian affine scheme. Let $\Gamma$ be a finite group. Let $E$ be an elliptic curve with an action of $\Gamma$ over $S$ formed from base change from $\C$. Let $\rho: \Gamma\to \C^\ast$ be the generating character of $\Gamma^\vee$ such that $\mc T_E = \mc O_E \otimes \rho$. We will denote by $\msc{C}$ the Deligne-Mumford stack $[E/\Gamma]$.

\subsubsection{Sheaves and dg categories of orbifolds}
See \cite{Elagin_equivariant_arxiv,KrugSosna_2015,perry_equiv_hoch} for generalities on equivariant derived categories and dg-categories. 

More generally let $X$ be a $\Gamma$-scheme. We assume that $X$ is smooth and the action of $\Gamma$ is faithful. 
Let $\Coh^\Gamma(X)$ denote the $R$-linear abelian category of $\Gamma$-equivariant coherent sheaves which we identify with $\Coh(\msc X)$ where $\msc X = [X/\Gamma]$. Let $\on{Ind}: \Coh(X) \to \Coh^\Gamma(X)$ given by $\on{Ind}_1^\Gamma(\mc E) = \oplus_{g\in \Gamma} g^* \mc E$ denote inflation.  More generally there is an inflation function $\Ind_{\Gamma}^{\Gamma'}(-)$ for $\Gamma \subset \Gamma'$ and its right and left adjoint $\Res^{\Gamma'}_{\Gamma}(-)$ ignoring the linearization.  Recall that $\Hom_{\Coh(\msc X)}(A,B) = \Hom_{\Coh(X)}(A,B)^\Gamma$.   Let $\Gamma^2 = \Gamma_1 \times \Gamma_2$ act on the product $X\times X$. We define $\Gamma_{\Delta}\subset \Gamma^2$ to be the diagonal. This acts via conjugation when we view $X\times X$ as a space supporting operators $D_{\perf}(X_1)\to D_{\perf}(X_2)$. Given $\chi\in \Gamma^\vee$ let $\chi_\Delta \in \Gamma^\vee_\Delta$ denote the same character under the obvious isomorphism. 

If $\mc F \in \Coh^\Gamma(X)$ is a locally free sheaf then this produces a sheaf $\mc F_\Delta \in \Coh^{\Gamma_\Delta}(X\times X)$ supported on the diagonal. We let $\Delta^g$ denote the graph of the map $g: X\to X$ for $g\in \Gamma$. Then $\Delta \cap \Delta^g \simeq X_\Delta^g$. The group $\Gamma_{\Delta}$ preserves $\Delta^g$ so given $\mc F\in \Coh^\Gamma(X)$ we may define $\mc F_{\Delta^g}$ analogously. 

Let $(-)^\Gamma: \Coh^\Gamma(X) \to \Coh(X)$ denote invariants if the $\Gamma$ action on $X$ is trivial. This is involved in the calculation of the pushforwards, for example $\pi_{*}: \Coh^{\Gamma^2}(X_1\times X_2)\to \Coh^\Gamma(X_2)$ (when $X_1$ is proper) is the composition $\Coh^{\Gamma^2}(X_1\times X_2) \to \Coh^{\Gamma^2}(X_2) \xrightarrow{(-)^{\Gamma_1}}\Coh^{\Gamma_2}(X_2)$. 

\subsubsection{Tubular canonical algebras}

There is a description of coherent sheaves on $[E/\Gamma]$ using the derived equivalence with the category of representations of a certain quiver with relations \cite{Happel_Ringel_1986}. More generally there is an equivalence 
\[ D^b_{Coh}([C/\Gamma]) \simeq D^b(\on{mod}(\Lambda(\pmb p, \pmb \lambda))\] 
where $C/\Gamma \simeq \PP^1$ with branch points $\pmb p$ and positions encoded by $\pmb \lambda$ where $\Lambda(\pmb p, \pmb \lambda)$ is the canonical algebra. When $C = E$ these algebras are called tubular canonical algebras and posses additional derived autoequivalences related to the Fourier-Mukai transform.

\subsection{Sheaf bimodules over orbifolds} 
 We will write factors as $X_i, i = 1,2$ occasionally to avoid confusion, likewise for $\msc X_i$. A $\Gamma^2$-equivariant sheaf bimodule is a coherent sheaf $\mc E \in \Coh^{\Gamma^2}(X\times X) = \Coh(\msc X \times \msc X)$ whose scheme-theoretic support $W$ is finite over both factors. Let $\shbm^{\Gamma^2}(X,X) =: \shbm(\msc X, \msc X)$ denote the full subcategory of $\Coh^{\Gamma^2}(E\times E)$ of sheaf bimodules. 
Denote by $(-) \otimes_{\msc X_1} \mc E$ the functor 
\[\mc A \mapsto \pi_{2*}(\pi_1^* \mc A \otimes \mc E) \] 
which is defined at the level of abelian categories. In particular $\mc O_{\Delta_{\msc X}} = \on{Ind}(\mc O_{\Delta_X})$ is called the diagonal sheaf bimodule and $(-)\otimes_{\msc X_1}\mc O_{\Delta_{\msc X}} $ acts as the identity.  Given $\mc F \in \Coh^\Gamma(X)$ locally free we have and we have $\Delta_*\mc F = \Ind_{\Gamma_\Delta}^{\Gamma^2}(\mc F_\Delta) \in \shbm(\msc X, \msc X)$. 

Assume $\mc E \in \shbm(\msc X,\msc X)$ is locally free, i.e. its pushforward to $\msc X_1$ and $\msc X_2$ is locally free.  Denoting by
\[ \mc E^D := R\mc H om(\mc E, \omega_{\msc X\times \msc X})[1] = \mc Hom^1(\mc E, \omega_{\msc X\times \msc X})\] the Cohen-Macaulay dual of $\mc E$ and by
\begin{align}\label{eq:radj_bm}
\mc E^\ast &= \pi_2^* \omega_{\msc X_2}^{-1} \otimes_{\mc O_{\msc X_2 \times \msc X_1}} \mc E^D\\
\label{eq:ladj_bm}
{}^\ast \mc E &= \mc E^D\otimes_{\mc O_{\msc X_2 \times \msc X_1}} \pi_1^* {\omega_{\msc X_1}}^{-1}
\end{align}
 the left and right dual of $\mc E$, we have the adjoint properties
\begin{align*} \Hom_{\msc X_2}(\mc A \otimes_{\msc X_1} \mc E, \mc B) &= \Hom_{\msc X_1}(\mc A, \mc B\otimes_{\msc X_2} \mc E^\ast) \\
  \Hom_{\msc X_2}(\mc A , \mc B\otimes_{\msc X_1} \mc E) &= \Hom_{\msc X_1}(\mc A\otimes_{\msc X_2} {}^\ast\mc E, \mc B)
\end{align*}
which follow as in \cite{vdBergh_2012, MOU19arxiv}
on account of duality for Deligne-Mumford stacks \cite{nironi_grothendeick} (see also \cite[Appendix B]{bruzzo_sala_framed_on_stakcs}).

\subsubsection{Sheaf $\Z$-algebra}

Define the iterated adjoint $\mc E^{\ast i}$ to be $\mc E$ if $i = 0$, equal to ${\mc E^{\ast (i-1)}}^\ast$ if $i \ge 0$ and ${}^\ast({}^{\ast (i+1)}\mc E)$ if $i \le -1$. We have the double adjoint property \cite{vdBergh_2012} 
\begin{equation}\label{eq:double_adj}
  \ \mc E^{\ast\ast} = \pi_1^*\omega_{\msc X_1/S}^{-1} \otimes \mc E \otimes \pi_{2}^* \omega_{\msc X_2/S}.
\end{equation}
The adjunction gives a morphism $i_n: \mc O_{\Delta_{\msc X}} \to \mc E^{\ast n} \otimes_{\msc X}\mc E^{\ast(n+1)}$ whose image $\mc Q_n$ is invertible. 

Now as in \cite[\S 4.2]{vdBergh_2012} from a rank 2 locally free sheaf bimodule $\mc E \in \shbm(\msc C,\msc C)$ we obtain an ($R$-central) sheaf $\Z$-algebra $\mc A^{\mc E}$.  This is essentially just an category enriched over the $R$-linear monoidal category $\shbm(\msc C_S, \msc C_S)$ whose objects are indexed by $\Z$. The $\Z$-algebra is generated by $\mc A_{ii} = \mc O_{\Delta_\msc X}$ and $\mc A_{i, i+1} = \mc E^{\ast i}$ and the relations are given by an ideal generated in degree two by the line bundles $\mc Q_n$. The sheaf bimodule $\mc A_{m,n}$ is locally free of rank $n-m+1$ \cite[\S 7]{vdBergh_2012}.
\subsubsection{Non-commutative $\PP^1$ bundles}

We also obtain a category over $R$ of $\mc A$-modules $Gr(\mc A)$ consisting of formal direct sums $\oplus_{i \in \Z} M_i$ with $M_i \in \Qcoh(\msc C)$ together with maps $M_i \otimes_{\mc O_\msc X} \mc A_{ij} \to M_j$ satisfying appropriate axioms.  We write these modules as $M = (M_n)$. 
Write $M(\ell)$ for the module $(M_{n  + \ell})_{n \in \Z}$ for the sheaf $\Z$-algebra $\mc A(\ell)$ defined analogously. 
Then $\Tors(\mc A)$ is the subcategory of all direct limits of right-bounded objects. 
There are also categories $\gr(\mc A)$ and $\tors(\mc A)$ consisting of Noetherian objects. 
Then $\Qcoh(\PP_{\msc X}(\mc E)) $ and $\Coh(\PP_{\msc X}(\mc E))$ are defined to be the quotients $Gr(\mc A)/\Tors(\mc A)$ and $\gr(\mc A)/\tors(\mc A)$ respectively.
There are functors $\tau: Gr(\mc A) \to \Tors(\mc A)$ assigning the maximal torsion submodule, $\pi: Gr(\mc A) \to \Qcoh(\PP_{\msc X}(\mc E))$ the quotient map and its right adjoint $\omega: \Qcoh(\PP_{\msc X}(\mc E))$.

A non-commutative $\PP^1$-bundle over a global orbifold curve $\msc C_S$ is strongly noetherian, which follows from the same argument as \cite{vdBergh_2012} \cite{Chan_Nyman_2013}, see \cite[\S 6]{rains19birational_arxiv} since $\Qcoh(\msc C_S)$ is noetherian. 
There exists a $R$-linear bifunctor 
\begin{equation}
  \label{eq:rhom_NC}
  RHom_{S}(-,-): D^b_{Coh}((\PP_{\msc C}(\mc E)))^{op}\times D^b_{Coh}(\PP_{\msc C}(\mc E)) \to D^b_{Coh}(S).\end{equation}

\subsubsection{Universal family of bimodule extensions}

In what follows, we are only considering the case of $\mc E \in \shbm(\mathscr C, \mathscr C)$ when $\mc E$ is a bimodule extension of $\mc T_{\msc C}$ by $\mc O_{\msc C}$ coming from an element of 
\[ \Ext^1_{\msc C\times \msc C}(\Delta_* \mc T_{\msc C}, \mc O_{\Delta}) \simeq \Ext^1_{\msc C\times \msc C}(\mc O_{\Delta}, \Delta_* \mc T_{\msc C}\otimes \omega_{\msc C\times \msc C})^\vee = HH_0(\msc X). \]

We use the following lemma concerning the structure of the sheaf $\Z$ algebra for bimodule extensions. 
\begin{lem}
  When $\mc E$ is a bimodule extension of pushforwards of rank 1 locally free sheaves $\mc L_1 ,\mc L_2$. on the diagonal, $\mc A^{\mc E}_{m,n}$ is a filtered bimodule whose associated graded is $\mc A_{m,n}^{\mc L_1 \oplus \mc L_2}.$ 
\end{lem}
\begin{proof}
Should follow from \cite[Thm. 7.1.2(2)]{vdBergh_2012}. 
\end{proof}

We can compute the dimension of the space of bimodule extensions using the orbifold HKR isomorphism \cite{arinkin2019formality,baranovsky2003orbifold}
\[ HH_\ast([X/\Gamma]) =\left( \bigoplus_{g\in \Gamma}(HH_\ast(X^g)\right)^{\Gamma}.\] The relevant term for us is $\ast = 0$.

Concretely, we choose coset representatives $1\times g$ for $g\in \Gamma$ of the quotient $\Gamma^2/\Gamma_\Delta$. Then we have (omitting notation for the forgetful restriction functor)
\begin{align*} \Ext^1_{\msc C \times \msc C}(\Delta_* \mc T_{\msc C} , \mc O_\Delta) &= \Ext^1_{\msc C\times \msc C}(\Ind_{\Gamma_\Delta}^{\Gamma^2}\mc O_{\Delta}\otimes \rho_{\Delta}, \Ind_{\Gamma_\Delta}^{\Gamma^2} \mc O_{\Delta})\\
  &= \left(\bigoplus_{g\in\Gamma} \Ext^1(\mc O_{\Delta}\otimes \rho_\Delta, \mc O_{\Delta^{g}})\right)^{\Gamma_\Delta}\\
  &= \left( \Ext^1(\mc O_{\Delta}\otimes \rho_\Delta, \mc O_{\Delta})\right)^{\Gamma_{\Delta}}\oplus\left(\bigoplus_{g\in \Gamma}\mathbb{H}^0(\rho_{\Delta}^{-1} \otimes [\mc O_{\Delta^g} \to \mc O(\Delta)|_{\Delta^g}]) \right)^{\Gamma_\Delta}\\
  &= HH_0(E) \oplus\bigoplus_{g\in \Gamma\backslash \{ e\} }\mathbb{H}^0(\mc O_{E^g})
\end{align*}
Here we have used the quasiisomorphism $\mc O_{\Delta} \simeq [\mc O(-\Delta) \to \mc O]$ and that $\mc O(\Delta)|_{\Delta\cap \Delta^g} \simeq \mc N_{E^g/E}$. The first term in the resulting sum corresponds to inflations of bimodule extensions located at the diagonal and the second term consists of extensions at orbifold points. Notice that $\C^\ast_\hbar$ naturally scales $HH_0([E/\Gamma])$ uniformly with weight $\hbar$. 

Let $S = HH_0(\msc C)$ be seen as an affine scheme corresponding to the graded ring $R = \C[S]$. Then there is a universal family $\Xi \in \shbm(\msc C_S, \msc C_S)$ of rank 2 locally free sheaf bimodules corresponding to the universal extension.
We then obtain an $R$-linear dg-category $D(\Qcoh(\PP_{\msc C_S}(\Xi)))$. For concision we will replace the symbol $\PP_{\msc C_S}(\Xi)$ by the symbol $\msc S$. 

There is a map 
\begin{equation}\label{eq:tau_stupid} \tau: HH_0(\msc C) \to HH_0(\msc C) \end{equation}
which simply exchanges the two factors of $\msc C \times \msc C$ sending one bimodule extension to the other.

\subsubsection{Semiorthogonal decomposition}

We now consider functors which are analogues of $R\rho_{i*}: \mc F \mapsto R\pi_{*}(\mc F(i))$ for a commutative $\PP^1$ bundle $\pi: \PP(\mc V) \to X$. In the noncommutative case $M(i)$ is only in general a module on a different noncommutative space formed from $\mc A^{\mc E}(i)$. We use the modification of \cite{rains19birational_arxiv} which works better in families rather than the definition of e.g. \cite[\S 4]{Mori_2007}. 

Thus we first define $\rho_i^*: \Coh(\msc C_i) \to \Coh(\PP_{\msc C}(\mc E))$ for $i \in \{0,1\}$ to be $\rho_i^*(-) = \pi((\mc A_{id}\otimes_{\msc C_i} (-))_{d\in \Z})$. The right adjoint is the functor which is the composition of $\omega$ with $(M_d)_{d\in \Z} \mapsto M_i$, the composition of the two right adjoints of the functors defining $\rho_i^*$. 

Because our sheaf bimodules are all locally free $\rho^*_i$ is exact. The functor $\rho_{i*}$ is always left exact. 

\begin{prop}[\cite{ben2008perverse,rains19birational_arxiv}]\label{prop:sod_pp_univ}
There is an $R$-linear semiorthogonal decomposition 
\begin{equation}
  \label{eq:sod_pp_univ}
  D^b_{Coh}(\msc S) = \langle L\rho_1^* D^b_{Coh}(\msc C_S),L\rho_0^* D^b_{Coh}(\msc C_S)\rangle .
\end{equation}
\end{prop}

Proposition \ref{prop:sod_pp_univ} induces an isomorphism 
\begin{equation}
  \label{eq:k_0_split} K_0(\msc S) \simeq K_0(\msc C_S) \oplus K_0(\msc C_S). \end{equation}

As calculated in \cite{Mori_2007} the Euler pairing id induced from the pairing on the $K_0(\msc C_S)$ factors of \eqref{eq:k_0_split}, semiorthogonality of \eqref{eq:sod_pp_univ} and that
\begin{align*}
   \chi_{\msc S} (L\rho_1^* a, L\rho_0^* b) &= \chi_{\msc C}(a, R\rho_{1*}L\rho_0^* b) \\
   &= \chi_{\msc C}(a, b \otimes_{\msc C}\mc A^{\mc E}_{0,1})\\
   &= \chi_{\msc C}(a,b\otimes_{\msc C} \mc E).
\end{align*}

Let $\Kos^{ab}_S(\mc E)$ denote the Grothendieck abelian category of triple 
$(a,b,\eta)$ with $a,b \in \Qcoh(\msc C_S)$ and $\eta: a\otimes_{\msc C_1}\mc E \to b$ a map. Likewise define $\on{kos}^{ab}_S(\mc E)$ using $\Coh(\msc C_S)$.  Let $\Kos_S(\mc E)$ denote the dg category of $\Kos^{ab}_S(\mc E)$ with cohomology objects in $\on{kos}^{ab}_S(\mc E)$. 
There is a natural $R$-linear quasiequivalence 
\begin{equation}\label{eq:kappa_kos_equiv}\kappa: \Kos_S(\mc E) \to D^b_{Coh}(\msc X_S)\end{equation}
between the category of Koszul data and the noncommutative surface \cite{ben2008perverse} \cite{ rains19birational_arxiv}. Elements of one of the categories $\Kos^{ab}_S(\mc E), \on{kos}^{ab,}_S(\mc E)$, or $\Kos_S(\mc E)$ will be called Koszul data. Under the decomposition \eqref{eq:k_0_split} the class of $\kappa((a,b,\eta))$ is equivalent to $([b] - [a\otimes \mc E])+ [a] $. 

At the level of Koszul data the Euler pairing is 
\[ \chi_{\msc X} (\kappa(a,b,\eta), \kappa(c,d,\mu)) = \chi_{\msc C}(a,c) + 
\chi_{\msc C}(b,d)- \chi_{\msc C}(a\otimes_{\msc C} \mc E, d).\] 
When there is an exact sequence $0\to \mc O_{\Delta} \to \mc E \to \Delta_* \mc T_{\msc C} \to 0$ we have 
\begin{equation}\label{eq:euler_char}
  \chi_{\msc X} (\kappa(a,b,\eta), \kappa(c,d,\mu)) = \chi(a,c) + 
  \chi(b,d)- \chi(a, d) - \chi(a\otimes \mc T, d).
\end{equation}

\subsubsection{Restriction to divisor at infinity} 

Since $\Xi$ fits into an exact sequence 
\[ 0 \to \mc O_\Delta\boxtimes \mc O_S \to \Xi \to \Delta_* \mc T_{\msc C}\boxtimes \mc O_S \to 0\] 
we have a canonical map $a \xrightarrow{\times 1} a\otimes_{\msc C}\Xi$. For $\mc E \in D^b_{Coh}(\msc X)$ with Koszul data $(a,b,\eta)$ we define 
\begin{equation}
  \label{eq:res_to_infty} \mc E|_{D_\infty} = \on{Cone}(a \xrightarrow{\eta \circ (\times 1)} b ) \in D^b_{Coh}(D_\infty) \end{equation}
which agrees with the usual notion when the $\PP^1$ bundle is commutative. 

This notion is different from the notion of derived restriction to the canonical divisor, given by  $\on{Cone}(\eta)$, studied in \cite{rains19birational_arxiv} but since for framed torsion-free sheaves $(\mc E, \phi)$  the restriction to the anticanonical divisor is determined by $\mc E|_{D_\infty}$ many of the results in \cite{rains19birational_arxiv} apply here as well.

\subsubsection{Serre duality} 
Serre duality for noncommutative $\PP^1$ bundles was demonstrated in \cite{Nyman_2005}.

Let $M \in Gr(\mc A^{\mc E})$ be any module. Using \eqref{eq:double_adj} there is a natural $\mc A^{\mc E}$-module structure on 
\[ M\otimes_{\msc S} \omega_{\msc S/S} := M(-2)\otimes_{\msc C} \omega_{\msc C/S} = (M_{i-2}\otimes_{\mc O_{\msc C}} \omega_{\msc C/S} )_{i\in \Z}\] 
owing to an equivalence $\mc A^{\mc E}(-2) \simeq \mc A^{\mc E}$ induced by the above. 
 Let $\Serre(-) = (-) \otimes_{\msc S}\omega_{\msc S/S}[2]$. 

\begin{prop}
There is a functorial equivalence 
\[ R\Hom_{S}(\mc F, \mc G) \simeq R\Hom_{S}(\mc G, \Serre(\mc F))^\vee \] 
\end{prop}

\subsubsection{Verdier duality} \label{ssec:verdier}

A key feature of the duality in our setup is that the appropriate notion Verdier duality is a duality between different surfaces, defined by dual bimodules. General properties of a related duality are studied in \cite{Mori_Nyman_2021}. A related claim is that Verdier duality exchanges twisted $D$-modules for one twist with those for a different twist. 

Given a closed point $s$ let $\PP_{\msc C}(\Xi_s) = \msc S_s$ denote base change to $s$ and $\Xi_s$ the associated $(\mc O_{\msc C}, \mc O_{\msc C})$-bimodule. We consider an antiequivalence $\mathbb{D}: D^b_{Coh}(\PP(\msc C)(\Xi_{s}))^{op} \to D^b_{Coh}(\PP_{\msc C}(\Xi_{s'}))$. Is will be defined at the level of Koszul data. Accordingly we have 
\begin{equation}
  \label{eq:verdier} \mathbb{D}(a,b, \eta) = (\mathbb{D}_{\msc C}(b), \mathbb{D}_{\msc C}\mathbb(a), \eta') 
\end{equation}
where $\eta': \mathbb{D}_{\msc C}(b)\otimes_{\msc C_1} \Xi_{s'} \to \mathbb{D}_{\msc C}(a)$ is the image of $\phi$ under the chain of equivalences 
\begin{align*}
\Hom_{\msc C_2}(a\otimes_{\msc C_1} \Xi_s, b) &= \Hom_{\msc C_1}(a, b \otimes_{\msc C_2} \Xi_s^\ast)\\
&= \Hom_{\msc C_1}(\mathbb{D}(b \otimes_{\msc C_2} \Xi_s^\ast), \mathbb{D}(a))\\
&= \Hom_{\msc C_1} (\mathbb{D} (b)\otimes_{\msc C_2} {\Xi_{s'}}, \mathbb{D}(a))
\end{align*}
  We calculate now $\Xi_{s'}$ by showing 
  \begin{multline*}
     \mathbb{D}(b \otimes_{\msc C_2} \mc E^\ast) = R\mc Hom(\pi_{1*}(\pi_2^* b \otimes \mc E^\ast), \mc O_{\msc C}) = \pi_{1*} R\mc Hom(\pi_2^*b, R\mc Hom(\mc E^\ast, \pi_2^* \omega_{\msc C_2})[1]) \\
     = \mathbb{D}(b)\otimes_{\msc C_2} ( R\mc Hom(\mc E^\ast, \pi_2^* \omega_{\msc C_2})[1]) = \mathbb{D}(b)\otimes_{\msc C_2} \mc E.
  \end{multline*}
Thus $\Xi_{s'}$ is the $(\mc O_{\msc C_2}, \mc O_{\msc C_1})$-bimodule which is given by $(\mc O_{\msc C_1}, \mc O_{\msc C_2})$-bimodule $\mc E$ without taking an adjoint. Only when the bimodule extension in $HH_0(\msc C)$ is a rank 2 bundle on the diagonal does this give an identical $\PP^1$ bundle. For those bimodule extensions corresponding to untwisted $\mc D$-modules we get an isomorphic noncommutative surface but the isomorphism is somewhat nontrivial as the extension is scaled by $-1$. 

We denote by  $\tau : S\to S$ the reflection on $HH_0(\msc C)$ induced by this map. 
The above calculation demonstrates the following. 
\begin{prop}\label{prop:verdier_duality_reflection}
There is an antiequivalence  $\mathbb{D}: D^b_{Coh}(\msc S_S) \to D^b_{Coh}(\msc S_S) $ which becomes an $R$-linear antiequivalence after base change along the map $\tau: S\to S$. 
\end{prop}

The map $\mathbb{D}: D^b_{Coh}(\msc S_S) \to D^b_{Coh}(\msc S_S)$ is (up to shift) the derived functor of the functor of global duality for $\Z$-algebras \cite[Prop. 5.4]{rains19birational_arxiv}. The underived functor is denoted 
\begin{equation}\label{eq:ad} M \to \on{\ad} M. \end{equation} 

\subsubsection{Mutations and Fourier-Mukai transforms} 

More generally we consider an autoequivalence equivalence $\Phi \in \Aut(D_{Doh}(\msc C))$ and its basechange to $\Aut(D_{Doh}(\msc C_S))$ denoted by the same letter. Then $\Phi$ induces an action on $HH_0(\msc C)$ \cite[\S 8]{caldararu_Mukai_I} which we denote $\tau_\Phi$. 

Viewing $\mc E \in \on{shbimod}(\msc C, \msc C)$ as a Fourier-Mukai kernel, we can form $\Ad_{\Phi} \mc E$ be precomposing with the kernel defining $\Phi^{-1}$ and post-composing with the one defining $\Phi$. When $\mc E$ is an extension of the dual of the Serre functor with the identity this action agrees with the one in \cite[\S 8]{caldararu_Mukai_I} on $HH_0(\msc C)$. 

The description of $D^b_{Coh}(\msc S_S)$ in terms of Koszul data naturally lifts derived equivalences from $\msc C$ to this surface via the formula 
\begin{equation} \label{eq:derived_on_glueing}\Phi(a,b, \eta) = (\Phi(a), \Phi(b), \eta' )
\end{equation}
where 
$\eta': \Phi(a) \otimes_{\msc C_1} \mc \Ad_{\Phi} \Xi  \to \Phi(b) $ is induced by the functorality of $\Phi$ and the isomorphism $\Phi(a) \otimes_{\msc C_1} \Ad_{\Phi}(\Xi) \simeq \Phi(a \otimes_{\msc C_1} \Xi)$. 
\subsection{Local structure}

Locally on $\msc C$ near a point with trivial stabilizer the bimodules correspond to those giving $\mc D$-modules or $\mc T\oplus \mc O$-modules.  
Near an orbifold point there is additional structure. Let $C = \A^1$ and $\msc C = [\A^1/\mu_k]$. Let $\chi$ be the character of $\mu_k$ corresponding to the generator $x$ of $\C[\A^1]$ and $\zeta$ a fixed generator of $\mu_k$. 

\subsubsection{}
See \cite{Burban_Schiffmann_2013} for generalities on the category of coherent sheaves on $[\A^1/\mu_k]$. In particular, 
the category of Torsion sheaves $\Tors_0$ of torsion sheaves supported at $0$. Then $s_i := \chi^i \otimes \Bbbk_0$, $i = 1, \ldots,  k$ are simple objects generating $\Tors_0$ as a Serre subcategory of $\Coh([\A^1/\mu_k])$.  Letting $s_{i + k} = s_i$, we have that $\ext^1(s_i, s_{i+1}) = 1$ with other $\Ext$ groups between the $s_j$ vanishing. The Rees construction (i.e. $H^0(-): \Coh(\msc C) \to \Coh([\mrm{pt}/\mu_k])$ keeping track of the action of $x$) gives an equivalence of categories 
$\Tors_0 \simeq \Rep(Q_k) $
where $Q_k$ is the cyclic quiver with $k$ nodes.  

\subsubsection{}
Pick $g \in \mu_k \backslash \{ 1 \}$ and 
let $C_{sing} = \Delta\cup \Delta^g$ denote the nodal curve, $C_1 = \Delta$ and $C_2 = \Delta^g$. 

There is an exact sequence 
\[ 0 \to \omega_{C_{sing}}^\vee \to \nu_* (\omega_{C_1}^\vee\oplus \omega_{C_2}^\vee)  \to \Bbbk_0 \to 0 \] 
such that the composition $\omega_{C_{sing}}^\vee \to \nu_* (\omega_{C_1}^\vee\oplus \omega_{C_2}^\vee)  \to  \nu_* \omega_{C_1}^\vee$ is surjective and gives an exact sequence 
\begin{equation} 
0 \to \mc O_{C_2} \to \omega_{C_{sing}}^\vee \to \nu_* \omega_{C_1}^\vee \to 0
\end{equation} 
and pushing forward to $\A^1\times \A^1$ is a summand in a bimodule extension in $HH_0(\msc C)$. 

Slightly more generally, noting that $\rho = \chi^{-1}$,  given 
\[a_g \in \Ext^1_{[\A^1\times \A^1/\Gamma_{\Delta}]}(\mc T_\Delta, \mc O_{\Delta^g})\simeq \Hom^0(\mc O, \mc O_{\Delta^g})\]
we have a bimodule extension 
\begin{equation}
  \label{eq:stacky_bimod_summand}
  0 \to  \mc O_{\Delta^g} \to \on{coker}\left( \mc O \xrightarrow{(\iota, a_g)} \mc O (\Delta) \oplus  \mc O_{\Delta^g}  \right) \to  \mc T_{\Delta} \to 0  \end{equation}
A general element of $HH_0(\msc C)$ is given by a tuple $\underline{a} = ( a_g )_{g\in \mu_k}$ with $a_1 \in HH_0(\A^1) \simeq \C$. There is an analogous sequence to \eqref{eq:stacky_bimod_summand} for $g = 1$ and hence an analogous perfect complex.  Tracing through the adjunction of the inflation functor, the bimodule  $\mc E_{\underline a}\in \on{shbimod}(\msc C, \msc C)$ corresponding to $\underline a \in HH_0(\msc C)$ is equivalent to the perfect complex 
\begin{equation}
\label{eq:equiv_bimod_perf}
\mc E_{\underline a} = \left[
  \begin{array}{c} 
    \bigoplus_{g\in \Gamma} \mc O \\
     \oplus \\ 
     \bigoplus_{g\in \Gamma}  \mc O(-\Delta^g)
    \end{array} \xrightarrow{
      \begin{pmatrix} 
        \iota & 0 \\
         M_{\underline a} &\iota  
        \end{pmatrix}}  \begin{array}{c}
            \bigoplus_{g\in \Gamma}  \mc O (\Delta^g)\\
             \oplus \\  
             \bigoplus_{g\in \Gamma} \mc O\end{array}  \right]
\end{equation}
where $M_{\underline a}\in \on{Mat}_{k\times k}(\C)$ is the matrix whose $(g', g)$ coefficient is $ a_{g'g^{-1}}$,  we have identified $a_g \in \C$ with a number and all natural inclusions are denoted $\iota$. 

\subsubsection{}
We calculate $t \otimes_{\msc C}\mc E_{\underline a}$ for some objects $t \in \Tors_0$. Given simple objects $s_i, s_{i+1}$ let $e_{i+1,i}$ denote the extension. Let $\pi_i : \C[\Gamma]\to V_{i}$ denote the projection onto the $\chi^i$ isotypic subspace $V_i$ of the regular representation spanned by $v_i \in \C[\Gamma]$ so that $v_{i + k} = v_i$. Let $\chi_1^i$ denote a character of $\Gamma_1$, likewise for $\Gamma_2$. 

Let $x,y$ be coordinates on the prequotient of $[\A^1\times \A^1/\Gamma_1\times \Gamma_2] = \msc C_x\times \msc C_y$ and let $t = \C[x]/x^{n+1}$ with its induced equivariant structure. We have 
\[ \pi_{2*}(\pi_1^*(t)\otimes \mc E_{\underline a}) = \left[
  \begin{array}{c} 
    \C[\Gamma]\otimes \mc O_{\msc C_y}\langle 1, x, \ldots, x^n\rangle  \\
     \oplus \\ 
     \bigoplus_{g\in \Gamma}  (x-gy)\mc O_{\msc C_y }\langle 1, x, \ldots, x^n\rangle
    \end{array} \to
      \begin{array}{c}
            \bigoplus_{g\in \Gamma}  \frac{1}{x-gy}\mc O_{\msc C_y}\langle 1, x, \ldots, x^n\rangle\\
             \oplus \\  
             \C[\Gamma]\otimes  \mc O_{\msc C_y} \langle 1, x, \ldots, x^n\rangle\end{array} \right]
\]
with $\Gamma_1$-invariant part  
\[t\otimes \mc E_{\underline a} = \left[
\begin{array}{c} 
\mc O_{\msc C_y}\langle v_r\otimes x^r\rangle \\
     \oplus \\ 
      \mc O_{\msc C_y }\langle f_r \rangle
    \end{array} \to
      \begin{array}{c}
            \mc O_{\msc C_y}\langle h_r \rangle\\
             \oplus \\  
             \mc O_{\msc C_y}\langle v_r\otimes x^r\rangle 
  \end{array} \right]
 \] 
with $h_r = ( \frac{\zeta^{i(1-r)}x^r}{x-\zeta^i y})_{i = 0}^{n}$ and $f_r = (\zeta^{(-r-1)i}x^r(x-\zeta^iy))_{i = 0}^{n}$ the top row maps $v_r \otimes x^r$ to $-yh_r + h_{r+1}$ where $h_{n+1} = 0$. Likewise in the bottom row $f_r$ maps to $v_{r+1}\otimes x^{r+1} - y v_r\otimes x^r$.

A basis for the quotient $q_2: \mc O_{\msc C_y}\langle v_r \otimes x^r \to Q\rangle $ of the second row is given by the image $t_a = q(y^av_0\otimes 1)$ of $y^a v_0 \otimes 1$ for $ a= 0, \ldots, n$ and for the quotient $q_1$ of the first row a basis is $y^a h_0$ for $a = 0, \ldots, n$. The $\mc O_{\msc C_y}$-module structure on the quotient $q: [-\to -] \to t\otimes \mc E_{\underline a}$ of this complex is thus determined by the fact that 
\begin{align} \nonumber q(yh_r) &= q(h_{r+1}- \iota(v_r\otimes x^r)) = q(h_{r+1}) + q(M_{\underline a} v_r \otimes x^r) \\
  &= q(h_{r+1}) + \on{Tr}_{V_r}(\sum_g a_g g\cdot-)q(v_r \otimes x^r)\\
  &= q(h_{r + 1}) +\chi^r(\sum a_g g) q(y^r v_0 \otimes 1).
\end{align}
Letting $A_r = \chi^r(\sum a_g g) = \on{Tr}_{V_r} (\sum_g a_g g\cdot -)$ in our basis $q(h_r), q(y^rv_0 \otimes 1)$ of $t\otimes \mc E_{\underline a}$ the action of $y$ is 
\[ y v = \left(
  \begin{array}{c|c}
  J_{n+1} & A\\
  \hline
  0 & J_{n+1}
\end{array} \right)v\] 
with $A = \on{diag}(A_0, \ldots, A_{n})$ and $J_k$ is a $k\times k$ Jordan block. 
An elementary linear algebra calculation gives the following and its corollaries. 
\begin{lem}
The extension $0 \to t \to t\otimes \mc E_{\underline a} \to t\otimes \mc T_{\msc C} \to 0$ splits if and only if $\on{Tr}(A) = 0$. 
\end{lem}

\begin{cor} When $t = s_i$ we have 
\begin{equation}
  \label{eq:gen_simple_bimodact} s_i \otimes_{\msc C} \mc E_{\underline a} = \begin{cases}
    s_i\oplus s_{i-1} &  \text{ if } \chi^i(\sum_{g} a_g g) = 0 \\
e_{i,i-1} &\text{ otherwise. }  
  \end{cases}
\end{equation}
\end{cor}
Let $t \in \Tors_0$ be a simple sheaf whose global sections form some number of copies of the regular representation. Let $e$ be the nontrivial extension of $\rho\otimes t$ by $t$ induced by $\Hom^0(t,t)^\vee \simeq \Ext^1(\rho\otimes t, t)$. 
\begin{cor}
  The bimodule action on $t$ is given by 
\begin{equation}
  \label{eq:simple_reg_bimodact} t \otimes_{\msc C} \mc E_{\underline a} = \begin{cases}
    t\oplus \rho\otimes  t &  \text{ if } a_1 = 0 \\
e &\text{ otherwise. }  
  \end{cases}
\end{equation}
\end{cor}

Thus under the Mukai correspondence equivalence between root spaces in $\widehat{\mf g}_{A_{k-1}} $ and finite dimensional virtual representations of $\mu_{k}$ identifying $\widehat{\mf h}$ with $HH_0([\A^1/\mu_k])$ the vanishing of the character $\chi^\alpha$ is the root hyperplane $\alpha^\perp$. 

\begin{rmk}\label{rmk:local_bimod_nontorsion}
  Because we are using the dg-enhanced category $D^b_{Coh}(\msc C)$ a minor modification of the above calculation also provides a calculation of $t\otimes_{\msc C}\mc E_{\underline{a}}$ when $t$ is not a torsion sheaf, namely  we have 
  \[H^0(t\otimes\mc E_{\underline a}) \simeq H^0(t) \oplus \rho \otimes H^0(t)\in \on{QCoh}([\mrm{pt}/\mu_k])\]
  and the $\mc O_{\msc C}$-module structure on $t\otimes_{\msc C}\mc E_{\underline a}$ has the localization $\widetilde{H^0(t)} \in \Coh(\msc C)$ of $H^0(t)$ as a submodule while for $\rho\otimes m \in \rho \otimes H^0(t)$ the action of $y$ is given by $\rho \otimes y m + \sum_g a_g g \epsilon (\rho \otimes m)$ where $\epsilon: \rho \otimes H^0(t) \to H^0(t)$ is the obvious (non-equivariant) identification. 
\end{rmk}

\subsubsection{Identification with preprojective algebra} 

We briefly mention the relationship between objects in $\Coh(\PP_{[\A^1_y/\mu_k]}(\mc E_{\underline a}))$ and finite dimensional representations of the deformed preprojective algebra $\Pi^\lambda$ for the cyclic quiver $Q_k$. Let $Q_k$ have cyclic orientation, let $V = \oplus_i Vi$ denote a $\Pi^\lambda$-representation and $\beta_{CW}$ and $\beta_{CCW}$ denote the action of the clockwise and counterclockwise arrows. Let $H \subset D_{\Coh}(\PP_{[\A^1/\mu_k]}(\mc E_{\underline a}))$ denote the subcategory whose Koszul data is of the form $(t,t,\eta)$ with trivial restriction to infinity with $H^0(t) \in \on{Rep}_\C(\mu_k)$ finite dimensional.   The identification is summarized in the following table: 

\begin{table}[h!]
  \begin{centering}
  \begin{tabular}{c|c}
  $V \in \on{Mod}-\Pi^\lambda$  & $\kappa((t,t, \eta)) \in  H$ \\\hline
  $\lambda_i$ & $\chi^i(a_g g)$ \\
  $V$ & $H^0(t)$ \\
  $\beta_{CW}$ & $y\cdot -$ \\
  $ \beta_{CCW}$ & $\beta \in \Hom(t\otimes \mc T, t)$
  \end{tabular}
  \caption{\label{tbl:preprojective}
Identification between representations of $\Pi^\lambda$ and certain modules over noncommutative $\PP^1$ bundles. Here $\beta \in  \Hom(t\otimes \mc T, t)$ is the map induced from $\eta$ by the splitting of $t\otimes \mc E_{\underline t}$. }
\end{centering}
  \end{table}

  Let $x \simeq [\mrm{pt} /\mu_k]$ denote a stacky point and $x\to [C/\Gamma]$ a closed embedding corresponding to a point $\tilde{x} \in C$. Let $z$ be a uniformizing parameter in a local ring $C'$ around $\tilde{x}$. There is an induced map $HH_0([C/\Gamma]) \to HH_0([C'/\mu_k])$. Given $\xi \in HH_0([C/\Gamma])$ let $\xi_x$ denote its restriction in $HH_0([C'/\mu_k])$. Notice that if $t\in \Coh([C'/\mu_k])$ is such that there exits $\eta$ with $\kappa((t,t, \eta)) \in D_{\Coh}(\PP^1_{[C'/\mu_k]}(\mc E_{\xi'}))$  has trivial restriction to $\infty$ then $\kappa((t,t, \eta))$ lies in the standard heart of $D_{\Coh}(\PP^1_{[C'/\mu_k]}(\mc E_{\xi'}))$.

  \[ \mc T_{x} \subset \Coh(\PP^1_{[C/\Gamma]}(\mc E_{\xi}))\] 
  denote the full subcategory of coherent objects supported over $x$ with trivial restriction to infinity. It is straightforward to check that the assignment of Table \ref{tbl:preprojective} is functorial.  We summarize the identification in the following proposition. Let $\on{Mod}_n-\Pi^{\lambda(\xi')}$ denote the full subcategory of modules for which the clockwise maps are nilpotent. Call such modules seminilpotent. 

  \begin{prop}\label{prop:preprojective_equivalence} Given $\xi \in HH_0([C/\Gamma])$ there exits $\lambda(\xi)$ such that 
there is an equivalence of categories $\mc T_{x} \simeq \on{Mod}_n-\Pi^{\lambda(\xi')}$ between torsion sheaves on $\Coh(\PP^1_{[C'/\mu_k]}(\mc E_{\xi'}))$ and finite dimensional seminilpotent representations of the preprojective algebra. 
  \end{prop}

\subsection{Equivalence with commutative surfaces} \label{ssec:equiv_w_comm}
Let $\xi \in HH_0(E)$ denote the bimodule extension perpendicular to commutative deformations of $E$, corresponding to D-modules. In the orbifold case we consider $\xi \in HH_0(\msc C)$ under the inclusion $HH_0(E) \hookrightarrow HH_0(\msc C)$. 
Let $HH_0(\msc C)_{com} = \xi^\perp \subset HH_0(\msc C)$ denote the hyperplane perpendicular to this direction. Let $S_{com} \subset S$ denote the corresponding subscheme. 

Let $\Gamma{\text -}Hilb/S_{com}$ denote the moduli functor parametrizing $\Gamma$-clusters, i.e. elements of $\Kos_S(\Xi)$ of the form $(a,a, \phi)$ where $a \in \Coh_{\Gamma}(E)$ is torsion, of length $|\Gamma|$ and $H^0(a) \simeq \C[\Gamma]$ as a $\Gamma$-representation such that $(a,a,\phi) \in \Coh(\PP_{\msc C}(\Xi))$ and is a quotient of $\mc O_{\PP_{\msc C}(\Xi)}$. Equivalently these are the equivariant analogues of \emph{point modules} from \cite{Nyman_2001} c.f. \cite[\S 4]{vdBergh_2012}. Namely a $\Gamma$-cluster module in  $\on{Gr}(\mc A^{\mc E})$ over a $G$-scheme $X$ is an $\mc A^{\mc E}$ module $(M_i)$ such that for $i \gg 0$ each $M_i \in \Coh_{\Gamma}(X/S)$ is finite over the base $S$ and pushes forward to  a rank $|\Gamma|$ locally free sheaf on $S$ which is isomorphic to $\C[\Gamma]\otimes \mc L_i$ where $\mc L_i$ is a line bundle on $S$.  

\begin{conj}\label{conj:commutative_hyperplane}
  The family of noncommutative surfaces $D^b_{Coh}(\msc S_{S_{com}})$ is equivalent to a family of commutative surfaces $D^b_{Coh}(\Gamma{\text -} Hilb/S_{com})$. 
\end{conj}

The above claim in the $\Gamma = 1$ case reduces to the fact that noncommutative $\PP^1$ bundles recover the category of sheaves on a $\PP^1$ bundle in the commutative case.

A collection of hyperplanes $ \{ H_i \}$ in a complex vector space $V$ is called \emph{strongly dense} if it is dense and if $\{ H_i \cap V' \}$ is strongly dense when restricted to proper subspace $V' \subset V$ when $\dim (V) > 2$.

\begin{prop}\label{prop:dense_hyperplanes} 
There is a countable strongly dense collection of hyperplanes $S_i \subset HH_0(E)$ such that the family $D_{coh}(\msc S_{S_i})$ is derived equivalent to a family of commutative surfaces. Assuming Conjecture \ref{conj:commutative_hyperplane} the same is true for $HH_0(\msc C)$. 
\end{prop}

\begin{proof}
The action of $\Aut(D^b_{Coh}(E))$ or more generally of $\Phi \in \Aut(D^b_{Coh})([E/\Gamma])$ acts on $HH_0(\msc C)$ via $\Phi_{HH}$ and induces an family of derived equivalences 
\[ D^b_{Coh}(\PP_{\msc C}(\Xi_S))/S \xrightarrow{\Phi} D^b_{Coh}(\PP_{\msc C}(\Xi_{\Phi_{HH}(S)}))/\Phi_{HH}(S). \] 
Thus the result follows from the fact that the usual commutative hyperplane gives a family of commutative surfaces. 
\end{proof}

More generally we will see in \ref{prop:roots_are_hyperplanes} a representation theoretic interpretation of these hyperplanes, namely that we have a bijection 
\begin{equation}\label{eq:roots_HHhyperplanes} 
R_{A_{-1}}^{ell} \xrightarrow{\sim} \{ S_i \}
\end{equation} 
between roots in an elliptic root system and hyperplanes.

\subsubsection{Blwodown along commutative curve}

The blowdown $\pi: \overline{S_R} \to S_R$ of the central surface to the rational elliptic surface $S_R$ is expected to have an analogue over $HH_0(\msc C)$. 

\begin{conj}\label{conj:SOD_blowup}
The family $\msc S_S$ admits a strong $S$-linear semiorthogonal decomposition 
\[ D_{coh}(\msc S_S) = \langle \msc X_S, \mc A_S\rangle \] 
such that the fibers of $\mc A$ over closed points admit full strong exceptional collections of the same length. 
\end{conj}

\subsection{Sheaves with stable framing }

\subsubsection{Central surface}
Let $(\mc G, \phi : \mc G|_{D_\infty} \xrightarrow{\sim} \mc F)$ be a framed torsion-free sheaf on $E\times \PP^1$ with $\mc F \in \Coh(D_\infty)$ a stable locally free sheaf.  More generally we allow $\mc G$ to be any framed coherent sheaf, or a complex of sheaves which is quasiisomorphic to a coherent sheaf in a neighborhood of $D_\infty$. 
\begin{lem} \label{lem:generically_F} For all but finitely many closed fibers $E_a, a\in \PP^1$ we have $\mc G|_{E_a} \simeq \mc F$. 
\end{lem}
\begin{proof} The sheaf $\mc G$ is locally free in a neighborhood of $D_\infty$ \cite{Nevins_thesis}. 
We get a map $U \to \mc M_E$, from  some neighborhood $U$ of $\infty$ in $\PP^1$ to the stack of coherent sheaves on $E$ which maps an open $U'$ with $\{\infty\} \subset U'\subset U$ to $\mc M^{s}_{E, \alpha(\mc F)}$, which isomorphic to $E$ and every map $U'\to E$ is constant. 
\end{proof}

Let $a \in \PP^1$ be a closed point such that $\mc G|_{E_a}\not \simeq \mc F$. Let $\mc G_{a,k} = \mc G|_{E_{ka}}$ be the restriction of $\mc G$ to the $k$th formal neighborhood of $E_{ka}$ defined by the sequence 
\[0 \to \mc G(-kE_a) \to \mc G \to \iota_{E_{ka}*}\mc G|_{E_{ka}} \to 0.\] 
Let $n(a)$ be the smallest value so that $\mc G(-mE_a) / \mc G(-(m+1)E_a) \simeq \mc F$ for all $m \ge n(a)$. 

\begin{defn}
The \emph{inhomogeneities} of $(\mc G, \phi)$ are the sheaves $\{ \mc G_{a, n(a)}\mid a \in m(\mc G) \}$ where $m(\mc G)$ is the set of base of points of \emph{inhomogeneous fibers} $E_{n(a)a}$. 
\end{defn}

Notice in particular that $\sum_{a\in m(\mc G)} n(a)$ is not invariant under flat families of sheaves. 


\begin{rmk}Given a relative Fourier-Mukai transform $\Phi_\gamma$, $\gamma \in \SL(2, \Z)$ such that $\Phi_\gamma(\Coh) \cap \Coh \neq \{ 0 \}$, there is an algorithm to determine the sheaf $\mc F_a$ that occurs at the singular fibers 
such that $\mc O_E(-a)\in  \Phi(\Coh(E)) $. 
\end{rmk}

The following is a generalization of the sequence $\mc G \to \mc G^{\vee\vee} \to \mc G^{\vee\vee}/\mc G$ accommodates also Fourier-Mukai partners of length zero sheaves. Given a torsion sheaf $\mc Q \in \Coh_{tors} (E\times \PP^1)$ supported away from $D_\infty$ let 
\[\mc Q = \mc Q_n \twoheadrightarrow \mc Q_{n-1} \twoheadrightarrow \ldots \twoheadrightarrow Q_0 = 0 \] 
denote the successive quotients of its Harder-Narasimhan filtration.

\begin{prop}\label{prop:canon_inj}For $\mc G \in M^{tf}(v; \mc F)$
There is a canonical sequence of injections
\[ \mc G  = \mc G^0 \hookrightarrow \mc G^{1} \hookrightarrow  \mc G^{2} \hookrightarrow  \ldots \hookrightarrow  \mc G^{n} \] such that 
\begin{enumerate}[a)]
\item There is an isomorphism $\mc G^n \simeq \mc F\boxtimes \mc O_\PP(\ell)$ for some $\ell$ and $\mc G^n/\mc G$ is a torsion sheaf $\mc Q$. 
\item The quotient $\mc G^n / \mc G^k$ is isomorphic to the HN quotient $\mc Q_{n-k}$ of $\mc Q$. 
\end{enumerate}
\end{prop}
\begin{proof}
  Apply the sequence $\mc G \to \mc G^{\vee\vee} \to \mc G^{\vee\vee}/\mc G$ on a sequence $S = S_0, S_1, \ldots, S_n$  of derived equivalence surfaces where $\Phi_{S_k \to S_{k+1}}(\mc G^{k+1})$ is torsion-free. 
\end{proof}

\subsection{Moduli stacks of complexes} 

Given a family $\mc E \in \shbm(\msc C_S, \msc C_S)$ of sheaf bimodules it follows exactly as in \cite{rains19birational_arxiv} that \cite{toen2007moduli} implies there exists a derived stack $\mc M_{\PP_{\msc C}(\mc E)/S}$ whose $T$-points consist of objects in $\on{perf}(\PP_{\msc C}(\mc E)_T)$. The tangent complex at a geometric point $[F]$ of $\mc M_{X/S}$ is $R\Hom(F,F)[1]$. Likewise there is an algebraic stack $\mc M^{ug}_{\PP_{\msc C}(\mc E)/S}$ locally of finite presentation whose values at $T/S$ an arbitrary (non-derived) scheme are universally glueable complexes. 

\subsubsection{Towards a GIT quotient presentation} \label{ssec:git_affine_fibrations}

When the framing sheaf is simple, the stack $M^{tf}(v; \mc F)$ is representable by a quasiprojective moduli space \cite{Nevins_thesis} in the $\Gamma = 1$ case.

Here we give a chart for the stack of objects in $D_{\Coh}\PP_{[E/\Gamma]_S}(\Xi)$ which we expect to cover the substack of stable objects.  

Let $\mc M_{>\mu(\mc F)}(\alpha) \subset \mc M_{\msc C}(\alpha)$ denote the finite type complete open substack of the moduli stack of elements in $\Coh(\msc C)$ where the slope of each HN factor is $> \mu(\mc F)$. We will write this as a as a stack quotient 
\[ \mc M_{> \mu(\mc F)}(\alpha) = [R_{> \mu(\mc F)}/\GL(\oplus_{i =1}^r V_i)]\] 
where $R_{> \mu(\mc F)}$ is a dense open subset of a Quot scheme $Quot(\oplus_{i = 1}^r \mc L_i \otimes V_i; \alpha)$ consisting of those quotients whose HN factors have the appropriate type and $\mc L_i$ are line bundle summands of a tilting object in the $\Gamma \neq 1$ case and some line bundle in the $\Gamma = 1$ case. We assume that $\mu(\mc L_i) \lneq \mu(\mc F)$ for all $i$. From now on denote $\GL(\oplus_{i = 1}^r V_i)$ by $G$, the object $\mc L_i \otimes V_i$ by $\mc V_i$  and $\oplus_{i =1 }^r \mc  V_i$ by $\mc V$. 

Consider the bundle $\mc H$ on $R_{> \mu(\mc F)}$ whose fiber is $\Hom(\mc F, \mc Q)$ over the point $[ \oplus_i \mc L_i \otimes V_i \xrightarrow{q} \mc Q]$ with its natural $G$-equivariant structure. Denote the total space of $\mc H$ by $R^{\leftarrow \mc F}$.  Denote points of $R^{\leftarrow \mc F}$ by $[\mc F \oplus \mc V \to \mc Q]$ which we will identify with a complex of sheaves with $\mc Q$ in degree 1. Now let $S_a = HH_{0a} \subset HH_0(\msc C) = S$ denote the 1 dimensional subspace consisting of deformations supported scheme theoretically on the diagonal in the $HH_0(E)$ summand of $HH_0(\msc C)$. This is naturally a direct summand on account of the splitting provided by the HKR isomorphism. Let $\Xi_a$ denote the pullback of $\Xi$ to $S_a$ be the universal bimodule extension along $HH_{0a}$. 
Let $ E$ denote the bundle over $R^{\leftarrow \mc F}_{S_a} = R^{\leftarrow \mc F}\times {S_a}$ whose fiber over $x = ([\mc F \oplus \mc V \to \mc Q], s)$ is $\Ext^1(\mc Q,[\mc F \oplus \mc V \to \mc Q]\otimes \Xi_s^* )$. Induced by the sequence of bimodules
 $0 \to \Delta_*\mc O_{\msc C_{S_a}} \hookrightarrow \Xi_{a} \to \mc \Delta_*T_{\msc C_{S_a} }$ we get an exact sequence
\[0 \to [\mc F \oplus \mc V \to \mc Q]\otimes \omega_{\msc C_1} \to [\mc F \oplus \mc V \to \mc Q]\otimes_{\msc C_2} \Xi_a^* \to  [\mc F\oplus \mc V \to \mc Q] \to 0 .\] 
On account of the vanishing $\Hom^0(\mc Q, [\mc F \oplus \mc V \to \mc Q]) = 0 $ the above induces an exact sequence of bundles whose fiber over $x$ is
\begin{multline}
  \label{eq:EZ_es} 0 \to \Ext^1(\mc Q,[\mc F \oplus \mc V \to \mc Q]\otimes \omega_{\msc C_1}  ) \to \Ext^1(\mc Q,[\mc F \oplus \mc V \to \mc Q]\otimes_{\msc C_2} \Xi_s^*  ) \\\xrightarrow{f} \Ext^1(\mc Q,[\mc F \oplus \mc V \to \mc Q] ) \to 0. \end{multline}

Letting $\iota$ denote the canonical inclusion of $\mc Q$ into the second term in the sequence in the target of $\Ext^1(\mc Q, [\mc F\oplus \mc V \to \mc Q])$ which is the restriction to this fiber of a section $\iota_{S_a}$. Using this we obtain an affine bundle $Z$ over $R^{\leftarrow \mc F}_{S_a}$ which is the affine subbbundle of $E$ defined by $f^{-1}(\iota_{S_a})$. This affine bundle $Z$ is a torsor for the bundle whose fiber over $x$ is the leftmost term in \eqref{eq:EZ_es}. Over the zero extension in $HH_{0a}$ the sequence \eqref{eq:EZ_es} splits and $Z_0$ is actually a vector bundle. We now define a function on $Z$. 

The map $q: \mc V \to \mc Q$ and the canonical map $[\mc F \oplus \mc V \to \mc Q] \to \mc V$ induce a diagram of affine spaces over $R_{S_a}^{\leftarrow \mc F}$ whose fiber over $x$ is 

\[
\begin{tikzcd}
  Z_s \arrow[rr] \arrow[d]                                                          & {} \arrow[r] & {\Ext^1(\mc V, \mc V\otimes \omega_{\msc C_1})} \arrow[d]   \\
  {\Ext^1(\mc Q, [\mc F\oplus V\to \mc Q]\otimes_{\msc C_2}  \Xi_s^*)} \arrow[d] \arrow[rr] & {} \arrow[r] & {\Ext^1(\mc V, \mc V\otimes_{\msc C_2}  \Xi_s^*)} \arrow[d] \\
  {\{\iota\} \in \Ext^1(\mc Q, [\mc F\oplus \mc V \to \mc Q])}                      &              & {0 \in \Ext^1(\mc V, \mc V)}                             
  \end{tikzcd}
\]
where the vertical maps denote fiber sequences of affine spaces. The upper right hand corner is the trivial bundle over $R_{S_a}^{\leftarrow \mc F}$ with fiber 
\[ \Ext^1(\mc V, \mc V\otimes \omega_{\msc C_1}) \simeq \mf{end}(\oplus_i V_i)^*  = \oplus_{i = 1}^r \mf{gl}(V_i)^*. \] 

We denote the map $Z\to \mf{end}(\oplus_i V_i)^*$ by $\mu_{b}$, and further denote the projection $R^{\leftrightarrow \mc F}_{S_a} \to S_a$ by $\mu_a$. Together we will denote the corresponding map by
\begin{equation}\label{eq:moment_to_HH} \mu := (\mu_a, \mu_b) : Z \to HH_0(\msc C) \subset  HH_{0a} \oplus  \mf{end}(\oplus_i V_i)^*
\end{equation}
where we have identified the complement of $HH_{0a}$ in $HH_0(\msc C)$ with the center of $\mf{end}(\oplus_i V_i)$. The map $\mu$ will play the role of a moment map for a Hamiltonian reduction. 

Let $\zeta = (\zeta_a, \zeta_b) \in HH_0(\msc C)$ denote a element of the target of $\mu$. Our expectations for this construction are expressed in the following conjecture. 

\begin{conj} The stack $[\mu^{-1}(\zeta)/ G]$ is a chart of $\mc M(v; \mc F)$ on $\PP_{[E/\Gamma](\Xi_\zeta)}$ which contains all $\mc F$-framed torsion-free sheaves. 
For generic values of $\zeta$ the action of $G$ on $\mu^{-1}(\zeta)$ is free. There is a $G$-equivariant line bundle on $\mu^{-1}(\zeta)$ for which GIT stability is equivalent to being torsion-free. 
\end{conj}

\section{Weyl groups and derived equivalence in families}
\label{sec:gewg}
Let $G$ be a semisimple Lie group and $W$ its Weyl group. 
Slodowy's approach to the Grothendieck-Springer theory, see \cite[\S 3]{Chriss_Ginzburg_2010} for an exposition, constructs $W$ representations on fibers of the Springer map $T^*G/B \to \mc N$ using the deformation of this map provided by the Grothendieck-Springer map $\mc Y \to \mf g$. In our setting of moduli space of objects in families of derived categories of $\PP^1$ bundles over orbicurves, the Weyl group $W$ will be replaced with a stabilizer subgroup of group containing the elliptic Weyl group in the elliptic orbifold case, and in the general case of a weighted projective line, a stabilizer subgroup of a group containing the affinization of a star-shaped Weyl group as in index 2 subgroup. An extension of these groups accommodating shifts in the derived category will act on family covering an action of $W$ on the base $HH_0(\msc C)$.

\subsection{Reflection groups}

In the case of elliptic orbifolds, the relevance of elliptic root systems and elliptic weyl groups to stability conditions on elliptic orbifolds was described in \cite{Rota_2022} where walls in the stability space for the subcategory of torsion sheaves on an elliptic orbifold were related to roots in an elliptic root system.

\subsubsection{Elliptic root systems}

An elliptic root system in a rank $\ell$ semidefinite real or complex inner product space $(F, \langle -, -\rangle)$ of dimension $\ell + 2$ is defined by a distinguished set of real roots $R_{re}\subset F$ and imaginary roots $R_{im} \subset \on{rad}(\langle -, -\rangle)$ whose union $R = R_{re} \cup R_{im}$ is the root system. A choice of a line $G \subset \on{rad}(\langle -, -\rangle)$ is a marking of $R$ and this induces an affine root system $R_{aff} = R/(R\cap G)$ on the semidefinite space $F/G$ and finite root system $R_{fin} = R/R \cap \on{rad}(\langle -, -\rangle)$ on $F/\on{rad}(\langle -, -\rangle)$. We always assume that $G$ is the span of an imaginary root. The pair $(R,G)$ is a marked elliptic root system. 
\begin{defn}
  A \emph{fully marked elliptic root system} is a triple $(R, \delta_1, \delta_2)$ where $R$ is an elliptic root system and $\delta_1$ and $\delta_2$ form a basis of $R_{im}$. 
\end{defn}

As noted in \cite{dehority_toroidal_vosa}, the above definition of an elliptic root system differs slightly from the one in \cite{Saito_1985} mainly in how it handles the Cartan type $A_{-1}$ corresponding to the empty root system. Then $F \simeq \Bbbk^2$ and $R\simeq \Z^2$ is a lattice in $F$ of imaginary roots. 

Given a finite dimensional root system in the doubly extended Deligne exceptional series 
\begin{equation}
\label{eq:deligne_series}
R = R_{fin} \in \{ A_{-1}, A_0, A_1, A_2, G_2, D_4, F_4, E_6, E_7, E_8 \} 
\end{equation}
let 
\[ R^{ell} = \{ \beta + m \delta_1 + n\delta_2 \mid \beta \in R_{fin} \sqcup\{ 0 \} ,~  m,n \in \Z \} \subset F_{fin} \oplus \Bbbk \delta_1 \oplus \Bbbk \delta_2 \] 
denote the corresponding elliptic root system. 

\subsubsection{Elliptic Weyl groups} 

Given an elliptic root system $R^{ell} \subset F$ the Weyl group of $R^{ell}$ is the subgroup 
\[ W_R^{ell} \subset \on{Aut}(F, \langle - , -\rangle )\]  generated by reflections $w_{\beta}$ through $\beta \in R^{ell}_{re}$. 

The choice of marking $G$ determines a splitting 
\begin{equation}\label{eq:elliptic_weyl_split} 0 \to  \Z[R_{fin}] \to W_R^{ell} \to W_{R}^{aff} \to 0 \end{equation}
 via the induced action on $R_{aff}$ where $W_R^{aff}$ is the affine Weyl group associated to $R$. Recall also that $W_{R}^{aff}$ itself fits into a similar exact sequence
\begin{equation}\label{eq:affine_weyl_split} 0 \to \Z[R_{fin}] \to W_{R}^{aff} \to W_R \to 0. \end{equation}

\subsubsection{Affinized star-shaped root systems} 

More generally, consider a root system $R_{KM} \subset \mf h$ where $\mf h \subset \mf g_{KM}$ is a Cartan subalgebra on a Kac-Moody algebra with star-shaped Dynkin diagram. An affinized start-shaped root system consists of the set 
\[ R^{aff}_{KM} = \{ \beta + m \delta \mid \beta \in R_{KM} \sqcup \{ 0 \}, ~m\in \Z \} \subset \mf h \oplus \Bbbk \delta \] 
where the pairing on $\mf h \oplus \Bbbk \delta $  is induced from that on $\mf h$. The Weyl group  $W_{R_{KM}}^{aff} \subset \Aut(\mf h \oplus \Bbbk \delta )$ is the set of reflections through real roots and it likewise fits into an exact sequence 
\[ 0 \to \Z[R_{KM}] \to W_{R_{KM}}^{aff} \to W_{R_{KM}} \to 0 .\]

\subsubsection{Baer sum}

Let $\Gamma \subset \Aut(C)$ denote a group so that the Dynkin diagram associated to $\msc C = [C/\Gamma]$ has cartan type $R_{KM}$. Let $R^{aff}_{KM} \subset F$ denote the corresponding affinized star-shaped root system, agreeing with an elliptic root system in the elliptic orbifold cases.  

Given $\xi_1, \xi_2 \in HH_0(\msc C)$ their sum is defined using the group structure on $HH_0(\msc C) \simeq \Ext^1(\mc O_\Delta, \Delta_*\mc T_{\msc C})$ which is the Baer sum of extensions. If $\mc E_{\xi}\in \on{shbimod}(\msc C, \msc C)$ is the bimodule corresponding to the extension $\xi$ define $\mc E_{\psi, \xi}$ to correspond to 
$\delta(+) \in \Ext^1(\mc O_\Delta, \Delta_*\mc T_{\msc C}^{\oplus 2})$ where $\delta: \Hom^0((\mc O_\Delta^{\oplus 2} ,\mc O_{\Delta}) \to \Ext^1(\Delta_*\mc T_{\msc C}^{\oplus 2}), \mc O_\Delta)$  is the connecting map from the exact sequence 
\[ 0 \to \mc O_\Delta^{\oplus 2} \to \mc E_{\psi} \oplus \mc E_{\xi}  \to \Delta_*\mc T_{\msc C}^{\oplus 2} \to 0 \] 
and $+\in \Hom^0((\mc O_\Delta^{\oplus 2} ,\mc O_{\Delta})$ is the addition.  Then we have a canonical maps
\[  \mc E_{\psi + \xi} \xrightarrow{B_+} \mc E_{\psi, \xi} \xleftarrow{B_{\oplus}} \mc E_{\psi} \oplus \mc E_{\xi}\]
which for any $a \in D^b_{Coh}(\msc C)$ induce maps 
\begin{equation}
  a\otimes_{\msc C} \mc E_{\psi + \xi} \xrightarrow{1 \otimes B_+} a\otimes_{\msc C}\mc E_{\psi, \xi}, ~~~  a\otimes_{\msc C} ( \mc E_{\psi}\oplus \mc E_{\xi})  \xrightarrow{1 \otimes B_{\oplus} } a\otimes_{\msc C}\mc E_{\psi, \xi}.
\end{equation}
Because $\ker(b_{\oplus}) \simeq \mc O_{\Delta}$ under the inclusion $\mc O_{\Delta} \xrightarrow{(1,-1)} \mc O_{\Delta}^{\oplus 2} \to \mc E_{\psi} \oplus \mc E_{\xi}$ and map $\eta = (\eta_1, \eta_2): a\otimes_{\msc C}(\mc E_{\psi}\oplus\mc E_{\xi}) \to b$ such that the induced maps $a \to a\otimes_{ \msc C}\mc E_{\psi} \xrightarrow{\eta_1} b$ and $a \to a\otimes_{ \msc C}\mc E_{\xi} \xrightarrow{\eta_2} b$ agree factors through $\overline{\eta}: a\otimes_{\msc C}\mc E_{\psi, \xi} \to b$ and thus induces the Koszul data 
\begin{equation}\label{eq:baer_koszul_data}
(a,b, \overline{\eta}\circ (1\otimes b_+ )) \in \Kos(\mc E_{\psi + \xi}).
\end{equation}

\subsubsection{Root hyperplanes} 

We now complete the identification between roots and hyperplanes in $HH_0$. 

Generalizing Section \ref{ssec:equiv_w_comm}, given a class $\beta \in K_0([C/\Gamma])$ there is a subspace $S_\beta$ defined by 
\begin{equation}\label{eq:sbeta} S_\beta = \{ s \in HH_0(\msc C) \mid 0 \to b \to b\otimes \mc E_s \to b \otimes \mc T_{\msc C} \to 0 \text{ splits if } [b] = \beta \}. \end{equation}

Let $R_{\msc C}^{re} \subset K_0(\msc C)$ denote the set of exceptional objects. These span $R_{\msc C} = K_0(\msc C)$ unless $\Gamma = 1$ in which case we assume $C = E$ and make this identification the definition. 

\begin{prop}\label{prop:roots_are_hyperplanes}
  There is an isomorphism 
  \[ HH_0([C/\Gamma]) \simeq F^\vee, ~~~ R_{\msc C} \simeq R^{aff}_{KM}\]
  such that the root hyperplanes $\beta^\perp$ correspond to hyperplanes $S_\beta$. 
  \end{prop}
  \begin{proof}
    First we need to show that $S_\beta$ is a linear subspace. Clearly this is invariant under scaling extension classes. Given $\psi, \xi \in S_\beta$ let $\eta_1: a\otimes \mc E_\psi \to a$ and $\eta_2: a\otimes \mc E_\chi \to a$ denote the projections onto the first terms induced by the splitting for $[a] = \beta$. Then the Koszul data \eqref{eq:baer_koszul_data} corresponds to a split extension. Thus $S_\beta$ is a linear subspace. 

    For real roots we have $\dim\Ext^1(b\otimes \mc T_{\msc C}, b) = \dim \Hom(b,b) = 1$. In general by pairing with $1 \in \Hom(b,b)$ we get a map 
    \[ \Ext^1(b\otimes \mc T_{\msc C},b) \simeq \Hom(b,b)^\vee \to \C \] 
    and composing with the induced map $HH_0(\msc C) \to \Ext^1(b\otimes \mc T_{\msc C},b)$ gives the map $\langle \beta, -\rangle : HH_0(\msc C) \to \C$ which is generically non-zero. That the corresponding root system is the affinization of the Kac-Moody root system is one of the results of \cite{Burban_Schiffmann_2013}. 
  \end{proof}

\subsubsection{Imaginary reflections, Fourier transforms and Weyl groups} 

By definition for an elliptic root system $R^{ell}$ the elliptic Weyl group $W_{R}^{ell}$ preserves $\on{rad}(\langle-,-\rangle)$ pointwise. Geometrically in the elliptic orbifold cases 
under the identification of Proposition \ref{prop:roots_are_hyperplanes} and $R_{KM}^{aff} \simeq R^{ell}$, 
when we have an action 
\[ W_R^{ell} \to \GL(HH_0([E/\Gamma])) \] 
the elliptic Weyl group will only be a subgroup of those automorphisms induced by families of derived equivalences. 

Because of this in the elliptic orbifold or $A_{-1}$ cases we let 
\begin{equation}\label{eq:GW_def} 
IW_R^{ell} \subset GL(HH_0[E/\Gamma])
\end{equation} 
denote the group generated by $W_R^{ell}$ and $\Aut(D^b_{Coh}([E/\Gamma]))$, where by convention the latter contains anti-autoequivalences. Because the group of autoequivalences surjects onto $\Aut(\on{rad}(\chi_{[E/\Gamma]}(-,-))) \simeq \GL(2, \Z)$ \cite{Burban_Schiffmann_2013} we have an exact sequence 
\[ 0 \to W_R^{ell} \to IW_R^{ell}\to \GL(2, \Z) \to 0.  \] 

when the elliptic root system is seen as the root system of a toroidal algebra $\mf g_{tor} \supset \Hom(\C^{\ast 2}, \mf g)$ the imaginary reflections correspond to automorphisms of the underlying torus $\C^{\ast 2}$. 

In particular, there are derived equivalences in $\Aut(D^b_{Coh}([E/\Gamma]))$ whose induced adjoint action on $GW^{ell}_R$ exchange the two left hand terms isomorphic to $\Z[R_{fin}]$ in \eqref{eq:elliptic_weyl_split} and \eqref{eq:affine_weyl_split} which are both normal subgroups of $W_R^{ell}$. 

A reflection group lying between $IW_R^{ell}$ and $W_R^{ell}$ with applications to monodromy is $W_{R,G}^{ell}$ which depends on a marking $G$. Define 
\begin{equation}\label{eq:monodromy_weyl_ell} W_{R,G}^{ell} := \{ w \in IW_R^{ell} \mid w G = G \}  \end{equation}
so that since the stabilizer of a line in a lattice is $\Z \rtimes \Z/2\Z \subset \GL(2, \Z)$ we have an exact sequence 
\[ 0 \to W_R^{ell} \to W_{R,G}^{ell} \to \Z\rtimes \Z/2\Z \to 0 \] 
and letting $W_R^{aff, \pm}\subset \Aut(\widehat{\mf h})$ denote the group generated by $W_R^{aff, \pm}$ and the reflection through the imaginary root hyperplane, we have an exact sequence 
\begin{equation}\label{eq:affine_translation_ell_seq}
0 \to \Z[R_{aff}] \to W_{R,G}^{ell} \to W_R^{aff, \pm} \to 0 
\end{equation}
analogous to \eqref{eq:elliptic_weyl_split}. When $R = A_{-1}$ we have $W_{A_{-1}}^{ell} = \{ 1 \}, IW_{A_{-1}}^{ell} = \GL(2, \Z)$ and $W^{ell}_{A_{-1},G} = \Z \rtimes \Z/2\Z$.

\subsection{Reflection functors} 

Intending to produce analogues of the geometric reflections in \cite{Nakajima_2003} using the derived category we study the analogues of the functors for the deformed preprojective algebra \cite[\S 3]{Crawley-Boevey_Kimura_2022} generalizing those on the preprojective algebra from \cite{Buan_Iyama_Reiten_Scott_2009,Sekiya_Yamaura_2013, Baumann_Kamnitzer_2012,Baumann_Kamnitzer_Tingley_2014}. 

\subsubsection{$A_{-1}$ case}

We provide proofs of the reflection functors in the non-orbifold case $R = A_{-1}$. 

\begin{thm}\label{thm:derived_reflection}
There is a group $G$ action on $D_{\Coh}(\PP_{E_S}(\Xi))$ and an action $\rho: G \to \End(HH_0(E))$ such that $g$  becomes $S$-linear after base change along $\rho(g)$. The map $\rho$ factors through the action of $\GL(2, \Z)$ on $HH_0(E)$. 
\end{thm}
\begin{proof} 
The group $G$ is generated by Verdier duality from Section \ref{ssec:verdier} and the group of autoequivalences $\Aut(D^b_{Coh}(E))$ of the underlying curve which acts by \eqref{eq:derived_on_glueing}. Because Verdier duality acts via reflection through the commutative hyperplane and the action of derived equivalences is equivalent to the action of $\Aut(D^b_{Coh}(E))$ on $HH_0(E)$ we have that $G$ factors through $\GL(2, \Z) = IW^{ell}_{A_{-1}}$.  
\end{proof} 

\subsubsection{Other cases} 

In the other cases we were unable to find a nice formula for the deformed reflection functors corresponding to real roots. We expect the following 

\begin{conj} \label{conj:reflection}
  There is a group $G$ action on $D_{\Coh}(\PP_{[E/\Gamma]_S]}(\Xi))$ and an action $\rho: G \to \End(HH_0([E/\Gamma]))$ such that $g$  becomes $S$-linear after base change along $\rho(g)$. The map $\rho$ factors through the action of $IW^{ell}_R$ on $HH_0(E)$. 
\end{conj}

The exact same formulas as in Theorem \ref{thm:derived_reflection} provide for the imaginary reflections and when restricting to the hyperplane $S_\alpha \subset S$ where there are torsion sheaves of class $\alpha$ for $\alpha \in R^{ell}_{re}$ the reflection is given by the spherical twist functor. 

More generally, for the real roots we expect an action of a group extending the affinized star-shaped weyl group to act on $D_{\Coh}(\PP_{[C/\Gamma]}(\Xi))$ for $C$ of higher genus. 

\section{Bridgeland stability conditions}
\label{sec:stability}
While the construction of \ref{ssec:git_affine_fibrations} allows for a construction of moduli stacks regardless of the framing sheaf $\mc F$, when $\mc F$ is simple and of trivial monodromy we have an additional description on $M^{tf}(v, \mc F)$ and its singular models together with their deformations using Bridgeland stability conditions.over the central commutative surface as well as in families over $HH_0(\msc C)$. Precisely speaking we will only prove the main claims for $\Gamma = \{ 1 \}$. 

\subsection{Central surface}

The central surface is $\overline{\msc X_R} = [E\times \PP^1/ \Gamma]$ and derived equivalent to a smooth projective surface $\overline X_R$ and thus the usual construction of stability conditions applies in this situation. We will be interested in a sequence of subspaces 
\[\cdots \subset U_n \subset \cdots  \subset U_1\subset U \subset \Stab(\overline{\msc X_R}) \] 
of the space of geometric stability conditions which give rise to relatively ample line bundles in $N^1(X_R^{[n]}/\A^n)$.

\subsection{Deformed surfaces}

For the deformed surfaces $\PP_{E}(\Xi_s)$ in the $R = A_{-1}$ case for $s \neq 0$ we study ingredients necessary for stability conditions on their derived categories. First, in \ref{ssec:ns_family}-\ref{ssec:amp_fam} we study divisors for general $R$ and in families.

\subsubsection{Neron-Severi lattice}\label{ssec:ns_family}

Recall from \cite[\S 7]{rains19birational_arxiv} that the rank of $\mc E \in K(\msc X)$ is defined to be the difference $\rk(\mc E) = r_2 - r_1$ of the class of $\mc E = \mc E_1 + \mc E_2$ under the decomposition of Proposition \ref{prop:sod_pp_univ}. The numerical Grothendieck group is $K_{num}(\msc X) = K_0(\msc X)/ \Pic^0(\msc C) \oplus \Pic^0(\msc C)$, the identity component being trivial unless $\Gamma = \{1\}$. The N\'eron-Severi lattice is $\ker(\rk(-))/ \Z[\mrm{pt}]$ and the first chern class $c_1(\mc E)$ is defined to be the equivalence class of $[\mc E]-\rk(\mc E)[\mc O_{\msc X}]$ in $\NS(\msc X)$. The intersection pairing $\langle-,-\rangle$ on $\NS(\msc X)$ is equal to $-\chi(-,-)$ which is well defined. An effective divisor is $c_1(\mc E)$ for a 1-dimensional sheaf $\mc E$. An ample divisor is one which pairs positively with all effective divisors. We make likewise define the \emph{uniformly numerical relative Grothendieck group} $\mc N(D^b_{Coh}(\msc X_S)/S)$ exactly as in \cite[\S 21.5]{Bayer_Lahoz_Macrì_Nuer_Perry_Stellari_2021}. 

It follows that $\msc X_S$ has constant Picard rank over $S$, which satisfies 
\[\rk \NS(\msc X_s) = \rk(\mc N(D^b_{Coh}(\msc X_S)/S)) -2.\] Since we are concerned only with objects in the derived category framed trivially at $D_\infty$ all relevant sheaves are expected to come from the first factor in the semiorthogonal decomposition of Conjecture \ref{conj:SOD_blowup} which would also have constant Picard rank. Likewise we are primarily interested in a particular limit, inspired by the F-theory constructions of string theory, of vanishing volume of the elliptic fiber in the central surface and its analogues in other fibers.

\subsubsection{Ample cone} \label{ssec:amp_fam}
We want to understand relatively ample divisors $\omega$ over $S$.  
In the $\Gamma = \{ 1 \}$ case the cone of effective divisors is generated by $[E_S]:= c_1(\mc O_{D_\infty\times S})$ and $[P_S]:= c_1(\mc F_S)$ where $\mc O_{D_\infty \times S}$ is such that its Koszul data is 
\[ \kappa^{-1}(\mc O_{D_\infty \times S}) ) = (\mc O_{E \times S}, \mc O_{E \times S}, \eta)\]
and $\mc F_S$ has Koszul data 
\[ \kappa^{-1}(\mc F_S) = (0, \Bbbk_e, 0) \] 
where $\eta$ is the quotient by the inclusion $\mc O_{E \times S}\to \mc O_{E \times S}\otimes_{E\times S} \Xi$.We have $\langle E_S, P_S \rangle = 1$ and $\langle E_S, E_S \rangle, \langle P_S, P_S \rangle =0$ exactly as in the central fiber and the ample cone for a closed point $s \in S$ of $\PP_E(\Xi_s)$ is the positive cone of $\NS(\bbP_{E}(\Xi_s))_\Q = {\RN{2}_{1,1}}_\Q$. We thus can identify the uniformly numerical relative Grothendieck group as 
\[ \mc N(\PP_{E_S}(\Xi)/S) \simeq {\RN{2}_{1,1}}\oplus {\RN{2}_{1,1}}. \] 

\subsubsection{Other surfaces} \label{eq:ns_lattice_other_types}

For $R \neq A_{-1}$ we have an identification 
\[ K_{num}(\msc S_R) = K_{num}([E/\Gamma]) \oplus K_{num}([E/\Gamma]) \] 
and the semiorthogonal decomposition predicted by Conjecture \ref{conj:SOD_blowup} should give identifications 
\[ \NS(\msc X_R) \simeq \RN{1}_{1,9}\] 
matching the result on the central fiber, $\NS(S) \simeq \RN{1}_{1,9}$ for an rational elliptic surface, and the uniformly numerical relative Grothendieck group is 
\[ \mc N(\msc X_R) \simeq \RN{2}_{1,1} \oplus \RN{1}_{1,9}. \]

In particular these results hold for $R= D_{4}$ where no noncommutative blowdown is necessary. 

\subsection{Harder-Narasimhan structures} 

Here we specialize to $R = A_{-1}$ and describe Harder-Narasimhan filtrations for noncommutative ruled surfaces and semistable reduction giving rise to HN-structures over curves. 

\subsubsection{Harder-Narasimhan filtrations over points} 

Pick a numerically constant relatively ample $H \in \NS(\msc X_S)$ and another element $B\in \NS(\msc X_S)$. This defines as usual for each $s \in S$ 
\[ \on{ch}_s: K_{num}(\msc X_s) \to \Q^{\rk K_{num}(\msc X_s)}\]
and its global analogue $\on{ch}: \mc N(\msc X_S) \to \Q^{\rk \mc N(\msc X_S)}$.  Letting $\on{ch}^B(-) = e^{-B}\on{ch}(-)$ for $B  \in \NS(\msc X_S)_\Q$ numerically constant we define 
\[ v^B(-) = (\int_{\msc X_s} \omega^2 \on{ch}_0^B(-), \int_{\msc X_s} \omega \on{ch}_1^B(-), \int_{\msc X_s} \on{ch}_2^B(-))\]
where the integrals denote the natural pairing on the Grothendieck group.

First we construct a maximal destabilizing subobject for any  $\xi \in S$.  Let $\msc X = \msc X_{\xi}$. 

Let $\mu_{H,B}(-) = v_1^B(-)/v_0^B(-)$ and omit $B$ if $B = 0$. Let $\nu_{H,B}(-) = v_2^B(-)/v_0^B(-)$. Let $s \in S$ be a point. We define Gieseker stability, slope stability and Matsuki-Wentworth stability exactly as in the case on the commutative surfaces following the theory developed in \cite{rains2019noncommutative,rains19birational_arxiv} of the usual results except for potentially boundedness results for families of Gieseker semistable sheaves.

\begin{defn}
A sheaf $\mc F \in \Coh(\msc X)$ is $(H,B)$ with $r> 0$ is twisted (semi)-stable if for every $0 \neq \mc G \subset \mc F$ we have $\mu_{H,B}(\mc G) \lneq \mu_{H,B}(\mc F)$ or ($\mu_{H,B}(\mc G) = \mu_{H,B}(\mc F)$ and  $\nu_{H,B}(\mc G) (\le) \lneq \nu_{H,B}(\mc F)$. )
\end{defn}

\begin{prop}
Every $\mc F \in \Coh(\msc X)$ has a maximal destabilizing subobject $\mc F' \subset \mc F$. 
\end{prop}
\begin{proof}
The proof given in \cite[Lemma 1.3.5]{huybrechts_lehn_2010} carries through without change to the setting of modules over the $\Z$-algebra $\mc A^{\mc E_\xi}$. 
\end{proof} 

It follows that there exist for any $\mc F \in \Coh(\msc X)$ a filtration 

\[ 0 = \on{HN}_0(\mc F)\subset \on{HN}_1(\mc F) \subset \cdots\subset \on{HN}_n (\mc F) = \mc F\]
such that $\on{HN}_i(\mc F)/\on{HN}_{i-1}(\mc F)$ is $(H,B)$-twisted semistable.  

\subsubsection{Semistable reduciton} 

Crucial to the study of stability in families is an analogue of semistable reduction.

Recall \cite{rains19birational_arxiv} that a family $\mc F \in \Coh(\PP^1_{\msc X\times R}(\Xi))$ is called $R$-flat if 
\[ \mc F\otimes_R^{L}\Bbbk(s) \in \Coh(\PP^1_{\msc X\times R}(\Xi))\] 
for any point $s \in \Spec(R)$ and that $R$-flat sheaves are perfect. 

Recall also that there is an inclusion 
\[ M \to \ad \ad M \] 
for the duality functor $\ad$ from \eqref{eq:ad}. A sheaf $\mc F$ is called reflexive if $\mathbb{D}(\mc F)$ is a sheaf, and this implies that $\ad \ad \mc F = \mc F$. 

Let $(R, (\pi))$ be a discrete valuation ring with residue field $k$ and fraction field $K$. 
\begin{prop}\label{prop:ss_red}
Let $\mc F \in \Coh(\PP^1_{\msc X\times R}(\Xi))$ be $R$-flat such that $\mc F_{K}$ is $(H,B)$-semistable. Then there exits an $\mc G \subset \mc F$ such that $\mc G_K = \mc F_K$ and $\mc G_K$ is $(H,B)$-semistable.
\end{prop}
\begin{proof}
  The usual proof of Langton's theorem \cite[\S 2.B]{huybrechts_lehn_2010} carries through with the notions of flatness above and Rains' construction of the double dual  for noncommutative ruled surfaces on account of the fact that $\Coh(\PP^1_{\msc X\times R}(\Xi))$ is a Noetherian abelian category. 
\end{proof} 

\subsection{Construction of stability conditions}

\subsubsection{Tilt stability over points} 
We restrict to $R = A_{-1}$. 

Tilt stability for $\PP_E(\Xi_s)$ is a mild generalization from commutative surfaces.

\begin{prop}\label{prop:torsion_pair_tilt}
The pair of subcategories 
\begin{align*} \mc T^{H,B}_s &= \langle  E \in \Coh(\PP_E(\Xi_s)) \text{ semistable with } \mu_H(E) > \beta\rangle \\
 \mc F^{H,B}_s &= \langle  E \in \Coh(\PP_E(\Xi_s)) \text{ semistable with } \mu_H(E) \le \beta \rangle
\end{align*}
for a torsion pair. 
\end{prop}
\begin{proof}
This is an immediate consequence of the existence and uniqueness of Harder-Narasimhan filtrations for objects in $\PP_E(\Xi_s)$
\end{proof}

Let $\mc B_{H,B,s} = \langle \mc F^{H,B}_s[1], \mc T^{H,B}_s\rangle$ denote the tilted heart with respect to the torsion pair from Proposition \ref{prop:torsion_pair_tilt}. 

A closely related torsion pair has been studied in the context of noncommutative $\PP^1$ bundles related to the usual perverse heart.  Recall from \cite{ben2008perverse} that for each closed $s \in S$ we have a torsion pair $\Coh(\bbP_E(\Xi_s)) = \langle \Coh_0(\bbP_E(\Xi_s)), \Coh_{>0}(\bbP_E(\Xi_s))\rangle$ induced by the canonical inclusion of the maximal subobject in $\Coh(\bbP_E(\Xi_s))$ of linear growth. This is clearly well defined over $S$. Let $\mc A_{0,s} =  \langle \Coh_0(\bbP_E(\Xi_s))[1], \Coh_{>0}(\bbP_E(\Xi_s))\rangle$ denote the abelian category formed by tilting with respect to this torsion pair. 

While the tilted hearts associated to the torsion pairs of Proposition \ref{prop:torsion_pair_tilt} are suited to the study of stability conditions, all of the objects in the semistable moduli spaces will live in a category $\mc A_{0,s}$ or its image under a derived equivalence. While universal openness of lying in $\mc B_{H,B,s}$ will follow from \cite[Prop. 20.8]{Bayer_Lahoz_Macrì_Nuer_Perry_Stellari_2021}, universal openness of flatness for the hearts $\mc A_{0,s}$ can be shown by a standard argument. 

\subsubsection{}

Recall the notion of a stability condition in the relative setting introduced in \cite{Bayer_Lahoz_Macrì_Nuer_Perry_Stellari_2021}. We still restrict to $R = A_{-1}$. Let $H$ and $B \in \NS(\overline X_S)$ be numerically constant with $H$ ample. 

Let $Z_{H,B,s}(-) = \int_{\overline X_s} e^{iH}\on{ch}_s^B(-)$ be a stability function. 

\begin{thm}\label{thm:stability_family}
The collection $\underline{\sigma}_{H,B} = (Z_{H,B,s}, \mc B_{H,B,s})_{s \in S}$ is a stability condition on $D^b_{Coh}(\PP_{E_S}(\Xi))$ over $S$ with respect to $\RN{2}_{1,1} \oplus \RN{2}_{1,1}$. 
\end{thm}
\begin{proof}
  We check a number of requirements. 

1) The individual stability conditions are Bridgeland stability conditions. 

We have to show that $Z_{H,B,s}$ is a stability function on $\mc B_{H,B,s}$ satisfying the Harder-Narasimhan property \cite{Bridgeland_2007}. This is known to follow from the Bogomolov inequality for $\mu_{H, B}$-semistable torsion-free sheaves \cite{Macrì_Schmidt_2017}. Letting 
\[ \Delta(\mc E) = \on{ch}_{s,1}^B(\mc E)^2 - 2\on{ch}_{s,0}^B(\mc E)\cdot \on{ch}_{s,2}^B(\mc E) \] for $\mc E \in \Coh(\overline{X}_s)$ we need to show that $\Delta(\mc E) \ge 0$ for $\mc E $ which is $\mu_{H,B}$-semistable torsion-free. Let $S' = \C s \subset HH_0(E)$ denote the line in $S$ given by scaling the extension class. The stack $\mc M_{H,B}(\on{ch}, \PP_E(\Xi_{S'}))/S'$ of $\mu_{H,B}$-semistable torsion-free sheaves is proper over $S'$ \cite{rains2019noncommutative}. Given $\mc E \in \Coh(\overline{X}_s) $ which is $\mu_{H,B}$-semistable torsion-free we get a map $\C^\ast \to \mc M_{H,B}(\on{ch}, \PP_E(\Xi_{S'}))/S'$ sending $1$ to $\mc E$ using the equivalence induced by scaling the extension in $HH_0(E)$ which by properness extends to $f: \C \to\mc M_{H,B}(\on{ch}, \PP_E(\Xi_{S'}))/S'$ which gives a $\mu_{H,B}$-semistable torsion free sheaf with the same chern character $\mc E_0  \in \Coh(\overline{X}_0)$ where the Bogomolov inequality is known. It follows in the case of rational $H$ and $B$ that $\mc B_{H,B,s}$ with stability function $Z_{H,B,s}$ satisfies the Harder-Narasimhan property. The support property for $s \neq 0$ follows by the same properness argument as above. The rest of the standard proof that $\sigma_{H,B,s}$ is a stability condition for general $H$ and $B$ using Bridgeland's Deformation Theorem of stability conditions follow s through as usual. 

2) The collection $\underline{\sigma}_{H,B}$ universally locally has constant central charges. 

This trivially follows from construction. 

3) The collection $\underline{\sigma}_{H,B}$ universally satisfies openness of geometric stability. 

The argument that applies in the case of commutative surfaces \cite[Prop. 25.3]{Bayer_Lahoz_Macrì_Nuer_Perry_Stellari_2021} carries through exactly onces the boundedness of Quot schemes is verified which is known for these noncommutative surfaces \cite[Thm. 10.37]{rains2019noncommutative}, since we have semistable reduction which is Proposition \ref{prop:ss_red} and the flattening stratifications for sheaves in \cite{rains19birational_arxiv} it follows that destabilizing sequences extend over specialization. 

4) The collection $\underline{\sigma}_{H,B}$ integrates to a HN structure over any Dedekind scheme $C\to S$ essentially of finite type. 

The argument in \cite[p. 115]{Bayer_Lahoz_Macrì_Nuer_Perry_Stellari_2021} using \cite[Thm. 18.7]{Bayer_Lahoz_Macrì_Nuer_Perry_Stellari_2021} carries through using semistable reduction and the universal openness of flatness of the tilted hearts defined be Proposition \ref{prop:torsion_pair_tilt}, which in turn follows from the flattening stratification \cite{rains19birational_arxiv}, using \cite[Proposition 25.1]{Bayer_Lahoz_Macrì_Nuer_Perry_Stellari_2021} and the openness of reflexivity which follows from semicontinuity.   
\end{proof}

\subsection{Moduli spaces} 

Having established the existence of relative stability conditions the existence of algebraic spaces serving as good moduli spaces follows from the usual theory. Using \cite{toen2007moduli}, the derived stack $\mc M_{\msc X/B}$ of objects in $\on{Perf}(\msc X/B)$ has subfunctor $\mc M^{ug}_{\msc X/B}$ of universally glueable objects which is represented by an algebraic space locally of finite type. Over Noetherian test schemes, these will coincide with objects with coherent cohomology.  Then the results of \cite{Bayer_Lahoz_Macrì_Nuer_Perry_Stellari_2021} produce open substacks of semistable objects.  

\begin{thm}[{{\cite[Thm 21.14]{Bayer_Lahoz_Macrì_Nuer_Perry_Stellari_2021}}}]\label{thm:cite_moduli_spaces}
  For any $v \in \Lambda$
\begin{enumerate} 
\item The moduli stack of semistable objects $\mc M_{\underline{\sigma}_{H,B}}(v)$ is an algebraic stack of finite type over $S$.
\item The moduli stack $\mc M_{\underline{\sigma}_{H,B}}(v)$ admits a good moduli space $M_{\underline{\sigma}_{H,B}}(v)$.
\item If $\mc M_{\underline{\sigma}_{H,B}}(v) = \mc M^{st}_{\underline{\sigma}_{H,B}}(v)$ then $\mc M_{\underline{\sigma}_{H,B}}(v)$ is a $\Bbb{G}_m$-gerbe over $M_{\underline{\sigma}_{H,B}}(v)$. 
\end{enumerate}
\end{thm}

\subsubsection{Line bundle}
An analogue of the Bayer-Macr\`i Positivity lemma requires a construction of a determinant line bundle. While a natural tensor product structure on $D^b_{Coh}(\PP_{E_S}(\Xi))$ would be necessary for the usual push-pull construction we avoid this by using the bimodule structure to define the replacement for the pushforward and the tensor product simultaneously, using the construction of noncommutative $\PP^1$ bundles over arbitrary base and the internal Hom functor \eqref{eq:rhom_NC}. 

Let $T$ be a test scheme or stack and let $\mc U \in D^b_{Coh}(\PP_{T\times E_S}(\mc O_T \boxtimes \Xi))$ denote the universal family. Given 
$F \in K_{num}(\PP_{E_S}(\Xi))$ we can produce by base change $F_T \in K_{num}(\PP_{T\times E_S}(\mc O_T \boxtimes \Xi))$ and then we define 
\[ \lambda_{\mc U}(F) := \det(R Hom_{T_S}(F_T, \mc U)) \in N^1(T_S/S)\] 
to be the determinant line bundle associated to $F$. 

Fixing $v$, use the $\widetilde{\GL^+(2, \R)}$ action to rotate a stability condition $\underline(\sigma) = (Z_s, \mc B_s)_{s\in S}$ so that $Z_s(v) = -1$ for some (and hence all) $s \in S$. There exits $F_Z \in K_{num}(\PP_{E_S}(\Xi))_\R$ so that $\im Z_s(-) = \chi(F_{Z_s}, -)$. The construction avoiding a taking a tensor product with the universal bundle now carries through. 

\begin{thm}[{{\cite{Bayer_Lahoz_Macrì_Nuer_Perry_Stellari_2021,Bayer_Macrì_2014}}}]
The formula 
\[ \ell_{\underline \sigma} = [\lambda_{\mc U}(F_Z)] \in N^1(\mc M_{\underline \sigma}(v)/S) \] 
defines a positive numerical relative Cartier divisor on $\mc M_{\underline \sigma}(v)$. Furthermore, $\ell_{\underline{\sigma}}$ descends to a divisor $L_{\underline{\sigma}}$ on $M_{\underline{\sigma}}(v)$. 
\end{thm}

\subsection{Stability and quasiprojectivity} 

The good moduli spaces $M_{\underline{\sigma}}(v)/S$ for specific choices of $\underline{\sigma}$ will be shown to be projective compactifications of deformations of $M^{tf}(v; \mc F)$, exhibiting the latter as quasiprojective over $S$.

\subsubsection{Determinant map}

Let $J = \on{Jac}_E\times S \simeq E_S$ be the Jacobian of $D_\infty/S$. 

There is a map 
\[ \on{det}_S: M_{\underline{\sigma}}(v)/S \to J/S \] 
closely related to the one referred to in \cite{rains19birational_arxiv} as the determinant-of-restriction map 
which sends $\mc E$ to $\on{det}(E|_{D_\infty})$. Furthermore we can exhibit $M_{\underline{\sigma}}(v)/S$ as a $J$-torsor since there is an action 
\[ J\times_S M_{\underline{\sigma}}(v) \to M_{\underline{\sigma}}(v)/S \]
given by "tensoring by a line bundle" $L \in \Pic^0(E)$, or more precisely the map 
\[ (a,b, \eta) \mapsto (a\otimes L, b\otimes L, \Ad_{-\otimes L} \eta)\] 
at the level of Koszul data. Picking $\chi \in \on{Jac}(E)$ a closed point and letting $\chi_S = \chi \times S$ let 
\[M^\chi_{\underline\sigma}(v) = \on{det}_S^{-1}(\chi_S)\] 
denote the moduli spaces with fixed determinant. Define the moduli stack $\mc M^\chi_{\underline{\sigma}_0}(v)$ analogously.

Restricting to the central fiber with stability condition $\underline{\sigma}_0$, consider the subspace
\[ M_{\underline{\sigma}_0}(v)_{\mc F} := \{ \mc E \in M_{\underline{\sigma}_0}^\chi(v) \mid \mc E|_{D_\infty} \simeq \mc F \} \]
where $\mc F \in \Coh(E)$ satisfies $\det(\mc F ) = \chi$ and $c_1(\mc F)$ is the $P$ component of the $c_1$ component of $v$. 

This is well defined because this property is preserved under S-equivalence and it is a dense open subset because the closed point $[\mc F]$ is a dense open subset of the stack of coherent sheaves on $E$ with fixed determinant.

We will largely be focused on the subspace of $M_{\underline{\sigma}}^\chi(v)$ naturally deforming $M_{\underline{\sigma}_0}(v)_{\mc F}$ namely the subfunctor whose closed points are the set 
\begin{equation}\label{eq:main_family}
   M_{\underline{\sigma}}(v)_{\mc F} := \{ \mc E \in M^\chi_{\underline{\sigma}}(v) \mid \mc E|_{D_\infty} \simeq \mc F \}. 
\end{equation}

\subsubsection{Symplectic structure}

For generic $\underline{\sigma}$ we see that $M_{\underline{\sigma}}^\chi(v)_{\mc F}$ is a family of objects in the standard heart of $D^b_{Coh}(\PP_{E_S}(\Xi))$. For any $\mc E \in M_{\underline{\sigma}}^\chi(v)_{\mc F}$ we define 
$\mc E(-{D_\infty})$ as the cone of the map 
\[ \mc E(-D_\infty) = \on{Cone}(\mc E \xrightarrow{\on{res}} \iota_* \mc F_S)[-1]\]
where $\on{res}$ and its target are defined at the level of Koszul data by 
\[ (a,b, \eta) \xrightarrow{\on{res}} (0, \mc E|_{D_\infty}, 0) \] 

where $\mc E|_{D_\infty}$ is defined in \eqref{eq:res_to_infty} and the map $b \to \mc E|_{D_\infty}$ is the canonical one. Define $\mc E(-kD_\infty)$ inductively as $\mc E(-(k-1)D_\infty)(-D_\infty)$. We have an equality 
\[ \mc E(-2D_\infty)[2] = \Bbb{S}(\mc E).\]

Furthermore via the exact triangle 
\[ \to \mc E(-D_\infty) \to \mc E \to \iota_* \mc F_S \to  \]
we have a long exact sequence 
\begin{alignat*}{2}
\cdots \to& \Ext^0(\mc E , \mc E) \to& \Ext^0(\mc F, \mc F)\to \\
\Ext^1(\mc E, \mc E(-D_\infty)) \to& \Ext^1(\mc E , \mc E) \to& \Ext^1(\mc F, \mc F)\to 
\end{alignat*}
the middle line of which is exact and is seen to be the fiber over a point $\mc E \in \mc M_{\underline{\sigma}}(v)$ of the tangent sequence 
\[0 \to  \mc T_{M_{\underline{\sigma}}^\chi(v)_{\mc F}/\mc M_{{\underline{\sigma}}}(v)}\to \mc T_{\mc M_{\underline{\sigma}}(v)} \to \on{det}^* \mc T_{\on{Jac}(E)}\to 0 \] 
induced by the fiber sequence $M_{\underline{\sigma}}^\chi(v)_{\mc F} \to \mc M_{\underline{\sigma}}(v) \to \on{Jac}(E)$. In particular we have 
\[ T_{[\mc E]} M_{\underline{\sigma}}^\chi (v)_{\mc F} \simeq \Ext^1(\mc E, \mc E(-D_{\infty})). \] 

Since we have an embedding 
\[ \mc M_{\underline{\sigma}}(v) \to \mc Spl_{\msc X}\] 
to the stack of simple sheaves given by an action of a derived equivalence sending stable objects to sheaves followed by the classifying map, and the latter admits a Poisson structure \cite[\S 11]{rains19birational_arxiv} we see that the nondegenerate pairing given on fibers by
\[ \Ext^1(\mc E, \mc E(-D_\infty))\otimes \Ext^1(\mc E, \mc E(-D_\infty)) \to \Ext^2(\mc E, \mc E(-2D_\infty)) \to \Bbbk\] 
satisfies the Jacobi identity and gives a symplectic structure on $M^\chi_{\underline{\sigma}}(v)_{\mc F}$.

\subsubsection{Stability conditions under duality} 

Given a stability condition $\sigma$ on the homotopy category of a dg category $\mc D$ and a derived equivalence $\phi: \mc D \to \mc D'$ there is an induced stability condition $\phi_* \sigma$ on the homotopy category of $\mc D'$. In particular the derived equivalences from Section \ref{sec:gewg} act on the stability conditions constructed in this section. 

Liu, Lo and Martinez and Lo and Martinez have analyzed the action of derived equivalences on geometric stability conditions on Weierstrass elliptic surfaces \cite{Liu_Lo_Martinez_2019,Lo_Martinez_2022} and many of the results were extended to non-Weierstrass surfaces in \cite{Yoshioka_2022}. In particular $E\times \PP^1$ is Weierstrass and the results apply. Many of the results consider geometric stability conditions $\sigma_{H,B}$ when the polarization $H$ lies in the Friedman chamber applicable to the F-theory limit. The explicit formula for the corresponding stability condition is given by \cite[Eq. 5.0.3]{Lo_Martinez_2022} where in op. cit.  $e = P^2 = 0$ which we only need as the following inexplicit consequence. 

\begin{prop}[\cite{Lo_Martinez_2022,lo2021weight, Lo_Martinez_2022}]
After potentially increasing $b$, the space $U$ is preserved under the action of $\widetilde{\SL(2, \Z)} \subset \Aut(D^b_{Coh}(\PP_{E_S}(\Xi)))$ of derived equivalences. 
\end{prop}

\subsubsection{Stability compatible with Lagrangian fibration}
\label{ssec:stability_lfib}
Choose $v$ such that there is a birational model for $M^\chi_{\underline \sigma_0}(v)$, a smooth moduli space of complexes on the central surface, which admits a fibration $\ell :M^\chi_{\underline \sigma_0}(v) \to \overline B$ to a half-dimensional base $\overline B \simeq \PP^n$ such that the map is equivariant with respect to $\csth$ fixing $0 \in \overline B$ with positive normal weights. Let $B \simeq \A^n$ denote the attracting locus of $0$. This will in particular be satisfied when $v$ is the image under derived equivalence of $v' = (1,c_1, d)$ corresponding to rank one torsion free sheaves in the Gieseker chamber. These will be the only examples studied here so we let $\gamma\cdot (1,c_1, d) = v'$ and $n = -d+c_1^2/2$. 

Letting $H = P + bE$ and $B = cP  +d E$, let  $ \Stab'(D^b_{Coh}(\PP_{E_S}(\Xi)))_v$ denote the subspace of space of stability conditions produced by Theorem \ref{thm:stability_family} defined by 
\begin{multline*} \Stab'(D^b_{Coh}(\PP_{E_S}(\Xi)))_v = \\
  \{ g \cdot \sigma_{H,B} \in \Stab(D^b_{Coh}(\PP_{E_S}(\Xi))) \mid  H^2 \gg 0, \frac{2dc -n}{c^2 + 1} < b, g \in \widetilde{\GL(2, \R)} \}  .\end{multline*}

Consider the map 
\begin{align}
  \Stab'(D^b_{Coh}(\PP_{E_S}(\Xi))) &\to N^1(\mc M_{\underline{\sigma}}(v)/S) \to N^1(M_{\underline \sigma_0}(v)_{\mc F})\label{eq:map_half}\\
  \underline{\sigma} &\mapsto L_{\underline{\sigma}} \mapsto L_{\underline{\sigma}}|_{M_{\underline{\sigma}_0}(v)_{\mc F}}\nonumber
\end{align}

Recall that $\phi(-): \mc N(\PP_{E_S}(\Xi))$ is the phase function of the stability condition. 
\begin{prop}\label{prop:all_walls_are_GU}
Every wall in $\Stab'(D^b_{Coh}(PP_{E_S}(\Xi)))_v$ inducing a contraction of $M^\chi_{\underline\sigma_0}(v)_{\mc F}$ is a wall where $\phi(v) = \phi((0,rE, s))$. 
\end{prop}
\begin{proof}
 This set of Mukai vectors is invariant under derived equivalences and so is this subspace of stability conditions whose class $H$ lies in the Friedman chamber \cite{Lo_Martinez_2022} so we may as well assume that $v = (1,0,1-n)$. Let $w = (r, c_1, d)$ be the class of a destabilizing object. Because of the condition $H^2 \gg 0$ we must have $r = 0$ and the form of $H$ also implies that $b \gg 0$ which implies that if $c_1 = c_{1,P}P + c_{1, E} E$ we must have that $c_{1,P} = 0$ or it would be impossible for the phases to overlap, from which the result follows. 
\end{proof}

A $v$-wall in a subspace of $\Stab(\mc D)$ is a wall which induces a contraction of moduli spaces $M_{\underline{\sigma}}^\chi(v)_{\mc F}$. Suppose a stability condition $\underline{\sigma} \in \Stab'(D^b_{Coh}(PP_{E_S}(\Xi)))$ induces a contraction of $M_{\underline{\sigma}}^\chi(v)_{\mc F}$. Given a destabilizing triangle 
\[ A \to \mc E \to B \to \] 
in a fiber over $s \in HH_0(E)$ for stability condition $\underline{\sigma}'$ on a wall and with $\mc E \in M_{\underline{\sigma}}(v)_{\mc F}$ we get a destabilizing sequence for all $s' \in \C^\ast s$ and by openness of stability a destabilizing sequence for an object in the central fiber. This implies that the set of $v$-walls is a subset of those for the central surface and is therefore a subset of 
\[ \mc W = \{ w_{r,s} \mid (r,s) \in \Z^2, r \ge 0, (r,s) \neq (1,0), s > 0 \text{ if } a = 0 \} \] 
where 
\[ w_{r,s} = \{ g\cdot \sigma_{H,B} \mid \phi(v) = \phi((0,rE, s)) \}.\] 
Up to the action of $g \in \widetilde{\GL(2, \R)}$ we can identify the wall equation as
\begin{equation}\label{eq:wall_rs}
w_{r,s} = \{ \sigma_{H,B} \mid 2cdr - br - nr = sd + bcs - rbc^2 \}. 
\end{equation}

Now let $v$ be any vector which is the image of $v$ under derived (anti-)autoequivalence. 

\begin{thm}\label{thm:vwalls_roots}
The $v$-walls in $\Stab'(D^b_{Coh}(\PP_{E_S}(\Xi)))$ 
are in bijection with roots 
$\beta \in R_{A_{-1}}^{ell}$  under the bijection 
 $K_{num}(E) \simeq R_{A_{-1}}^{ell}$ and the inclusion $K_{num}(E) \to \mc N(\PP_{E_S}(\Xi))$ such that
  $\langle \beta, v\rangle \le \langle v, v\rangle/2$. 
\end{thm}
\begin{proof}
  By employing a derived anti-autoequivalence we can assume that $v = (1,0,-n)$ as above. 
  The only thing left to prove after the above analysis is the restriction on the pairings of the Mukai vectors which follows from the analysis of \cite[\S 8]{Bayer_Macrì_2014_MMP}. 
\end{proof}

Let $U\subset \Stab'(D^b_{Coh}(\PP_{E_S}(\Xi)))$ denote the subspace of the source for which $L_{\underline{\sigma}}|_{M_{\underline{\sigma}_0}(v)_{\mc F}}$ is NEF on a moduli space which admits a fibration to a half-dimensional base $B \subset \PP^n$ agreeing with the one on the original moduli space on a dense open set. 

This produces a map 
\begin{align}\label{eq:Utorel_amp}
U \to N^1(M^\chi_{\underline\sigma_0}(v))' &\to N^1(\ell^{-1}(B)/ B)\\
\nonumber \underline{\sigma} \mapsto L^B_{\underline{\sigma}}
\end{align}
which is the composition of the inclusion into $ \Stab'(D^b_{Coh}(\PP_{E_S}(\Xi)))$, the map \eqref{eq:map_half} onto its image $N^1(M^\chi_{\underline\sigma_0}(v))' \subset N^1(M^\chi_{\underline\sigma_0}(v))$ and the restriction to $\ell^{-1}(B)$. 

Since we know that for $\mc F$ of rank 1 and $\underline{\sigma}_0$ in the Gieseker chamber we have $M_{\underline{\sigma}_0}(v; \mc F) \simeq T^*E^{[n]}$ and we know that $N^1(T^*E^{[n]})$ is spanned by the two divisors which are the elements $|1_{P}, 1^{n-1}\rangle$ and $|2_{\mrm{pt}}, 1^{n-2}\rangle $ created by the relevant Nakajima operators, we can determine a class in $N^1(M_{\underline{\sigma}_0}(v; \mc F))$ or one of its $K$-equivalent birational models by pairing with elements of 
\begin{equation}
  \label{eq:n1_basis_real2} N_1(M_{\underline{\sigma}_0}(v)_{\mc F}) = \R C_1 \oplus \R C_2 
\end{equation} generated by curves $C_1$ and $C_2$ where $C_1$ is an exceptional fiber of the Hilbert-Chow map and $C_2 \in N_1(T^*E^{[n]})$ is represented by a curve $C_2 \simeq E$ with $n-1$ points fixed and one point tracing out a fiber of $T^*E \to \A^1$. These pairings can be done in $\RN{2}_{1,1} \oplus \RN{2}_{1,1}$ which is the target of the Mukai vector in Theorem \ref{thm:stability_family} where under the Mukai homomorphism $C_1 = \langle ((n,1,0,0), -\rangle$ and $C_2 = ((0,0,0,1))$ where $(a,b,c,d)$ corresponds to $a\pmb{1} + c P + d E + b\mrm{pt} \in H^*(E\times \PP^1)$.  

Restricting for simplicity to the case of $v = (1,0,-n)$ We also have the explicit form of the class of the Bayer-Macr\`i line bundle \cite{Bayer_Macrì_2014} as 

\[ \ell_{\underline{\sigma}_0} = (r_{\underline{\sigma}},  s_{\underline{\sigma}} ,c_{\underline{\sigma}},) \] 
where 
\begin{align}
r_{\underline{\sigma}} &= B\cdot H\nonumber  \\
c_{\underline{\sigma}} &= -B\cdot H B + (-n + (B^2-H^2)/2)H \label{eq:bm_exactly}\\
s_{\underline{\sigma}} &= -n B\cdot H.\nonumber
\end{align}

Let $N^1(T^*E^{[n]}/\A^n)_{> 0} = \{ D \in  N^1(T^*E^{[n]}/\A^n)\mid D\cdot C_2 > 0 \} $ denote the half space which is positive on a curve in a generic fiber. 

\begin{prop}
The map $U \to N^1(T^*E^{[n]}/\A^n)$ factors through a surjection 
$U \to N^1(T^*E^{[n]}/\A^n)_{> 0}$. 
\end{prop}

\begin{proof}
That $\Stab'((D^b_{Coh}(PP_{E_S}(\Xi))))$ surjects onto $ N^1(T^*E^{[n]}/\A^n)_{> 0}$ via the same map as \eqref{eq:Utorel_amp} follows from the explicit form of the numerical class of the line bundle given in \eqref{eq:bm_exactly} and that pairings with the curves in \eqref{eq:n1_basis_real2} can attain any value. It remains to show that all of these elements give birational moduli spaces with  fibrations. Letting $s_1, \ldots, s_n$ denote the coordinates of the map $T^*E^{[n]} \to \A^n$ we know that these give rational functions on $M_{\underline{\sigma}_0}(v)_{\mc F}$. Their poles may be calculated in the analytic category where we know that $T^*E^{[n]}$ is an open subset of $X^{[n]}$ where $X$ is an isotrivial K3 surface and the corresponding birational moduli spaces are such that the coordinates $s_i$ extend to regular functions. 
\end{proof}

\subsubsection{Horosphere in positive cone}\label{ssec:horosphere} 

For $n > 1$ the space $ N^1(T^*E^{[n]}/\A^n)_{> 0}$ is a half-space in a 2d real vector space $ N^1(T^*E^{[n]}/\A^n)$ the walls correspond to $C^\perp$ for 
\[  C \in \{ m C_E + n C_{\pt} \mid 1\le 1 \le n \} \] 
where $C_{\pt}$ is the primitive curve contracted by the Hilbert-Chow map and $C_E$ is a curve in a generic fiber. 

\begin{figure}[h!]
  \centering
  \includegraphics[width=0.6\textwidth]{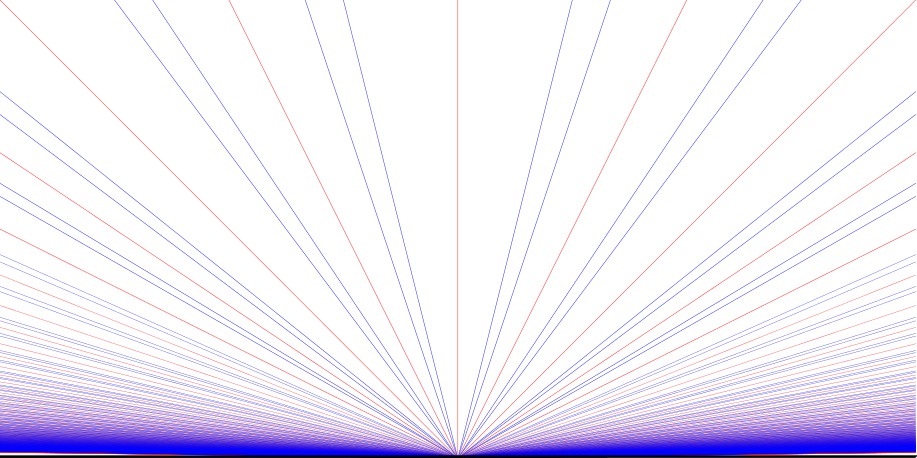}
  \caption{Wall and chamber structure on $ N^1(T^*E^{[n]}/\A^n)_{> 0}$ for $n = 4$. }
  \label{fig:walls_a04}
\end{figure}

Let $S$ denote an elliptic K3 surface with $\NS(S) = \Z\langle E, \sigma \rangle $ where $\sigma$ is a section so the intersection pairing in 
\[ \begin{pmatrix} 0 & 1 \\ 1 & -2 \end{pmatrix} \] 
in this basis. 

Let $n > 1$ and consider
\[ \Pos_1 = \{ H \in  \Pos(S^{[n]}) \mid H^2 = 1\} \] 
denote the level 1 (hyperbolic) hyperplane in the positive cone $\Pos(S^{[n]}) \subset N^1(S^{[n]}) \simeq N^1(S) \oplus \R \delta$. 

Let $\mathbb{H}^k$ denote hyperbolic space with compactification $\overline{\mathbb{H}^k}$ by adjoining the boundary of the Poincar\'e disk model. Given $L \in \overline{\mathbb{H}^k}$ let $Hor_L$ denote the horosphere at $L$ which is naturally endowed with a $k-1$-dimensional Euclidean metric. 

The induced wall-and-chamber structure on the affine hyperplane 
\[ H_1 = \{ D \in N^1(T^*E^{[n]}/\A^n)_{> 0} \mid D\cdot C_E = 1 \}  \]
is equivalent to the wall-and-chamber structure on a horosphere $Hor_L$ where $L \in  \overline{Pos}_1$ corresponds to the isotropic vector in $\overline{\on{Pos}(S^{[n]})}$ corresponding to the Lagrangian fibration $S^{[n]} \to \PP^n$. 
This is shown in Figure \ref{fig:horosphere} in an example. 

The positive cone for $S^{[n]}$ contains many more walls than those compatible with the Lagrangian fibration, corresponding to singular birational models which do not admit compatible maps to $\PP^n$. 

\begin{figure}[h!]
  \centering
  \includegraphics[width=0.9\textwidth]{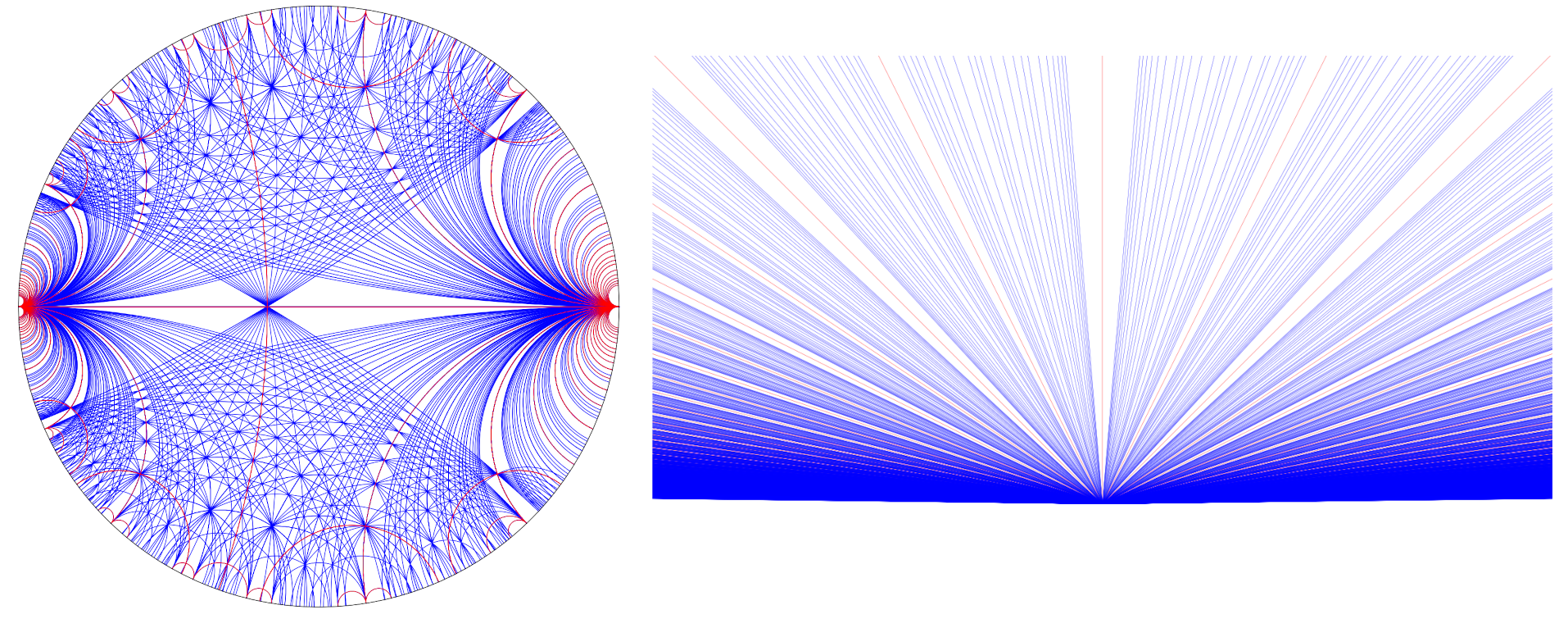}
  \caption{Wall and chamber structure on $ N^1(T^*E^{[n]}/\A^n)_{> 0}$ for $n = 12$ and analogous wall-and-chamber structure on horosphere on the positive cone detailed in \cite{DeHority_2020}.  }
  \label{fig:horosphere}
\end{figure}

\subsubsection{Analysis of stability}
Assume $v$ is the image under derived equivalence of $v' = (1,c_1, d)$ and we are in the situation of \ref{ssec:stability_lfib}. 
Let $\underline{\sigma} \in U$ be generic. Let $\mc F\in \Coh(D_\infty)$ denote a sheaf with the right chern character to serve as $\mc E|_{D_\infty}$ for $E \in M_{\underline\sigma}(v)$ such that $\det(\mc F) = \chi$. 

Let $\mf P$ denote the generic fiber of the $\PP^1$ base of $E\times \PP^1$. 
\begin{prop}\label{prop:generically_the_framing}
Given $\mc E \in M^\chi_{\underline\sigma_0}(v)$ there is an isomorphism
\[ \mc E|_{E\times \mf P} \simeq \mc F \boxtimes \mc O_{\mf P}. \]
\end{prop}
\begin{proof}
  This is true in the Gieseker chamber where a destabilizing sequence on the generic point would destabilize any fiber including the fiber over $\infty$. For the remaining stability chambers, it follows from the fact that the birational modifications obtained by crossing the walls from Proposition \ref{prop:all_walls_are_GU} do not modify the sheaf at the generic fiber. 
\end{proof}

Now consider a fixed closed $s \in HH_{0, comm}(E)$ corresponding to a commutative surface $X$ isomorphic to $\overline{E^\natural}$. Let $\omega = N E + P$ for some $N \gg 0$ be a polarization. 

\begin{lem}\label{lem:all_framable_are_stable}
Any $\mc F$-framable torsion free sheaf $\mc E \in \Coh(X)$ is $\omega$-stable for large enough $N$. 
\end{lem}
\begin{proof}
Given an exact sequence $0 \to \mc E_1 \to \mc E_3 \to \mc E_2 \to 0$  with $\mc E = \mc E_3$. Let $c_1(\mc E_i) = c_{P,i}P + c_{E,i}E$, likewise let $r_i, d_i$ be the rank and degree of $\mc E_i$. Because $\mc E$ is $\mc F$-framable and so locally free in a neighborhood of $D_\infty$ the sequence $\mc E_1 \to \mc E_1^{\vee \vee} \to \mc E_3^{\vee \vee}$ implies that there is a new sequence $0 \to \mc E_1'\to \mc E_3 \to \mc E_2'\to 0$ which is destabilizing if the original one is where $\mc E_1' \to \mc E_3$ is a bundle map in a neighborhood of $D_\infty$.  Because $\mc F$ is stable we have 
\[ \frac{c_{P,1}}{r_1} < \frac{c_{P,3}}{r_3}\] which implies stability of $\mc E$ for large enough $N$ which is the inequality \[\frac{N c_{P,1} + c_{E,1}}{r_3} < \frac{N c_{P,3} + c_{E,3}}{r_3} . \]
\end{proof}

Let $\underline{\sigma}\in U$ be generic.
\begin{prop}\label{prop:deq_all_moduli_are_framed}
There exists $\Phi \in \Aut(D^b_{Coh}(E\times \PP^1))$ such that 
$\mc E \mapsto (\Phi(\mc E), \phi)$ gives an isomorphism 
\[ M^\chi_{\underline{\sigma}_0}(v)_{\mc F} \xrightarrow{\sim} M^{tf}(\Phi_*v; \mc F)\] 
where $\phi : \Phi(\mc E)|_{D_\infty}$ is an isomorphism induced by a trivialization of the gerbe $\mc M^\chi_{\underline{\sigma}_0}(v) \to M_{\underline{\sigma}_0}(v)$. 
\end{prop}

\begin{proof}
Because of the wall and chamber structure on $U$ described by Theorem \ref{thm:vwalls_roots} we know that the chamber $\underline{\sigma}$ lies in is adjacent to a wall where the phase of $v$ overlaps with that of $(0, rE, s)$. Choosing a derived equivalence under which sheaves of class $(0, rE, s)$ are sent to torsion sheaves identifies this wall with the Gieseker-Uhlenbeck wall \cite{Bayer_Macrì_2014_MMP,Tajakka_2023} and stable objects are thus identified with either Gieseker stable sheaves in the Friedman chamber or with derived duals of such. But Lemma \ref{lem:all_framable_are_stable} implies that the stack of $\mc F$-framable Gieseker stable sheaves is isomorphic to the stack of all $\mc F$-framable torsion-free sheaves and the trivialization of the gerbe is equivalent to a choice of framing isomorphism. 
\end{proof}

\subsubsection{Quasiprojectivity }

We now finalize the main quasiprojectivity statement for moduli spaces of framed Bridgeland stable objects for stable framing sheaf $\mc F$. Each of the closed fibers of the family in Theorem \ref{thm:main_quasiprojectivity} is either 1) a moduli space of Bridgeland stable objects on a commutative surface under a derived equivalence or 2) a moduli space of $\mc F$-framed torsion-free sheaves over a generic $s \in HH_0(E)$ in which case a derived equivalence $\Phi \in \Aut(D^b_{Coh}(E))$ sending $\mc F$ to $ \Phi(\mc F) = \mc O_{E}$ induces $\Phi_S: D^b_{Coh}(\PP_{E}(\Xi_s)) \xrightarrow{\sim} D^b_{Coh}(\PP_{E}(\Xi_{\Phi_* s}))$ identifies $M_{\underline{\sigma}_s}(v; \mc F)$ with an open subset of the Hilbert scheme of points on $\PP_{E}(\Xi_{\Phi_* s})$ which is known to be projective \cite{rains19birational_arxiv}. Therefore the main interest is the families statement, that the moduli space is actually quasiprojective over $S$. 

This should be compared to the situation for deformations of smooth moduli spaces of sheaves on a K3 surface, where generic deformations are not projective and the twistor lines are not even K\"ahler.

\begin{thm}\label{thm:main_quasiprojectivity} 
For any $\underline{\sigma} \in U$ the family 
\[ M_{\underline{\sigma}}^\chi(v)_{\mc F}/S \] is quasiprojective over $S$ and semiprojective. 
\end{thm}
\begin{proof}
  We will do this via an analysis of the nef line bundle $L_{\underline{\sigma}}$ on $M_{\underline{\sigma}}^\chi(v)$. We know that $L_{\underline{\sigma}}|_{{M_{\underline{\sigma}_0}}^\chi(v)}$ is ample because by Propositions \ref{prop:deq_all_moduli_are_framed} and \ref{prop:all_walls_are_GU} each stability condition lies in a Gieseker chamber or on a Gieseker-Uhlenbeck wall. That the $J$-torsors over the compactifications of the former are projective is classical and the latter have been shown to be projective as well\cite{Tajakka_2023}.
  
  By the constancy of the Euler characteristic and by semicontinuity we know that the higher cohomology sheaves of $R\pi_{*}L_{\underline \sigma}$ are supported on a closed subvariety $A\subset S$ that is disjoint from $0 \in S$ and via the isomorphisms giving the $\csth$ action on $S$ we know that $A$ is $\csth$-invariant. Thus all fibers of $L_{\underline \sigma}$ are acyclic and since $S \simeq HH_0(E)$ is an affine space we have $R\pi_{*}L_{\underline \sigma}$ is a trivial bundle over $S$ with trivializing sections $s_1, \ldots, s_N$. 

  By the same argument, the locus in $S$ where the sections $s_1, \ldots, s_n$ of $L_{\underline \sigma}$ admit basepoints is also empty. The result follows. 

  To show that $M_{\underline{\sigma}}^\chi(v)_{\mc F}$ is semiprojective we note that $M_{\underline{\sigma}}^\chi(v)$ is proper over $S$. Let $X = \overline{\csth\cdot \mc E}$ be the orbit closure of $\mc E \in M_{\underline{\sigma}}^\chi(v)$. We can assume that $\mc E$ does not lie over $0 \in S$.  Also note that the Koszul data of any $\mc E \in M_{\underline{\sigma}}^\chi(v)$ after an equivalence of the surface corresponding to tensoring with a line bundle (i.e. under the semiorthogonal decomposition \cite[4.16]{rains19birational_arxiv} and the isomorphism with $\PP_{E}(\Xi))$) the Koszul data of $\mc E \in M_{\underline{\sigma}}^\chi(v)$ can be written as 
  \[ \mc E = \kappa( Q[-1], \on{Cone}(\mc F \to Q)[-1], \eta)\]
  where the group action scales the extension to zero in 
  \[ \mc Q \to \mc Q\otimes \Xi_{\hbar  s} \to \mc Q\otimes \mc T_{E} \to \] 
  and thus preserves the composition $\mc Q \to \mc Q\otimes \Xi_{\hbar  s} \xrightarrow{\eta} \on{Cone}((\mc F \to Q)[-1])$. 

\end{proof}

\section{Deformed contractions} 

Let $\underline{\sigma}^w \in U$ denote a point on a wall in $U$ and let $ \underline{\sigma}$ be a point in a neighboring chamber. Thus for any $s \in S$ we have $\underline{\sigma}_s$-semistable implies $\underline{\sigma}^w_s$ semistable. 

This gives map of $S$-varieties
\begin{equation}\label{eq:simultaneous_resolution}
   \pi_w : M^\chi_{\underline{\sigma}}(v)_{\mc F} \to M^\chi_{\underline{\sigma}^w}(v)_{\mc F}
\end{equation}
given by the composition 
\[ M^\chi_{\underline{\sigma}}(v)_{\mc F} \to \mc M^\chi_{\underline{\sigma}^w}(v)_{\mc F} \to M^\chi_{\underline{\sigma}^w}(v)_{\mc F}.\]

\subsubsection{Regular locus}

If $\sigma$ lies in the Gieseker chamber and $w$ is the wall corresponding to the root $\alpha = (0,1)$ under Theorem \ref{thm:vwalls_roots} then on the central fiber this map agrees by construction of the Uhlenbeck moduli space using Bridgeland moduli spaces \cite{Bayer_Macrì_2014_MMP,Tajakka_2023} so that there is a bijection between points of $M^\chi_{\underline{\sigma}_w}(v)_{\mc F}$ and those of an open locus of the Uhlenbeck compactification of the moduli space of instantons on $E\times \PP^1$ with fixed determinant. 

More generally we have the following: if $w_\beta$ is the wall corresponding to $\beta \in R_{A_{-1}}^E$ then denote $\pi_{w_\beta}$ by $\pi_\beta$. We also have a bijection \eqref{eq:roots_HHhyperplanes} between roots in $R_{A_{-1}}^E$ and hyperplanes $S_\beta$ in $HH_0(E)$. These identifications are compatible. 

\begin{prop} 
The map $\pi_\beta$ is birational over $S$ and an isomorphism away from the hyperplane $S_\beta \subset S$. 
\end{prop}
\begin{proof}
  That $\pi_{\beta}$ is regular away from $S_{\alpha}$ follows from the fact that fibers of $\pi_{\beta}$ lie in S-equivalences classes of $\mc F$-framable objects therefore a destabilizing triangle of $\mc E \in M_{\underline{\sigma}}(v)_{\mc F}$ must be of the form $A\to \mc E \to \mc E' \to $ where $\mc E$ and $\mc E'$ are both $\mc F$-framable and $A$ is of class $(0, rE, s)$ corresponding to the root $\beta$. But this implies that $A|_{D_\infty} = 0$ and this is only possible for $s \in S_{\beta}$. 

That $\pi_{\beta}$ is birational over $S$ follows from the fact that over $S_\alpha$ the map is an isomorphism on the regular locus of the Gieseker-Uhlenbeck map. 
\end{proof}

\subsection{Deformation and birational equivalence}

\subsubsection{Cohomology of equivariant semiprojective morphism}

Recall from \cite{Hausel_Sturmfels_2002} that a quasiprojective variety $X$ is called \emph{semiprojective} if it admits a $\C^\ast_\lambda$-action such that $\lim_{\lambda\to 0} \lambda\cdot x$ exists for all $x \in X$. Some examples are smooth Nakajima quiver varieties and moduli spaces of semistable Higgs bundles. 

Let $X$ be semiprojective, let $\C^\ast_\lambda$ act on $\A^1$ with fixed point $0$ with weight $1$ and let $\mc X\to \A^1$ be a quasiprojective $\C^\ast_\lambda$-equivariant morphism with $\mc X_0 \simeq X$. 

It would be nice to know that any two fibers are diffeomorphic by having a semiprojective version of Ehreshmann's lemma; such a result is known in the situation of nonabelian Hodge theory \cite{simpson1990nonabelian}, identifying the family $\mc X$ with a twistor family. However we at least know \cite[Appendix B]{HAUSEL_LETELLIER_RODRIGUEZ-VILLEGAS_2011} that for $c \in \A^1$ a closed point the inclusion $\mc X_c \hookrightarrow \mc X$ induces an isomorphism 
\begin{equation} \label{eq:hausel_cohomology}
H^*(\mc X) \xrightarrow{\sim} H^*(\mc X_c) ~~ (\simeq H^*(X))
\end{equation}
compatible with the mixed Hodge structures on $\mc X$ and all $\mc X_c$ because they are all pure. 

\begin{defn}
Two semiprojective varieties $X,X'$ are said to be $\csth$-equivariantly deformation equivalent if there exits smooth families $\mc X, \mc X'$ over $\A^1$ admitting group actions $\csth$ on $\mc X$ and $\mc X'$ equivariant with respect to the natural action of $\A^1$ and a equivariant isomorphism $\mc X_{\C^\ast} \simeq \mc X'_{\C^\ast}$. 
\end{defn}

In particular, if $X$ and $X'$ are semiprojective and $\csth$-equivariantly deformation equivalent then they have isomorphic cohomology. Let $\Delta_{\C^\ast}$ denote the graph of the isomorphism $\mc X_{\C^\ast} \simeq \mc X'_{\C^\ast}$. As usual we have the following:

\begin{prop}\label{prop:iso_cohomology_semiprojective}
The isomorphism $H^*(X)\simeq X^*(X')$ induced by \eqref{eq:hausel_cohomology} is given by convolution with 
\[[\overline \Delta_{\C^\ast}]\Big|_{X\times X'}. \]
\end{prop}
\begin{proof}
The usual argument \cite{Huybrechts_1999} doesn't use the diffeomorphism between fibers directly only the analogue of \eqref{eq:hausel_cohomology}. 
\end{proof}

\subsubsection{Monodromy reflections}

\subsubsection{}

Let $\mc F \in \Coh(D_\infty)$ be stable framing sheaf with coprime rank $r>0$ and degree $d$. 
\begin{thm}\label{thm:def_equiv}
  All $K$-equivalent smooth projective birational models of $M^{tf}(v; \mc F)$ over $B$ are of the form $M^\chi_{\underline\sigma_0}(v)_{\mc F}$. Any two are $\csth$-equivariantly deformation equivalent. 
\end{thm}
\begin{proof}
Given and $M'$ birational to $M^{tf}(v; \mc F)$ over $B$, let $A \in N^1(M^{tf}(v; \mc F)/B)$ be the image of its ample class. The ample class must pair positively with a curve in the generic fiber of $M^{tf}(v; \mc F)\to B$ and so lie in the half-space $N^1(M^{tf}(v; \mc F)/B)_{\R, >0}$. The image of the map \eqref{eq:Utorel_amp} is exactly this half space. Composing with the birational transformation $M^{tf}(v; \mc F) \dasharrow M^\chi_{\underline\sigma_0}(v)_{\mc F}$ gives a birational transformation sending an ample class to an ample class which is an isomorphism. 

Because any two $M^\chi_{\underline\sigma_0}(v)_{\mc F}$ and $M^\chi_{\underline\sigma'_0}(v)_{\mc F}$ are $\csth$-equivariantly deformation equivalent by choosing a generic line $L\subset HH_0(E)$ the result follows.

\end{proof} 

\section{Other types} 
\subsection{Orbifold cases} 

Given $s \in HH_0([E/\Gamma])$ and  $H,B \in \NS(\msc X_R)_\Q$ be numerically constant divisor classes, where we denote their base change to any $s$ by the same letter.  Let 
\begin{equation}\label{eq:cc}  Z_{H,B,s}(-) : = \int_{\PP^1_{[E/\Gamma]}(\Xi_s)} e^{iH} \on{ch}^B_s(-) \end{equation} 
denote a stability function. Relative stability conditions whose central charge is of this form such that the structure sheaves of all points are stable are referred to as \emph{geometric relative stability conditions}. 

Fix $n > 0$. Let $N^1(X_R^{[n]}/B)_{\R, >0}$ denote the subspace of relative classes pairing positively with a curve in a generic fiber. There is an identification 
\[X_R^{[n]} \simeq M^{tf}(v; \mc O_{D_\infty}) \] 
between the Hilbert scheme and the space of framed torsion-free sheaves of chern character $v = (1,0,-n)$ on $\overline{X_R}$ framed at $D_\infty$ by $\mc O_{D_\infty}$. Given $\underline{\sigma} \in \Stab(\PP_{[E/\Gamma]}(\Xi)/S)$ the same Theorem \ref{thm:cite_moduli_spaces} produces relative moduli spaces $M_{\underline{\sigma}}(v)$ of semistable objects in $D_{\Coh}(\PP_{[E/\Gamma]}(\Xi))$ for chern characters $v$.  For  $\mc F \in \Coh(D_\infty)$ simple,   let $M_{\underline{\sigma}}(v; \mc F)$ denote the subspace of $M_{\underline{\sigma}}(v; \mc F)$ consisting of $\mc F$-framable objects. Simplicity of $\mc F$ implies that the trivialization of the gerbe is equivalent to fixing a framing $\mc E|_{D_\infty} \simeq \mc F$. 

\begin{conj}\label{conj:exist_stability_conditions} 
  There exists a nonempty subspace $U \subset \Stab(\PP_{[E/\Gamma]}(\Xi)/S)$ of geometric relative stability conditions and a surjective map $\ell : U \to N^1(X_R^{[n]}/B)_{\R, >0}$ such that 
  \begin{itemize} 
    \item[1)] For any $\underline{\sigma}, \underline{\sigma}' \in U$ the families $M_{\underline{\sigma}}(v; \mc O_{D_\infty})$ and $M_{\underline{\sigma}}(v; \mc O_{D_\infty})$ are birational over $S$. 
    \item[2)] There is a bijection between the set of walls in $ N^1(X_R^{[n]}/B)_{\R, >0}$ inducing contractions, the set of walls in $HH_0([E/\Gamma])$ where the contraction is not regular, and a subset of roots of $R^{ell}$. 
   \end{itemize} 
\end{conj} 

In particular, assuming Conjecture \ref{conj:exist_stability_conditions} the analogue of Theorem \ref{thm:def_equiv} will hold in this setting. 

\subsection{Wild cases} 

In the other types, $A_0, A_1, A_2$ we expect an analogous relationship but the corresponding categories are instead more general noncommutative surfaces and don't admit descriptions in terms of noncommutative $\PP^1$ bundles over orbifolds. In particular, we expect for $R = A_0, A_1, A_2$ that there are dg-categories $D_R$ over a base $S \simeq R^{ell}\otimes \C$ (where $A_0^{ell} \simeq \Z^2$) such that Conjecture \ref{conj:exist_stability_conditions} holds with $D^b_{Coh}(\PP_{[E/\Gamma]}(\Xi))$ replaced with $D_R$ and the moduli spaces are the symplectic leaves of \cite{rains2019noncommutative}. The base $S$ should be the positive weight subspace of the tangent space to $[\overline{X_R}]$ in the moduli space of noncommutative surfaces. 

\subsection{Birational geometry of central fiber} 

We prove the part of Conjecture \ref{conj:exist_stability_conditions} that doesn't rely on families of stability conditions. 

In particular, there exits geometric stability conditions $\sigma_{H,B} \in \Stab^\dagger(\overline{X_R})$ with central charge given by \eqref{eq:cc} \cite{Arcara_Bertram_Lieblich_2013}. In particular all skyscraper sheaves $\Bbbk(x)$ are $\sigma_{H,B}$-stable of the same phase. 

\subsubsection{Basis for N\'eron-Severi lattice}

The Neron-Severi lattice of the surfaces $\overline{X_R}$ are determined by the fact that they are rational elliptic surfaces as in Section \ref{eq:ns_lattice_other_types}. 

We choose a basis 
\[ \Theta, E, C_\alpha ,~~~\alpha\in \Delta_{R} \sqcup \Delta_{R'} \] 
of $NS(\overline{X_R}) = \NS(\overline{X_{R'}})$ where $\Theta$ is a fixed section of the elliptic fibration structure on $\overline{X_R}$ and we have identified simple roots of $R$ and $R'$ with curves avoiding $\Theta$ in $E_0$ and $E_\infty$ respectively. Let $D_\alpha$ denote the class corresponding to the fundamental coweight of $C_\alpha$ in the root system $R+R'$ so that $\langle C_\alpha, D_\beta\rangle = \delta_{\alpha, \beta}$ and $\langle D_\alpha, E\rangle = \langle D_\alpha, \Theta\rangle = 0$. 

There is an induced basis 
\begin{equation}\label{eq:basis_n1} D_{\pt}, ~\lambda_{{\mc U}}(\Theta)|_{X_R}, ~\lambda_{\mc U}(D_\alpha)|_{X_R}, ~~ \alpha \in \Delta_R \end{equation} 
for $N^1(X_R^{[n]}/\A^{n})$ where $2D_{\pt} $ is the class of the exceptional divisor of the Hilbert-Chow map.

Given a chern character $\on{ch} = (1, 0, -n)$ corresponding to a Hilbert scheme, let $H = D_H + \Theta + bE$ and $B = D_B + c\Theta + dE$ with $D_H$ and $D_B$ in the span of the $D_\alpha$. Consider the subspace 
\begin{multline*} 
   \Stab'(\overline{X_R}))_v  =
   \{ g\cdot \sigma_{H,B} \mid H^2 \gg 0, g\in \widetilde{\GL(2, \R)},\\
    \phi((0, E, 0)) > \phi(v), \phi((0, C_{\alpha}, k)) > \phi(v) ~~~
    \text{ for } \alpha \in R' \text{ and } |k|< N \} 
\end{multline*} 
where $N \in \Z$ is some large enough number depending on $v$ and $\phi(-) = \on{Im}(Z_{H,B}(-))/ \on{Re}(Z_{H,B}(-))$ is the phase function. We will omit $v$ when it can be inferred. 

\subsubsection{Wall and chamber structure} 

A $v$-wall in $\Stab'(\overline{X_R})$ is a wall inducing a birational contraction of $M_{\underline{\sigma}}(v; \mc O_{D_\infty})$. 

There is an injection 
\begin{equation} 
R^{ell} \to K_{num}(\overline{X_R})
\label{eq:roots_to_k_class}
\end{equation}
given by 
\[ \alpha + m \delta_{\pt} + n \delta_E \mapsto (0, C_\alpha + n [E] , m) \] 
where $C_\alpha$ is the root curve corresponding to $\alpha \in R_{fin}$. 

\begin{thm}\label{thm:vwalls_roots}
  The $v$-walls in $\Stab'(\overline{X_R})$ 
  are in bijection with roots 
  $\beta \in R^{ell}$  under the inclusion \eqref{eq:roots_to_k_class} such that
    $\langle \beta, v\rangle \le \langle v, v\rangle/2$. 
  \end{thm}

  \begin{proof}
By the conditions on the phases of objects for stability conditions in 
$\Stab'(\overline{X_R})_v$ 
a the walls are a subset of those where the phase of $v$ overlaps with an objects 
$a \in D^b_{Coh}(\overline{X_R})$
 of rank $0$ with 
 $c_1(a) \in \Z\langle E, C_{\alpha}\rangle_{\alpha\in R_{fin}}$ 
 and so are a subset of 
\[ \mc W  = \{ w_{\beta,v} \mid \beta \in R^{ell}\} \] 
where we use the inclusion \eqref{eq:roots_to_k_class} and $w_{v_1, v_2}$ is the subspace of $\Stab'(\overline{X_R})$ where $\phi(v_1) = \phi(v_2)$. The restriction on $\langle \beta, v\rangle$ is provided by \cite[\S 10]{Bayer_Macrì_2014_MMP}. 
\end{proof}

\begin{figure}[h!]
  \centering
    \includegraphics[width=.8\textwidth]{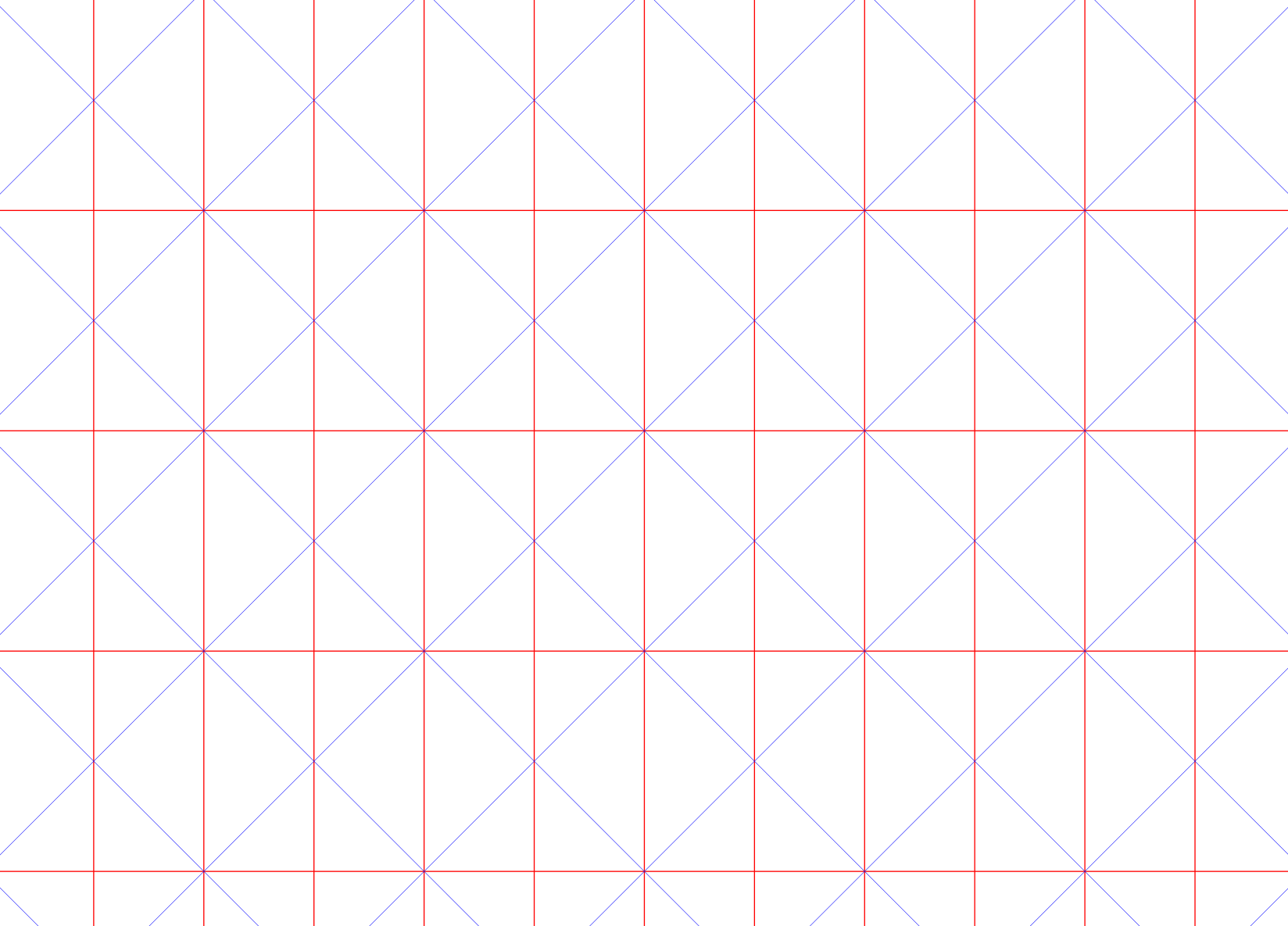}
    \caption{
    \label{fig:a12} Wall and chamber structure on level 1 hyperplane in $ N^1(X_R^{[n]}/\A^n)_{> 0}$ for $  n = 2$ points for surface $X_R$ with  $R = A_1$.}
 \end{figure} 
\begin{figure}[h!]
    \includegraphics[width=1\textwidth]{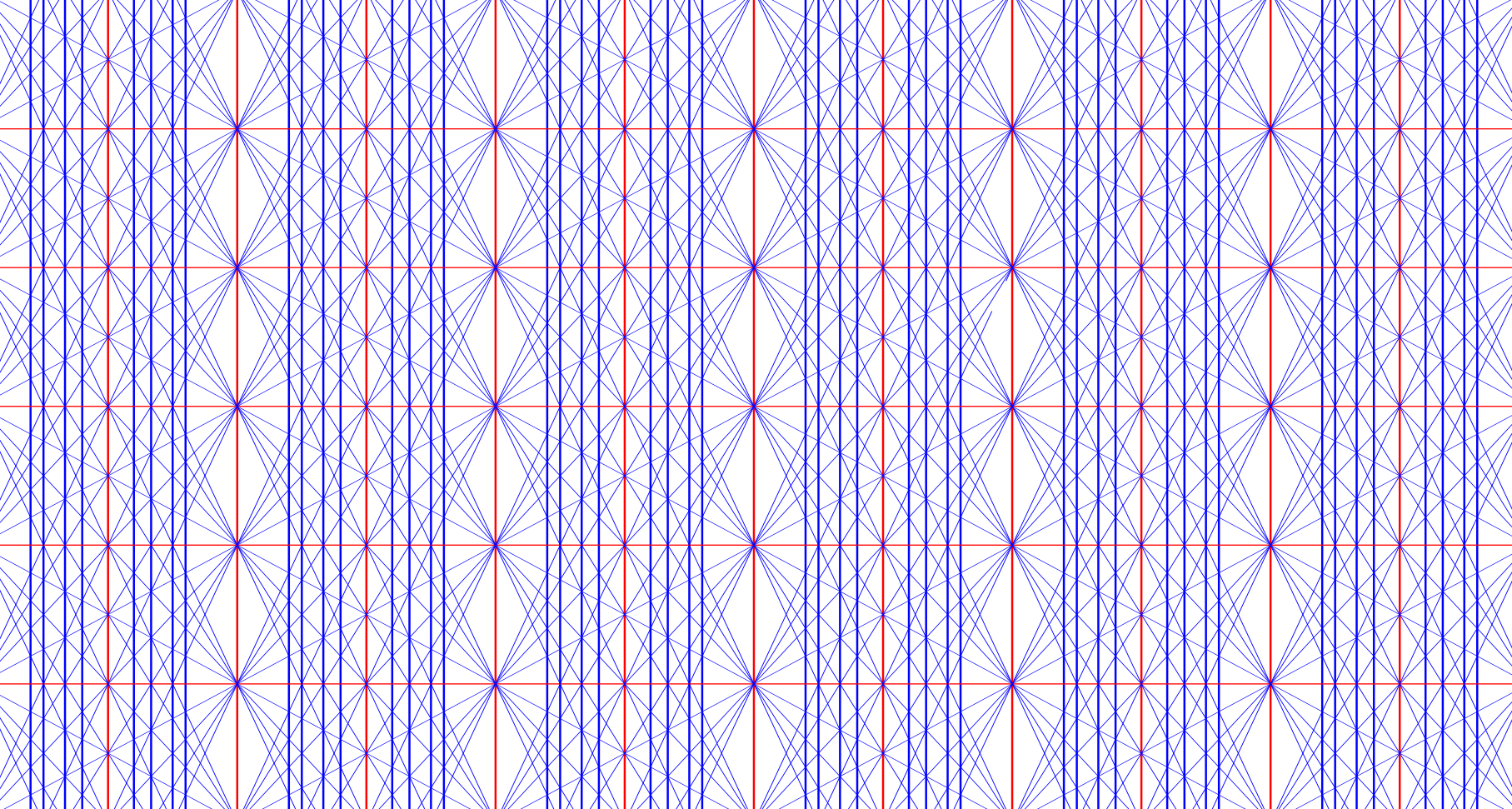}
  \caption{Wall and chamber structure on level 1 hyperplane in $ N^1(X_R^{[n]}/\A^n)_{> 0}$ for $  n = 5$ points for surface $X_R$ with  $R = A_1$.\label{fig:a15}  }
\end{figure}

Letting $N^1(\overline{X_R})' \subset N^1(\overline{X_R})'$ denote the image of $\Stab'(\overline{X_R})$ under the Bayer-Macr\`i map, the composition 
\[ \Stab'(\overline{X_R}) \to N^1(\overline{X_R})' \to N^1(X_R^{[n]}/\A^{n})\] 
factors through a surjection onto 
\[ N^1(X_R^{[n]}/\A^{n})_{>0} := \{ D \in N^1(X_R^{[n]}/\A^{n})_\R \mid \langle D, C_E \rangle > 0 \] 

we obtain by the same argument as Theorem \ref{thm:def_equiv} the following analogue of the main results of \cite{Bayer_Macrì_2014_MMP}. 

\begin{thm}\label{thm:other_types_birational}
All $K$-equivalent smooth projective birational models of $X_R^{[n]}$ over $\A^n$ are of the form $M_{\sigma}(v; \mc O_{D_\infty})$ for some $\sigma \in \Stab'(\overline{X_R})$. 
\end{thm} 

Note that the analysis here is disjoint from that in \cite{Kopper_2021} because the rational elliptic surfaces $\overline{X_R}$ are non-generic and admit $(-2)$-curves, and we are only concerned with the structure relative to the base $\A^n$.

\begin{rmk}
  There exist isotrivial K3 surfaces with fibers of the corresponding Kodaira types \cite{Sawon_2014}. The wall and chamber structures on the level 1 hyperplane in $ N^1(X_R^{[n]}/\A^{n})_{>0}$ exactly coincides with that on an affine subspace of a horocycle on $N^1(S^{[n]})$ for a choice of K3 surface $S$ by analysis of Section \ref{ssec:horosphere}. 
  \end{rmk} 
\section{Correspondence operators} 

Analogues of correspondences defining Nakajima operators are constructed by specializing Nakajima operators under the relative derived equivalences provided by Theorem \ref{thm:derived_reflection}. The result will be conjectural outside type $A_{-1}$. 

Let $\alpha$ be an imaginary root of $\mf g_R$ which is not a multiple of $\delta_E$ and let $B = \{b \xi \mid b \in \C \} \subset S_\alpha \subset S$ be a generic non-generic line, $B^\circ = B\backslash \{ 0 \}$. There exists a correspondence 
\[ Z_\alpha^\circ  \subset M^{tf}(v; \mc O_{D_\infty}) \times_{B^\circ} M^{tf}(v+\alpha; \mc O_{D_\infty}) \times_{B^\circ} X_{B^\circ}\]
where we denote $M^{tf}(v; \mc O_{D_\infty})$ and its basechange to $B^\circ$ by the same letter. Here $X_{B^\circ}$ is the family of commutative non-stacky surfaces under Conjecture \ref{conj:commutative_hyperplane} when $R\neq A_{-1}$ or under Theorem \ref{thm:derived_reflection}. 

This $Z_\alpha^\circ$ is exactly the correspondence defining Nakajima-Baranovsky operators under the derived equivalence. We then use the specialization construction \cite[\S 2.6.30]{Chriss_Ginzburg_2010} to produce cycles $Z_\alpha \in H_{T}(M^{tf}(v; \mc O_{D_\infty}) \times M^{tf}(v+\alpha; \mc O_{D_\infty}) \times \overline{X_R})$. Denote by the same letter its restriction to $X_R$ defined by 
\[ Z_\alpha = \lim_{b \to 0 } Z_\alpha^\circ \] 
where $b$ is the coordinate on $B$. 

Given $\gamma \in H^*_{T}(X_R) \simeq \mf g_{R, \alpha}$ we produce operators 
\[ Z_{\alpha, \gamma} = \pi_{12*}(Z_\alpha \cap \pi_3^*\gamma)\in H^*_{T}(M^{tf}(v; \mc O_{D_\infty}) \times M^{tf}(v + \alpha; \mc O_{D_\infty})\] 
as usual. 

If $\alpha = \beta + a\delta_E + b\delta_{\pt}$ is a real root and $B^\circ$ is a generic line in the hyperplane $S_\alpha$ we let 
\[ Z_\alpha^\circ \subset M^{tf}( v ;\mc O_{D_\infty})\times_{B^\circ} M^{tf}( v+\alpha ;\mc O_{D_\infty})\] 
denote the component of the Steinberg correspondence corresponding to the generator $e$ of the $\mf{sl}_2$ action on 
\[ \bigoplus_{k}  H^*(M^{tf}( v + k \alpha;\mc O_{D_\infty}))\] 
induced by the compatible system of stratified Mukai flopping contractions $\pi_{w_\alpha}$ associated to the wall $w_\alpha$. 

Likewise define $Z_\alpha = \lim_{b\to 0}Z_\alpha^\circ$. Let $\alpha$ be a real or primitive imaginary root of $\mf g_{R}$ and let $\mf g_{\Z \alpha}$ denote the subalgebra spanned by multiples of this root, so $\mf g_{\Z \alpha} \simeq \sl_{2}$ or $\mf g_{\Z, \alpha} \simeq \Heis_{H^*_T(X_R)}$. 

\begin{prop}
The cycles $Z_{k \alpha ,\gamma}$, $k \in \Z\backslash \{ 0\}$   (or $Z_{k \alpha}$  $k = \pm 1$ for real roots) generate a representation of $\mf g_{\Z, \alpha}$ under convolution on 
\[ \bigoplus_k M^{tf}(v + k[\alpha]; \mc O_{D_\infty})\] 
assuming Conjecture \ref{conj:commutative_hyperplane} outside of type $A_{-1}$ in the imaginary case. 
\end{prop}
\begin{proof}
This follows from the convolution commutes with specialization theorem \cite[Prop. 2.7.33]{Chriss_Ginzburg_2010} which says that 
\[\lim_{b \to 0}  Z_1\circ \lim_{b\to 0} Z_2    = \lim_{b \to 0}\left( Z_1 \circ_{B^\circ} Z_2 \right)\]
where $\circ_{B^\circ}$ denotes relative convolution over $B^\circ$. We first do the isotropic case.  It follows that the cycles define representations of the required algebras because $M^{tf}(v; \mc O_{D_\infty})_{B^0}$ is isomorphic to $M^{tf}_{X_b}(\Phi_*v, \Phi(\mc O_{D_\infty}))\times B^\circ$ where $b \in B$ is any non-zero point,  for a relative derived (anti)-equivalence $\Phi$ sending the class $[\alpha]$ to $[\pt]$ and under this isomorphism the cycles $Z_{k\alpha}^\circ$ and $Z_k \times B^\circ$ where the $Z_k$ are Nakajima-Baranovsky correspondences satisfying the relations of $\Heis_{H^*(X_R)}$. The proof in the real root case is identical. 
\end{proof}

\begin{conj}\label{conj:main_repn}
The assignment 
\begin{align*}
   w^{a,b}_\gamma &\mapsto \bigsqcup_{v} Z_{a \delta_{E} + b\delta_{pt}, \gamma}\\
   x^{a,b}_\alpha & \mapsto  \bigsqcup_{v} Z_{\alpha + a \delta_{E} + b\delta_{pt}} , \alpha \in R\backslash \{ 0 \}
\end{align*}
defines a representation of $\mf g_{R}\otimes \C[\hbar]$ on $\bigoplus_v H^*_{T}(M^{tf}(v; \mc O_{D_\infty}))$ after a change of basis of the imaginary root spaces. 
\end{conj}

We will prove in Theorem \ref{thm:main_repn} below this theorem for $R = A_{-1}$.

\subsection{Relation with Nakajima operators}

Fix an identification $Pic^0(T^*E) \simeq Jac(E)  = E$ and refer to all of these spaces as $J$. We identify here the operators giving the $w^{a,b}_\gamma$ generators of $\mf g_{A_{-1}}$.

Let $\mc K\in \Coh(S)_{\alpha}$ denote a torsion sheaf of class $\alpha$ which is generated by the image of a map $\sigma: \mc E \twoheadrightarrow \mc K \to 0 $ under the action of $\mc O_S$ where $\mc E\in M^{tf}(v; \mc O_\infty)$ is an $\mc O_\infty$-framed torsion-free sheaf. 

\begin{lem}
  The Jordan-H\"older factors of $\mc K$ must all have rank $\le 1$. 
\end{lem}

\begin{proof}
  We must have that $\mc K$ admits a surjective map from $\iota_* \mc O_{E_Y}$ where $Y  \subset \PP^1$ is a subscheme and $E_Y$ the fiber over $Y$. But this in particular implies the restriction on Jordan-H\"older factors. 
\end{proof} 

Thus we can write the Harder-Narasimhan type of torsion sheaf avoiding $D_\infty $ which arises as the quotient of an ideal sheaf in the form 
\[ \underline{h} := (r_1, r2d_1; r_2, r_2d_2; \ldots; r_a, r_ad_a; 0, k_1; \ldots; 0, k_b)\] 
denote as a sequence of chern classes of Harder-Narasimhan factors considered as elements of $K_{num}(E)$ so that $\alpha = \sum_i (r_i, r_i d_i) + \sum_j (0, k_j)$. Given $a$ whose Harder-Narasimhan filtration has this form write 
\[ \on{HN}(a) = \underline{h}. \] 

Let $n$ be the imprimitivity of $\alpha$ so that $\alpha= n\alpha_0$ with $\alpha$ primitive. Suppose $\alpha \neq (0, k)$. Let $\mf Z_{\underline h} \in  H_{mid, T}(M^{tf}(v; \mc O_\infty) \times M^{tf}(v + \alpha; \mc O_\infty)\times J \times \A^1)$ denote the cycle defined by 
\begin{equation}\label{eq:corr_Zh}
\begin{split} 
  \mf Z_{\underline h} &= n \cdot [Z_{\underline h}]\\ 
 Z_{\underline h} &:= \on{closure}(\{(\mc E_1, \mc E_2, \chi, a) \mid \\
 \mc E_1 \subset \mc E_2, &\on{ssup}(\mc E_2/\mc E_1) = E_a, \det(\mc E_2/\mc E_1) = \chi, \on{HN}(\mc E_2/\mc E_1) = \underline{h}  \}) \end{split} 
\end{equation}
which is an analogue of the correspondence defining Nakajima operators. When $\alpha = (0,n)$ we define $\mf Z_{(0,n)} $ the same way but require $\on{ssup}(\mc E_2/\mc E_1)$ to be a single point in the definition of $Z_{(0,n)}$.  

\begin{prop}\label{prop:cycle_ZH_irred_lagrangian}
  The cycle $ Z_{\underline h}$ is an irreducible Lagrangian cycle in $M^{tf}(v; \mc O_\infty) \times M^{tf}(v + \alpha; \mc O_\infty) \times T^*J$ where the last two factors have reversed symplectic form. 
\end{prop}
\begin{proof}

\end{proof}  

We now define correspondences $w_{\gamma, \underline{h}}^{a,b} \in H_T(M^{tf}(v; \mc O_\infty) \times M^{tf}(v + (0,aE,b); \mc O_\infty))$ according to the formula
\begin{equation}
w_{\gamma, \underline h}^{a,b} = \pi_{12*}(\pi_3^* \gamma \cap \mf Z_{\underline{h}}).
\end{equation}

When $\underline{h} = (0,n)$ we omit $\underline{h}$ from the notation. 
We now focus on this case $\alpha = (0, n)$ and relate these correspondences to to Nakajima operators. Identifying $S = T^*E$ with the small diagonal $S^\Delta \subset M_{0, n} \simeq \Sym^n(T^*S)$ and recalling the chosen identification $Jac(E) \simeq E$, the map $S \to T^*J$ is identified with the map $[n]\times \Id : T^*E \to T^*J$ where $[n]$ is the multiplication by $n$ map on $E$. We will frequently write $[n]$ in place of $[n]\times \Id$. In the basis $\{ E, \sigma_a, \sigma_b, \rm{pt}\}$ for $H^*(E) \simeq H^*(T^*E)$, we have $[n]^*E = E, [n]^*\sigma_{a, b} = n \sigma_{a,b},$ and $[n]^*\rm{pt} = n^2 \rm{pt}$. When $\underline h = (0, n)$ there is an equality 
\[ [n]_3^* \mf Z_{\underline h} = n \mf Z_n \subset M^{tf}(v; \mc O_\infty) \times M^{tf}(v + (0,0,n); \mc O_\infty) \times E \times \A^1\]
where $\mf Z_n$ is the correspondence defining Nakajima operators. 

The cycle defining $w_{\gamma}^{0,n}$ is 
\begin{align*} 
\pi_{12*}(\pi_3^* \gamma\cap \mf Z_{(0,n)}) &=  \frac {n}{n^2}\pi_{12*}(\pi_3^* \gamma\cap [|n|]_{3*}\mf Z_{n})\\
&= \frac {1}{n}\pi_{12*}[|n|]_{3*}([|n|]_3^*\pi_3^*\gamma \cap \mf Z_n)\\
&= \frac {1}{n}\pi_{12*}(\pi_3^*[|n|]^*\gamma \cap \mf Z_n).
\end{align*}
It follows that there are equalities 
\begin{align}\label{eq:w_to_nakajima}
  w_{\gamma}^{0,n} &= \frac {1}{n}\alpha_n([|n|]^*\gamma)\\
  [w_{\gamma}^{0, n}, w_{\gamma'}^{0, -n}] &= \frac{n}{n^2} \langle [n]^*\gamma, [n]^* \gamma'\rangle = n \langle \gamma,\gamma'\rangle. 
\end{align}. 

In particular $w_{E}^{0,n} =  \alpha_n(E)/n$ and $w_{\mrm{pt}}^{0,n}= n \alpha_n(\mrm{pt})$ which is a normalization which is used in \cite{dehority_toroidal_vosa}. 

We now relate the specialization cycles $Z_\alpha$ for $\alpha \in R^{ell}$ for $R = A_{-1}$  to the cycles $\mf Z_{\underline{h}}$. 

\begin{prop}\label{prop:cycle_decomposition} Let $\alpha \neq (k,0)$.  There exist numbers $a_{\underline h} \in \Q$ so that
\[ Z_\alpha = \sum_{\underline h} a_{\underline h} Z_{\underline h}. \]
\end{prop}

\begin{proof}
For deformations of holomorphic symplectic varieties the specialization of a Lagrangian correspondence $Z^\circ$ is supported on its closure $\overline{Z^{\circ}}$ which is a Lagrangian cycle over the central fiber. We know that the closure of $Z^{\circ}$ is contained in 
\[ 
  Z_{\det} := \overline{\{(\mc E_1, \mc E_2, \chi, a) \mid 
  \mc E_1 \subset \mc E_2, \on{ssup}(\mc E_2/\mc E_1) = E_a, \det(\mc E_2/\mc E_1) = \chi \}}\] 
on account of the fact that the quotient $\mc E_2/\mc E_1 $ having connected support is a closed condition and the map 
\[ M^{tf}(v; \mc O_{D_\infty})\times_{B^\circ} \times M^{tf}(v + \alpha; \mc O_{D_\infty}) \times_{B^\circ} X_{B^\circ} \xrightarrow{\det^{\times 3}} J\times J \times J \] 
extends to a map 
\[ M^{tf}(v; \mc O_{D_\infty})\times_{B} \times M^{tf}(v + \alpha; \mc O_{D_\infty}) \times X_B \xrightarrow{\det^{\times 3}} J\times J \times J \] 
and $Z_\alpha^\circ$ is contained in the inverse image of the locus 
\[ \{ (\chi_1, \chi_2, \chi_3) \in \Jac^{\times 3} \mid \chi_2 = \chi_1\otimes \chi_3\}. \] 

The result follows from Proposition \ref{prop:cycle_ZH_irred_lagrangian} the irreducible components of $Z_{\det}$ are $Z_{\underline{h}}$ and so $Z_\alpha$ is a linear combination of the $Z_{\underline{h}}$ but by dimension reasons in equivariant cohomology we must have that $a_{\underline{h}} \in \Q$. 

\end{proof} 

The crucial commutation relation is the one between $Z_\alpha$ and the Nakajima correspondence $\mf{Z}_{(0,n)}$. Let $m : J \times J\to J$ denote the multiplication map and let $M = \Gamma_m \boxtimes \Gamma_{\Delta}^\tau$ denote the graph of the relative multiplication in 
\[ T^*J\times T^*J\times T^*J \] so 
\[ M = \{ (\chi_1, a_1, \chi_2, a_2, \chi_3, a_3) \mid \chi_1 \otimes \chi_2 = \chi_3 , a_1 = a_2 = a_3 \} \] 
as a set. 

\begin{prop}\label{prop:one_summand_part}
For any $\underline{h}$ we have 
\[ [\mf Z_{(0,n)}, Z_{\underline{h}}] = \sum_{\underline{h'}} c_{n, \underline{h}}^{\underline h'} [\Delta \times M]\circ  Z_{\underline{h'}}\] 
for some $c^{\underline h'} _{n, \underline{h}}\in \Q$.
\end{prop}
\begin{proof}
 We separate into cases $n > 0 $ and $n < 0$. First assume $n < 0$. 

We interpret the left hand side as 
\[ \mf Z_{(0, n)} \circ Z_{\underline{h}} - Z_{\underline{h}} \circ \mf Z_{(0,n)} \subset M(v: \mc O_{D_\infty}) \times M( v + \alpha - n [\pt]; \mc O_{D_\infty})\times T^*J \times T^*J .\]

We will drop the framing sheaf from the notation for simplicity. 

Let 
\[ Z_{n, \underline{h}} \subset M(v)\times M(v -n[\pt]) \times M(v + \alpha - n[\pt]) \times T^*J\times T^*J \] 
denote 
\[ \begin{split}Z_{n, \underline{h}} = \{ (\mc E_1, \mc E_2, \mc E_3,\chi_1,a_1, \chi_2, a_2 ) \mid [\mc E_2/\mc E_1] = n[\pt], \on{HN}(\mc E_2/\mc E_2) = \underline{h}\\
  \supp(\mc E_2/\mc E_1 ) = x \subset E_{a_1}, \det(\Bbbk_x)^n = \chi_1 , \supp(\mc E_3/\mc E_2) = E_{a_2}, \det(\mc E_3/\mc E_2) = \chi_2 \} 
\end{split}\] 
and define $Z_{n, \underline h}$ similarly. Owing to the transverse intersections defining the convolutions giving rise to the commutator, the commutator is 
\[ \pi_{1,3*}(Z_{n, \underline h}) - \pi_{1,3*}(Z_{\underline h, n}).\]
These cycles agree outside the locus where $\mc E_3/\mc E_1$ is a nontrivial extension
\[ 0 \to a \to \mc E_3 / \mc E_1 \to \Bbbk_Y \to 0 \] 
of a torsion sheaf $\Bbbk_Y$ supported at a point $x$ by a sheaf $a$  with HN type $\underline h$ such that $a$ is supported on an infinitesimal neighborhood on the fiber containing $x$. 

Because of the sequence 
\[0\to \mc E_2/\mc E_1  \to \mc E_3/\mc E_1 \to \mc E_3/\mc E_2 \to 0 \] 
we have 
\[ \chi(\mc E_3/\mc E_1) = \chi(\mc E_3/\mc E_2)\otimes \chi(\mc E_2/\mc E_1) = m(\chi_2, \chi_1) \] 
and by the compatibilities of the supports we know that $[\mf Z_{(0,n)}, Z_{\underline{h}}]$ is supported on 
\[ [\Delta_{M(v) \times M(v + \alpha - k[\pt])} \times M] \circ Z_{\det} \] 
and by dimension reasons, the desired commutator must be a $\Q$-linear combination of the irreducible Lagrangian components of the latter, which is the result when $n < 0$. The case $n > 0$ follows by an almost identical argument. 
\end{proof} 

\subsubsection{Elliptic curve class} 

Let 
\[ V_{T^*E} = V^+\begin{pmatrix} 0 & 1 \\ 1& 0 \end{pmatrix} \otimes \mc F_{H^{\underline 1}(E)}\] 
denote the vertex algebra from \cite{dehority_toroidal_vosa} where 
\[ V^+\begin{pmatrix} 0 & 1 \\ 1& 0 \end{pmatrix} \subset V\begin{pmatrix} 0 & 1 \\ 1& 0 \end{pmatrix}\] 
is the sub-vertex algebra $\mc F_{H^{\underline 0}(E)}\otimes \C[e^{\pm E}]$ of the lattice vertex algebra associated to the lattice $\NS(E\times \PP^1)$. 

In particular we have 
\begin{equation}\label{eq:comm_vertex_algebra} Y(e^{mE}, z) = e^{mE}: \exp \left( \sum_{n\neq 0} \frac{\alpha_n(E) z^{-n}}{-n} \right):.
\end{equation}

\begin{prop}\label{prop:mid_dim_is_vertex}
For $\alpha = m\delta_E + n\delta_{\pt}$ with $n \neq 0$ there is an equality of operators 
\[ Z_\alpha \cdot \pi_3^*([E]) = c Y(e^{mE}, z)[z^{-n}] \] 
for $c \in \Q$ non-zero. 
\end{prop} 
\begin{proof}
Let $\mc C_\alpha\subset N^1(T^*E^{[n]}/\A^{n})$ denote a chamber with an adjacent wall labelled by $\alpha$. Let $\Phi_{\mc C, \mc C_\alpha}$ denote the correspondence inducing the isomorphisms of cohomology of Proposition \ref{prop:iso_cohomology_semiprojective}. for a chamber $\mc C$ let $Z_\alpha^{\mc C}$ denote the specialization of the cycle $Z_\alpha^\circ$ corresponding to the deformation of 
\[ M_{\sigma}(v; \mc O_{D_\infty}) \times  M_{\sigma}(v + \alpha; \mc O_{D_\infty}) \times X_{A_{-1}} \] 
along $B \subset HH_0(E)$ where $\ell_\sigma \in \mc C_\alpha$. 

Then we have an equality 
\begin{equation}\label{eq:z_alpha_is_flop} \Phi_{\mc C_\alpha, \mc C} \circ Z_{\alpha}^{\mc C_\alpha} \circ \Phi_{\mc C, \mc C_\alpha} = Z_\alpha.
\end{equation}

The result may be checked in non-equivariant cohomology as an equality of middle-dimensional Lagrangian correspondences proper over the Lagrangian fibrations. By the main result of \cite{DeHority_2020} and \eqref{eq:z_alpha_is_flop} there is an equality of operators 
\begin{equation}\label{eq:pushforward_j} j_*(Z_\alpha \cdot \pi_3^*([E])) = cj_*(Y(e^{mE}, z)[z^{-n}])\end{equation}
where 
\[ j : M(v) \times M(v+ \alpha) \times X_{A_{-1}} \to M^S(v) \times M^S(v+ \alpha) \times S \] 
is an open analytic embedding induced by an open embedding $T^*E \to S$ where $S$ is an elliptic K3 surface but since the flops $\Phi_{\mc C, \mc C'}$ and the extension correspondences are Lagrangian correspondences lying in $M^S(v)\times_{\PP^n} M^S(v + \alpha)$ where $\PP^N$ is the base of the Lagrangian fibration on $M^S(v)$ or $M^S(v + \alpha)$ depending on the sign of $n$, the equation \eqref{eq:pushforward_j} implies the desired result. 
\end{proof} 

\begin{prop}
  \label{prop:heis_commutators} 
  \begin{itemize} 
\item[(a)] For $\alpha = (k,0)$ there exists $a_{\underline{h}} \in \Q$ so that 
the conclusion of Proposition \ref{prop:mid_dim_is_vertex} is satisfied. 
\item[(b)]
There is an equality 
\[ [\mf Z_{(0,n)}, Z_{\alpha}]  = c [\Delta \times M]\circ Z_{\alpha+ (0,n)}\] 
for $c \in \Q$ nonzero. 
  \end{itemize} 
\end{prop} 
\begin{proof}
For the field $Y(e^{mE}, z)$ and an element $\alpha_k(\gamma)$ of the Heisenberg algebra we have 
\begin{equation}\label{eq:VOA_heis_comm} [\alpha_k(\gamma), Y(e^{mE}, z)] = z^k \langle \gamma, mE\rangle Y(e^{mE}, z) \end{equation}
By Proposition \ref{prop:one_summand_part} we get 
\begin{align*}
  [\mf Z_{(0,n)}, Z_{\alpha}]  &= \sum_{\underline h}  a_{\underline h} [\mf Z_{(0,n)}, Z_{\underline h}]\\
  &= \sum_{\underline h, \underline h'} a_{\underline h} c_{n, \underline{h}}^{\underline h'} [\Delta \times M]\circ  Z_{\underline{h'}}.
\end{align*}

We have an equality 
\begin{equation}\label{eq:apply_E}Z_\alpha \cdot \pi_3^*([E]) = \sum a_{\underline h} Z_{\underline h}\cdot  \pi_3^*([E]).
\end{equation} 

In particular since we may obtain the zero coefficient of $Y(e^{mE}, z)$ as a commutator of this form, for $\alpha = (k,0)$ define 
\[ Z_\alpha = \sum_{\underline h} a_{\underline h} Z_{\underline h}\] 
to be the unique linear combination of the cycles $Z_{\underline h}$ so that 
we obtain the conclusion of Proposition \ref{prop:mid_dim_is_vertex}. 

Furthermore, by applying \eqref{eq:apply_E} to cohomology classes $[\A^1] \times [E]\in H^*_T(T^*J\times T^*J)$, \eqref{eq:VOA_heis_comm} implies that  
\begin{align*}
   c\sum_{\underline h'} a_{\underline h'} Z_{\underline h'} \cdot \pi_3^*[E] &= [\mf Z_{(0,n)}, Z_{\alpha}]\cdot(\pi_{34}^*[\A^1]\times [E]) \\
   &= \sum_{\underline h, \underline h'} a_{\underline h}c_{n,\underline h}^{\underline h'}  Z_{\underline h'} \cdot \pi_3^*[E]
\end{align*}
for nonzero $c \in \Q$. Thus $a_{\underline h'} = \sum_{\underline h} a_{\underline h} c_{n, \underline h}^{\underline h'}$ up to a nonzero multiplicative constant as desired. 

\end{proof}

\subsubsection{Proof of $g_{T^*E}$ representation}
Using the preceding analysis we prove Conjecture \ref{conj:main_repn} when $R = A_{-1}$. Given a generic sequence of stability conditions $\sigma = \{ \sigma_v\}$ compatible in the sense that they lie in the same chamber with respect to any root walls, let 
\[V_{T^*E} =  \bigoplus_{m,n} H^*_T(M_{\sigma}(v + m[E] + n[\pt]; \mc O_{D_\infty}))\] 
which is canonically defined because of Theorem \ref{thm:def_equiv}. Let $Z_{m\delta_E + n\delta_\pt,\gamma}$ denote the specialization cycle intersected with $\pi_3^* \gamma$ in a pair of moduli spaces for these stability conditions. 

\begin{thm}\label{thm:main_repn}
The assignment 
\[ w^{m,n}_{\gamma} \mapsto \bigsqcup_v Z_{m\delta_E + n\delta_\pt,\gamma}\] 
defines a representation of $\mf g_{T^*E}\otimes \C[\hbar]$ on $V_{T^*E}$ up to a change of basis in each root space.  
\end{thm} 
\begin{proof}
By Proposition \ref{prop:heis_commutators} and \cite[Thm. 4.65]{dehority_toroidal_vosa} the $N =2$ analytic superfield 
\begin{multline*}
   \mc D_m(z, \psi_+, \psi_-) = \sum_{n \in \Z} Z_{m\delta_E + n \delta_{\pt}, [\pt]}z^{-n} \\
     +Z_{m\delta_E + n \delta_{\pt}, [\sigma_+]}z^{-n}\psi_- + Z_{m\delta_E + n \delta_{\pt}, [\sigma_-]}z^{-n}\psi_+ \\
     + Z_{m\delta_E + n \delta_{\pt}, [E]}z^{-n}\psi_+\psi_- \end{multline*}
has the commutation relations with the Heisenberg subalgebra $\Heis_{H^*(T^*E)}$ to imply that $\mc D_m(z, \psi_+, \psi_-)$ is given by the explicit formula 
\begin{align*}
  \mrm{d}_m(z)
  &=  mz^2:\omega(z) \Gamma_{mE}(z):  + m D_z \Gamma_{mE}(z) - D_z \left[ z: \alpha(\pt,z)  \Gamma_{mE}(z):\right] \\
  \nonumber
  &- m  z^2: \partial_z \alpha(mE, z) \Gamma_{mE}(z):\\
  \mrm{k}_m(z) &= \begin{cases}\frac{1}{m}\Gamma_{mE}(z) & m \neq 0 \\
    \pmb{\phi}(E, z)-E-\alpha_0(E)\log z & m = 0\end{cases}\\
    \sigma^+_m(z) &= z:\alpha(\sigma_+, z)\Gamma_{mE}(z):\\
    \sigma^-_m(z) &= z:\alpha(\sigma_-, z)\Gamma_{mE}(z):.
\end{align*}

in terms of the Heisenberg algebra up to a normalization of the coefficients by a constant (c.f. \cite{dehority_toroidal_vosa} for notation). But these formulae are shown in op. cit. to give rise to a representation of $\mf g_{T^*E}$.
\end{proof}

\subsection{Monodromy action} 

In the sequel we space the generators $w_\gamma^{m,n}$ so that they match exactly the fields in the proof of Theorem \ref{thm:main_repn}. 

Conjecturally for general $R$ and provably in type $A_{-1}$ the derived equivalences of Conjecture \ref{conj:reflection} give rise to reflection group actions on the equivariant cohomology of $M_{\sigma}(v; \mc F)$. Specifically let $\Gamma_v \subset G$ be those operators in $G$ from Conjecture \ref{conj:reflection} which stabilize the class of $v$. The monodromy induced by the deformation equivalence provided by Conjecture \ref{conj:exist_stability_conditions} induces representations 
\[ \Gamma_v \to \End(H^*_T(M_{\sigma}(v; \mc F)))\] 
for simple $\mc F$ and generic $\sigma$ which restrict on divisor classes to the reflections through the divisorial hyperplanes colored red in the examples in Figures \ref{fig:a12}-\ref{fig:a15}. 

In the particular case $R = A_{-1}$ we get a representation 
\[ \rho_{M}: W_{A_{-1}, G}^{ell} \to \End(H^*_T(T^*E^{[n]})). \]

In the non-equivariant setting the monodromy described by $\rho_M$ has implications for holomorphic anomaly equations for Gromov-Witten invariants of Hilbert schemes of points on K3 surfaces \cite{Oberdieck_2022}. 

Consider  $W_{A_{-1}, G}^{ell}\simeq \Z \times \Z/2\Z$ with generators $(s, f)$ of order $\infty, 2$ respectively. Define an action $\rho$ on 
$H^*_T(T^*E^{[n]})$ 
by the formula 
\begin{align}\label{eq:rhos}
  \rho(s)\left(  \prod_i \alpha_{-k_i}(\gamma_i) |\rangle    \right) &=  \prod_i w^{1, -k_i}_{\gamma_i} |\rangle  \\ 
  \label{eq:rhof}
   \rho(f)\left( \prod_i \alpha_{-k_i}(\gamma_i) |\rangle  \right) &=  e^{-nE}\prod_i(-1)^{k_i + 1 } \alpha_{-k_i}(\gamma_i) |\rangle 
\end{align}

\begin{thm}\label{thm:monodromy}
The formula $\rho$ defines the action of the monodromy representation $\rho_M$. 
\end{thm} 
\begin{proof} 
That the monodromy reflection through the Hilbert-Chow hyperplane is given by the symmetric function involution of the form \eqref{eq:rhof} is well known \cite{Okounkov_Pandharipande_2010,Oberdieck_2022}. That the $s$ generator is given by the formula \eqref{eq:rhos} follows from the fact that $\rho_M$ intertwines the actions of the infinite slope and slope $1$ Heisenberg subalgebra of $\mf g_{T^*E}$ and each of these generates the weight space $H^*_T(T^*E^{[n]})$ from one of the vacua $e^{kE}|\rangle$. 
\end{proof} 

More generally we see that the action of an element of $W^{ell}_{A_{-1}, G}\times \SL(2, \Z)$ acts via the formula \eqref{eq:monodromy_intro}, where the monodromy acting on the cohomology labels arises by parallel transport along moduli spaces relative to a family of elliptic curves. 

\printbibliography
\end{document}